%% file: PS_2013.tex
\documentclass[a4paper,10pt]{article}
\usepackage{a4wide}
\usepackage{amssymb,amsmath,amsfonts} \allowdisplaybreaks
\usepackage{graphicx}
\usepackage{ifthen}

\usepackage[linktocpage, pdftex]{hyperref}
    \usepackage{color}
    \definecolor{darkred}{rgb}{0.5,0,0}
    \definecolor{darkgreen}{rgb}{0,0.5,0}
    \definecolor{darkblue}{rgb}{0,0,0.5}
    \hypersetup{colorlinks,linkcolor=darkblue,filecolor=darkgreen,urlcolor=darkred,citecolor=darkblue}
\hypersetup{
pdftitle={A Banach space-valued ergodic theorem for amenable groups and applications},
pdfauthor={Felix Pogorzelski, Fabian Schwarzenberger},
}

\usepackage{color}
\usepackage{amsthm}
\usepackage{enumerate}
\numberwithin{equation}{section}    % Gleichungen mit Sections nummeriert
\theoremstyle{plain}
\newtheorem{Theorem}{Theorem}[section]
\newtheorem{Proposition}[Theorem]{Proposition}
\newtheorem{Corollary}[Theorem]{Corollary}
\newtheorem{Lemma}[Theorem]{Lemma}
\theoremstyle{definition}
\newtheorem{Assumption}[Theorem]{Assumption}
\newtheorem{Definition}[Theorem]{Definition}

\newtheorem{assumption}[Theorem]{Assumption}
\theoremstyle{remark}
\newtheorem{Remark}[Theorem]{Remark}

\renewcommand{\epsilon}{\varepsilon}
\renewcommand{\phi}{\varphi}
\newcommand{\one}{\mathbf{1}} % char. Funktion

\newcommand{\ato}[2]{\genfrac{}{}{0pt}{2}{#1}{#2}}
\providecommand{\abs}[1]{\lvert#1\rvert}
\providecommand{\bigabs}[1]{\bigl\lvert\ifthenelse{\equal{#1}{\cdot}}{{}\cdot{}}{#1}\bigr\rvert}

\providecommand{\biggabs}[1]{\biggl\lvert#1\biggr\rvert}

\DeclareMathOperator{\id}{id}
\newcommand{\1}[0]{\textbf 1}
\newcommand{\EE}{\mathbb{E}}
\newcommand{\RR}{\mathbb{R}}

\newcommand{\NN}{\mathbb{N}}
\newcommand{\ZZ}{\mathbb{Z}}
\newcommand{\PP}{\mathbb{P}}

\newcommand{\cC}{\mathcal{C}}
\newcommand{\cF}{\mathcal{F}}
\newcommand{\cP}{\mathcal{P}}

\newcommand{\cB}{\mathcal{B}}
\newcommand{\cH}{\mathcal{H}}
\newcommand{\cS}{\mathcal{S}}
\newcommand{\cA}{\mathcal{A}}
\newcommand{\cZ}{\mathcal{Z}}

\newcommand{\hm}[1]{\textbf{*}\leavevmode{\marginpar{\tiny%
$\hbox to 0mm{\hspace*{-0.5mm}$\leftarrow$\hss}%
\vcenter{\vrule depth 0.1mm height 0.1mm width \the\marginparwidth}%
\hbox to 0mm{\hss$\rightarrow$\hspace*{-0.5mm}}$\\\relax\raggedright #1}}}
\usepackage{microtype}
\title%[A Banach space-valued ergodic theorem for amenable groups and applications]
{A Banach space-valued ergodic theorem for amenable groups and applications}
\author%[F.~Pogorzelski]
{Felix Pogorzelski\thanks{Friedrich-Schiller-Universit\"at Jena, Mathematisches Institut, 07743 Jena, Germany, felix.pogorzelski@uni-jena.de} \and Fabian Schwarzenberger\thanks{Technische Universit\"a{}t Chemnitz, Fakult\"a{}t f\"u{}r Mathematik, 09107 Chemnitz, Germany, fabian.schwarzenberger@mathematik.tu-chemnitz.de}}
\date{}
\begin{document}

\maketitle

\begin{abstract}
In this paper we study unimodular amenable groups. The first part is devoted to results on the existence of uniform families of $\varepsilon$-quasi tilings for these groups. In light of that, constructions of Ornstein and Weiss in \cite{OrnsteinW-87} are extended by quantitative estimates for the covering properties of the corresponding decompositions. Afterwards, we apply the developed methods to obtain an abstract ergodic theorem for a class of functions mapping subsets of a countable, amenable  group into some Banach space. This significantly extends and complements the previous results in \cite{Lenz-02, LenzMV-08, LenzSV-10, LenzSV-11}. Further, using the Lindenstrauss ergodic theorem (cf.\@ \cite{Lindenstrauss-01}), we describe a link of our results to classical ergodic theory. We conclude with two important applications: the uniform approximation of the integrated density of states on amenable Cayley graphs, as well as the almost-sure convergence of cluster densities in an amenable bond percolation model. 
% Moreover, applications of this convergence result are studied: the uniform existence of the integrated density of states (IDS) for operators on %amenable Cayley graphs; the uniform existence of the IDS for operators on discrete structures being quasi-isometric to some amenable group; the %approximation of $\ell^2$-Betti numbers on cellular CW-complexes; the existence of certain densities of clusters in a percolated Cayley graph.
\end{abstract}

\input{introduction}

\input{Zerlegungen}
\input{Zerlegungen2_neu}

\input{ET_IDS}

\input{IDS}
\input{perc}

\input{appendix}

\bibliographystyle{amsalpha}
\bibliography{PS_lit}

\end{document}

%% file: introduction.tex
\section{Introduction}
In this paper, we study geometric and spectral approximation results for amenable groups. We proceed in the following three major steps. 
\begin{itemize}
\item  We extend the celebrated $\varepsilon$-quasi results for amenable groups of {\sc Ornstein} and {\sc Weiss} \cite{OrnsteinW-87}. We give precise effective estimates on the covering and uniformity properties of the tilings under consideration.
\item  Using the elaborated tiling techniques, we prove an almost-additive ergodic theorem (Theorem~\ref{thm:ET}) which is valid for all countable amenable groups. This generalizes the results in \cite{LenzMV-08, LenzSV-10} and complements \cite{Lenz-02}. The statement applies to Banach space-valued functions which are defined on the space of all finite subsets of the group.   

\item  We conclude the present paper with two major applications concerning the uniform approximation of the integrated density of states for discrete operators, as well as the existence and convergence of cluster densities in a bond percolation model. This generalizes previous results of~\cite{LenzMV-08, LenzSV-10, LenzSV-11, Weissbach-11, Grimmett-76, Grimmett-99} and extends them to the geometric situation of all amenable Cayley graphs.  
\end{itemize}

% These functions are assumed to be almost additive and $\cC$-invariant. Here the latter property means that the functions are compatible with a given colouring $\cC$ of the elements of the group in finitely many colours. 
%As a main example of those functions, one may have eigenvalue counting functions of self-adjoint, %finite hopping range operators in mind, cf.\@ e.g.\@ \cite{LenzS-06,LenzMV-08}. 
%In this situation, an appropriate coloring can for instance be induced by a random field determining %some percolation process or the values of some Schr\"odinger potential.  

%A regularity assumption on this coloring is the existence of certain frequencies %along F\o lner sequences.

The major tools in the proof of our main Theorem~\ref{thm:ET} are given by the results concerning $\epsilon$-quasi tilings of the group, developed in the consecutive Sections 3 and 4. 
The challenge here is to approximate large compact sets in the group by unions of nearly disjoint translates of smaller F{\o}lner sets. 
In our elaborations, we make use of ideas and techniques of {\sc Ornstein} and {\sc Weiss}, elaborated in \cite{OrnsteinW-87}. 
More precisely, we extend their results to obtain effective covering estimates. 
In light of that, we present the underlying constructions in a detailed and rigorous manner and we give a complete picture of the tilings under consideration.   
These results enable us to prove the mentioned ergodic theorem in Section 5, cf.\@ Theorem \ref{thm:ET}. In this context, we consider mappings which take their values according to a coloring $\mathcal{C}$ of the group by finitely many colors ($\mathcal{C}$-invariance). Those functions can be interpreted as abstract analogues of classical ergodic averages. Concerning the general ergodic theory of convergence theorems for amenable group actions, the reader may e.g.\@ refer to~\cite{Ollagnier-85}.   

In fact, we show the Banach space convergence of these averages along F{\o}lner sequences. As an ergodicity assumption, we deal with the case where the pattern frequencies induced by $\mathcal{C}$ exist along the F{\o}lner sequence under consideration.  
Using the Lindenstrauss ergodic theorem (cf.\@ \cite{Lindenstrauss-01}), we discuss sufficient conditions for the validity of the ergodic theorem in Section 6. 
This provides a link to classical ergodic theorems. 
The Sections~7 and~8 are devoted to applications of our main result to amenable Cayley graphs: the uniform approximation of the integrated density of states for finite hopping range operators (Section~7), as well as the investigation of the cluster distribution of a bond percolation model (Section~8). 
We attach an appendix in Section~9, where we give the elementary, but technical proofs of the Lemmas~\ref{lemma:folner} and~\ref{lemma:ow4}.  
% 
% A quasi tiling of a given set $T$ is roughly spoken a union of translates of finitely many sets such the translates are not necessary disjoint, but may overlap up to an $\epsilon$-portion, and the union of the translates is a subset of $T$ which covers all but an $\epsilon$-fraction of $T$. In \cite{OrnsteinW-87} the authors showed that such tilings exist for unimodular amenable groups. For our purposes we needed to strengthen this result in various directions. First of all we obtain detailed control over the densities of certain tiles, in fact we have precise estimates which part of the $T$ is coverd by translates of one certain set. Furthermore we prove a uniform covering property, which means that there exists set of quasi tilings using the same finitly many tiles such that in mean over all these tilings each tile appears at all possible places with the frequency.

%\vspace{2cm}

%AUSFUEHRLICHE BESCHREIBUNG DES INHALTES, MIT ANGABE DER SECTIONS
Having described the rough structure of the paper, let us now discuss its content in further detail. 
We also draw some connections between our work and the literature on similar topics. 
%The rest of this section is devoted to draw connections between the present paper and previous %work on similar topics. 
In Section~2 we introduce the most important notions for the groups under consideration. 
We outline the concepts of amenability, of the $K$-boundary and of different types of F{\o}lner sequences in the setting of unimodular groups. 
Further, some elementary properties are proven for these objects. 

Section~3 is devoted to a first discussion of $\epsilon$-quasi tilings of compact subsets of the group via $\epsilon$-disjoint subsets. 
These methods are used to prove the main decomposition theorems in Section 4. 
Note that an $\epsilon$-quasi tiling of a certain set $T$ is a family of subsets such that the elements of this family may overlap only in a small portion and such that they cover all but a small fraction of $T$, c.f. Definition \ref{defi:STP}. 
We repeat some arguments of {\sc Ornstein} and {\sc Weiss} \cite{OrnsteinW-87} to extend them by quantitative estimates giving an exact description of the shape, as well as of the degree of uniformity of the approximation (tiling) of the set $T$.   
%and based on this we prove important extensions thereof. These considerations are then used in the subsequent section to obtain deeper results on the existence of quasi tilings in unimodular amenable groups. 
%Again the basic concept of these theorems goes back to \cite{OrnsteinW-87}, however we exeed their results in various aspects. 
In light of that, Theorem \ref{thm:STP} shows that each unimodular amenable group satisfies the so-called \emph{special tiling property}. 
Roughly speaking - this condition assures the existence of $\varepsilon$-quasi tilings by translates of finitely many, arbitrarily invariant compact sets, such that the translates of one specific set cover a precisely determined portion. 
In the situation of countable groups, even more can be shown. 
Considering {\em families} of coverings for finite (compact) subsets of the group under consideration, we show in Theorem~\ref{thm:USTP} that on average, each part of the set $T$ can be covered in the same way. 
Covering arguments of this kind have far-reaching applications. 
For instance, in \cite{Huczek-12}, the author uses an adapted construction to prove the existence of zero-dimensional extensions of topological dynamical systems endowed with a free, amenable, countable group action by homeomorphisms. 
Another example is the so-called {\em Equipartition Theorem} for graphs, proven by {\sc Elek} in \cite{Elek-12}. This assertion is concerned with covering properties for weakly convergent, hyperfinite graph sequences with uniformly bounded vertex degree. 

%The present work is complementary, since we obtain tiling results which hold true for {\em all} %countable, amenable groups.  

%T%his yielding assertion extends the case of finitely generated, amenable Cayley graphs to B%enjamini-Schramm convergent, hyperfinite graph sequences. 
%It is not a direct generalization, as our tiling results hold true for all countable, amenable groups.    
%More precisly we ensure the existence of a family of q authuasi tilings consisting of translates of elements of a fixed, finite set of (very invariant) subsets of the group, such that among all tilings each subset appears at all possible positions with (nearly) the same frequency. This property is denoted by the term \emph{uniform special tiling property}, see Definition \ref{defi:USTP}, 

The results of Section 4 form the heart of the proof of the ergodic theorem, stated and proven in Section 5 as Theorem \ref{thm:ET}.
Given an arbitrary countable amenable group $G$, along with a coloring $\cC$ of the group and with a F\o{}lner sequence $(U_j)$ (along which the frequencies of all patterns exist) 
the ergodic theorem implies that for every almost additive and $\cC$-invariant function $F$, mapping finite subsets of the group into some Banach space we have that
\[
\overline{F} := \lim_{j\to\infty}\frac{F(U_j)}{|U_j|}
\]
exists in the topology of the Banach space. 
Furthermore, we can express this limit using a semi-explicit formula, containing the frequencies of the patterns, as well as the densities of different tiles in an $\varepsilon$-quasi tiling, c.f.\@ (ii) of Theorem \ref{thm:ET}. 
Assertions of this kind have been treated before. In \cite{Lenz-02}, the author characterizes unique ergodicity of a minimal subshift over some finite alphabet by the validity of a Banach space valued convergence theorem. 
%The idea to use a Banach space valued ergodic theorem for the uniform approximation of the integrated density %of states was developed in \cite{LenzS-06}. 
The works \cite{LenzS-06, LenzMV-08, LenzSV-10} consider the normalized convergence of Banach space valued, almost-additive functions along F{\o}lner sequences in certain countable, amenable groups.  
%Similar ergodic theorems for discrete models on the integers or in the euclidean space can be found in %\cite{Lenz-02,Klassert-07}. 
%For results which are analogous to our ergodic theorem on certain restrictive geometries, see e.g.\@ %\cite{LenzMV-08,LenzSV-10}. 
The hitherto most general result for groups can be found in the latter paper, where the authors are able to deal with a specific class of amenable groups. 
More precisely, it is assumed that the group contains a F{\o}lner sequence consisting of monotiles for the group and the associated grids must be symmetric (cf.\@ \cite{LenzSV-11}). 
Although the validity of these conditions is satisfied for all {\em residually finite}, countable, amenable groups, it is not clear at all how big the gap to the 
general case is, see also \cite{OrnsteinW-87}. 
Using the sophisticated machinery of $\varepsilon$-quasi tiles, there is no need to impose any restriction on the group, i.e.\@ we establish the approximation result of Theorem \ref{thm:ET} for 
arbitrary countable, amenable groups. 
Moreover, recent results generalize the context of finitely generated, uncolored groups to bounded vertex degree graphs. 
In \cite{Pogorzelski-13}, the author proves a convergence theorem for periodic (uncolored), almost-additive functions along weakly convergent, hyperfinite graph sequences. 
The Equipartition Theorem in \cite{Elek-12} is an essential ingredient of the proof. Those latter results are complementary to the present paper. 
% On the one hand, they are not restricted to the context of amenable Cayley graphs but deal with a larger class of bounded vertex degree graphs. 
Dealing with amenable Cayley graphs, we do not have to stay within the context of finitely generated or periodic structures. In fact, we prove a convergence result for {\em colored} almost-additive functions which holds true for {\em all} countable, amenable groups. Using the new $\varepsilon$-quasi tiles, we are also able to derive an abstract error estimate for the convergence which depends on the speed of convergence of the occurrence frequencies of the colored patterns, cf.\@ Inequality~(\ref{eqn:EST}).    

Considering the set $\RR$ of the real numbers as the Banach space, it is worth raising the issue of a similar convergence result for subadditive set functions $f$. In \cite{LindenstraussW-00}, the authors use an assertion of that kind in order to compute the topological mean dimension for dynamical systems. 
Exploiting the $\varepsilon$-quasi tilings proven in \cite{OrnsteinW-87}, it is shown in \cite{Krieger-07,Krieger-10} that the limit $\lim_{j \rightarrow \infty}|U_j|^{-1}f(U_j)$ exists for F{\o}lner nets $(U_j)$. A semigroup result including the investigation of topological entropy can be found in \cite{CeccheriniKC-13}. 
However, the authors need to assume periodicity of $f$ (coloring of the group with just one color), a condition which is too restrictive for the purposes we have in mind. 

The existence of frequencies of patterns along a F\o{}lner sequence is an assumption in the presented ergodic theorem. 
Section 6 provides sufficient conditions for the existence of frequencies in a randomly colored graph. 
Here we draw also connections to the pointwise ergodic theorem by {\sc Lindenstrauss} \cite{Lindenstrauss-01}.

With Theorem~\ref{thm:ET} at hand, one can study various approximation results for discrete models in a wide range of geometric situations. 
In Section 7, we prove the uniform approximation via finite volume analogues of the integrated density of states (IDS) for discrete space operators on the group. The idea to use a Banach space valued ergodic theorem for this purpose has been developed in \cite{LenzS-06} in the context of Delone dynamical systems. 
In the present paper, we generalize the corresponding assertions in \cite{Elek-06pre, LenzMV-08, LenzSV-10, LenzSV-11}. Further, our elaborations are complementary to the results in the more restrictive situation of {\em periodic} operators on graph sequences, cf.\@ \cite{Elek-08, Pogorzelski-13}.   
 
%Using convergence theorems, this topic has been examined before e.g.\@ in \cite{LenzS-06,Elek-06pre, %Elek-08, LenzMV-08, LenzSV-10, LenzSV-11}. 
%Beside this, the established ergodic theorem allows applications to operators defined on discrete sets, which only need to be quasi-isometric to a countable amenable group. Again we %can prove the uniform existence of the IDS.
%Furthermore one can apply the ergodic theorem to the approximation of $\ell^2$-Betti numbers for discrete spaces consisting of finite-dimensional cells, so-called CW-complexes. 
%This topic was studied in several papers with a related geometric setting, see \cite{DodziukM-98,Eckmann-99,Schick-01}.
We conclude our paper in Section 8 with a bond percolation model for Cayley graphs. 
Here our main Theorem~\ref{thm:ET} is applied to prove the existence of occurrence densities of clusters of a fixed size. 
In addition to this, we show the approximability of associated distribution functions, as well as the continuous dependence of these quantities on the percolation parameter. 
As the distribution of cluster sizes in a percolation process is an intensively studied field, our results are complementary to other works.
For instance, we generalize results of {\sc Grimmett} in \cite{Grimmett-76, Grimmett-99}, where the author proved continuous dependence of certain percolation quantities on the percolation parameter $p$ in the setting of $\ZZ^d$.
Besides this, in \cite{AizenmanDS-80,BandyopadhyayST-10}, the authors are interested in quantitative estimates concerning certain occurrence probabilities of clusters in the lattice case and for amenable groups, respectively.   
In particular, the asymptotic behavior for large cluster sizes has been studied before in the literature, e.g.~in connection with the so-called ``sharpness of the phase transition'' \cite{AizenmanN-84,AizenmanB-87,Menshikov-86,MenshikovMS-86,AntunovicV-08}.

%% file: Zerlegungen.tex
\section{Preliminaries}

Throughout this paper, we assume $G$ to be a unimodular, second countable, amenable Hausdorff group. The following section is devoted to the presentation of general properties of amenable groups. More precisely, we introduce a notion for a relative boundary of subsets in $G$ and we use this concept to define so-called weak and strong F{\o}lner sequences. The existence of weak F{\o}lner sequences is commonly referred to as a characterization of amenability of second countable Hausdorff groups. As shown below, each strong F{\o}lner sequence is also a weak F{\o}lner sequence and the groups under consideration always possess a strong F{\o}lner sequence (see Lemma \ref{lemma:folner}).

Let $\cB(G)$ be the $\sigma$-algebra generated by the open sets of $G$. For a set $A\in \cB(G)$ we denote by $|A|$ the Haar measure of $A$ (for countable groups $|A|$ is the counting measure of $A$). 
The set of all finite subsets of $G$ is called $\cF(G)$. Independently of the Haar measure on $G$ we write $\sharp(A)$ for the cardinality of a set $A\in\cF(G)$. 
The unit element of the group is denoted by $\id$.
%We start with a definition of amenability.
\begin{Definition}\label{defi:AMENABLE}
 An unimodular group $G$ is called \emph{amenable} if for each $\epsilon>0$ and $K\subseteq G$ compact there is a compact set $F\subseteq G$ and $K_0\subseteq K$ such that
\begin{enumerate}[(i)]
 \item  $|F\setminus kF| < \epsilon |F|$ for all $k\in K_0$
 \item  $|K\setminus K_0|<\epsilon$
\end{enumerate}
\end{Definition}

Let us proceed with definition of boundary terms of a group.

\begin{Definition} \label{defi:KBD}
Let $\emptyset \neq K, T \subseteq G$ be compact subsets in $G$. We call the set $\partial_K(T)$, defined by
\begin{eqnarray*}
\partial_K(T):= \{g \in G\,|\, Kg \cap T \neq \emptyset \,\wedge\, Kg \cap (G\setminus T) \neq \emptyset\}
\end{eqnarray*}
the \emph{$K$-boundary} of the set $T$.
Furthermore, a subset $T$ with $|T|>0$ is called \emph{$(K,\delta)$-invariant} if
\begin{eqnarray*}
\frac{|\partial_{K}(T)|}{|T|}   < \delta.
\end{eqnarray*}
\end{Definition}

\begin{Remark} \label{remark:fingen}
 In section 7 and 10 we will consider the case where $G$ is a finitely generated group, i.e. there exists a finite and symmetric generating system $S\subseteq G$. In this situation it is convenient to define the so called word metric $d_S:G\times G\to \NN_0$ on $G$. Here one sets the distance of two distinct elements $x,y\in G$ to be smallest number of elements in $S$ one needs to carry $x$ into $y$, i.e.
\[
 d_S(x,y):=\min \{ n\in \NN\mid \exists s_1,\dots,s_n\in S \text{ with } s_1\cdots s_n =xy^{-1} \} \quad\text{and}\quad d_S(x,x):=0.
\]
Using this distance one can define balls and boundaries as follows: the ball of radius $r\in \NN$ around the element $x\in G$ is $B_r(x):=\{y\in G\mid d_S(x,y)\leq r\}$ and $B_r:=B_r(\id)$. The $r$-boundary of a set $\Lambda\in \cF(G)$ is 
\[
 \partial^r(\Lambda):=\{x\in\Lambda\mid d_S(x,G\setminus\Lambda)\leq r\}\cup \{x\in G\setminus\Lambda \mid d_S(x,\Lambda)\leq r\}.
\]
In this situation it is easy to show that $\partial_{B_r}(\Lambda)=\partial^r(\Lambda)$ holds for all $\Lambda\in\cF(G)$ and $r\in\NN$.
\end{Remark}

We will see below that the $K$-boundary Definition \ref{defi:KBD} can be used for an appropriate notion for a F{\o}lner condition for sets in amenable groups. In addition to that, it has very convenient properties which are easy to deal with. In the following lemma, we provide a short list of those properties which are important for our purposes.

\begin{Lemma} \label{prop:prop}
Let $T,S,K \subseteq G$ be non-empty and compact and assume that $g \in G$. Then the following is true.
\begin{enumerate}[(i)]
\item $\partial_K(T) = \partial_K(G\setminus T)$.
\item $\partial_K(S \cup T) \subseteq \partial_K(S) \cup \partial_K(T)$.
\item $\partial_K(S \setminus T) \subseteq \partial_K(S) \cup \partial_K(T)$.
\item $|\partial_K(S \setminus T)| \leq |\partial_K(T)| + |\partial_K(S)|$
\item $\partial_K(T) \subseteq \partial_L(T)$ if $K \subseteq L \subseteq G$.
\item $\partial_K(Tg) = \partial_K(T)g$.
\item $\partial_K(TS) \subseteq \partial_K(T)S$.
\item $\partial_K(T \setminus S) \subseteq \partial_K(T) + (\partial_K(S) \cap T)$ if $\operatorname{id} \in K$.
\end{enumerate}
\end{Lemma}
\begin{proof}
The statements (i) to (vi) follow easily from Definition \ref{defi:KBD}.
To prove (vii) we fix $g \in \partial_K(TS)$. Thus $Kg \cap TS \neq \emptyset$ which implies that there is some $c \in S$ with $Kg \cap Tc \neq \emptyset$.
Since $Kg \cap (G\setminus TS) \neq \emptyset$, it then follows that $Kg \cap (G \setminus Tc) \neq \emptyset$.
We conclude that for every $g \in \partial_D(TS)$ we can find some $c \in S$ such that $gc^{-1} \in \partial_K(T)$, i.e. $g \in \partial_K(T)c \subseteq \partial_K(T)S$.\\
We finally prove the assertion (viii). So assume that $g \in \partial_K(T\setminus S)$, but $g \notin \partial_K(T)$. Then the fact that $Kg \cap (T\setminus S) \neq \emptyset$ leads to $Kg \cap T \neq \emptyset$, as well as to $Kg \cap (G \setminus S) \neq \emptyset$. Thus, if $g \notin \partial_K(T)$, this is only possible if $Kg \subseteq T$. Since $\operatorname{id} \in K$, it follows that $g \in T$. It remains to show that $Kg \cap S \neq \emptyset$, because then $g \in \partial_K(S)$. Indeed, since $g \in \partial_K(T\setminus S)$, we have $Kg \cap (G\setminus(T\setminus S)) \neq \emptyset$ and from $Kg \subseteq T$, it follows that $Kg \cap (S \cap T) \neq \emptyset$. This shows the claim. 
\end{proof}

\begin{Definition} \label{defi:FS}
Let $(F_n)$ be a sequence of non-empty compact subsets of a unimodular group $G$. If
\begin{align*}
\lim_{n \rightarrow \infty} \frac{|F_n\triangle K F_n|}{|F_n|} = 0
\end{align*}
for all non-empty, compact $K \subseteq G$, then $(F_n)$ is called {\em weak F{\o}lner sequence}. If
\begin{align*}
\lim_{n \rightarrow \infty} \frac{|\partial_K( F_n)|}{|F_n|} = 0
\end{align*}
for all non-empty, compact $K \subseteq G$, then $(F_n)$ is called {\em strong F{\o}lner sequence}.
We say that a (weak or strong) F{\o}lner sequence $(F_n)$ is {\em nested} if $\id\in F_1$ and $F_n\subseteq F_{n+1}$ for all $n \geq 1$. Furthermore a (weak or strong) F{\o}lner sequence $(F_n)$ is called \emph{tempered} if there exists $c>0$ such that
\[
 \left|\bigcup_{k<n}F_k^{-1}F_n\right|\leq c |F_n|\quad \quad (n\in \NN).
\]
\end{Definition}

In \cite{OrnsteinW-87} the authors proved that in each amenable group the following holds: given a compact set $K\subseteq G$ and $\delta>0$ there is a compact set $F$ which is $(K,\delta)$-invariant. It follows from this that given an amenable, unimodular and second countable group, one can always find strong F\o lner sequences. Furthermore, each strong F\o lner sequence is a weak F\o lner sequence. We collect these observations in the following Lemma. A proof can be found in the appendix, cf.\@ Section~\ref{sec:appendix}. 

\begin{Lemma}\label{lemma:folner}
Let $G$ be unimodular. The following statements hold:
\begin{enumerate}[(a)]
 \item If $G$ is amenable and second countable, then there exists a strong F\o lner sequence in $G$.
 \item Each strong F\o lner sequence is a weak F\o lner sequence.
 \item If there exists a weak F\o lner sequence in $G$, then $G$ is amenable.
 \item If $G$ is countable, then each weak F\o lner sequence is also a strong F\o lner sequence.
 \item If there exists a strong F\o lner sequence in $G$, then there exists also a nested strong F\o lner sequence in $G$.
 \item Each (strong or weak) F\o lner sequence has a tempered subsequence.
\end{enumerate}
\end{Lemma}

\begin{proof}
See appendix, Section~\ref{sec:appendix}. 
\end{proof}

\begin{Remark}
 Lemma \ref{lemma:folner} shows that in an unimodular, second countable group amenability is equivalent to both the existence of a weak F\o lner sequence, as well as to the existence of a strong F\o lner sequence. For further discussions in that direction see for example \cite{Greenleaf-69, OrnsteinW-87} and the seminal papers \cite{Folner-55, Adachi-93}.

 In the case of countable amenable groups we will only speak about F\o lner sequences as the specifications weak and strong are dispensable then.
\end{Remark}

In the following, for a number $b \in \RR$, we will write $\lceil b\rceil$ for the smallest integer greater than or equal to $b$, 
i.e. $\lceil b\rceil := \inf\{ m \in \ZZ \,|\, m \geq b \}$. Furthermore, we use the following notion. Given a number $0<\epsilon<1$, the number $N(\epsilon)$ is given by
\begin{equation}\label{eq:def:N}
 N(\epsilon):=  \lceil \log(\varepsilon)/\log(1-\varepsilon)\rceil.
\end{equation}

For later considerations, we will need the following technical lemma.
\begin{Lemma} \label{lemma:ele4}
Let $0 < \varepsilon < 1$ and assume that $(\alpha_i)_{i \in \NN}$ is a complex-valued null sequence. Then,
\begin{eqnarray*}
\lim_{\varepsilon\rightarrow 0} \sum_{i=1}^{N(\varepsilon)} \varepsilon (1-\varepsilon)^{N(\varepsilon) - i} \cdot \alpha_{i} = 0.
\end{eqnarray*}
\end{Lemma}

\begin{proof}
Let $\delta > 0$ be arbitrary and choose $N_\delta \in \NN$ such that $|\alpha_{i}| < \delta$ whenever $i \geq N_\delta$. As $(\alpha_{i})$ is a null-sequence, there is some uniform bound $C>0$ for the absolute values of $\alpha_i$, $i\in\NN$. We compute
\begin{align*}
\left|   \sum_{i=1}^{N(\varepsilon)} \varepsilon (1-\varepsilon)^{N(\varepsilon)-i}  \alpha_{i} \right|
&=  \left| \sum_{i=1}^{N_{\delta}}  \varepsilon (1-\varepsilon)^{N(\varepsilon)-i}  \alpha_{i} + \sum_{i=N_{\delta} + 1}^{N(\varepsilon)}  \varepsilon (1-\varepsilon)^{N(\varepsilon)-i}  \alpha_{i} \right| \\
&\leq   C \varepsilon  N_{\delta} + \delta \cdot \sum_{i=N_{\delta} + 1}^{N(\varepsilon)} \varepsilon (1-\varepsilon)^{N(\varepsilon) - i}
\leq  C\varepsilon N_{\delta} + \delta  
\end{align*}
Here the last inequality holds since $\sum_{i=1}^{N(\varepsilon)} \varepsilon (1-\varepsilon)^{N(\varepsilon) - i}=\sum_{i=0}^{N(\varepsilon)-1} \varepsilon (1-\varepsilon)^{i}\leq 1$. Since $\delta$ is arbitrary, the claim follows.
\end{proof}

\section{Ornstein/Weiss tiling lemmas} \label{sec:OW}

This section is devoted to general tiling results for unimodular groups. We expand the elaborations
of {\sc Ornstein/Weiss} in \cite{OrnsteinW-87} to give quantitative estimates for coverings of subsets of the group, cf.\@ Lemma \ref{lemma:tilingnew}. These results will be of great importance for the proofs of the tiling theorems in the next section. 

%For completeness reasons, we also include and prove two Lemmas (\ref{lemma:ow1} and %\ref{lemma:ow2}) which have already been proven in the paper of {\sc Ornstein/Weiss}. Moreover, %n Lemma \ref{lemma:ow4}, we give a rigorous proof of a statement that relies on a remark %mentioned in \cite{OrnsteinW-87}.   
%Although the Lemmas \ref{lemma:ow1} and \ref{lemma:ow2} can already be found in the paper of {\sc Ornstein/Weiss},     

We start with the notion of $\varepsilon$-disjoint subsets.

% Part of the material is not new, but was published by {\sc Ornstein} and {\sc Weiss} in the late %eighties, see \cite{OrnsteinW-87}. We present some results which will be of great importance for the %proofs of the tiling theorems in the next section.
%Lemma \ref{lemma:ow1} and special case of Lemma \ref{lemma:ow2} are to be found in \cite{OrnsteinW-87}, %however for the sake of the reader, we the arguments of their proofs. The statement of Lemma %\ref{lemma:ow4} is also not new, but relies on a remark {\sc Ornstein} and {\sc Weiss} made in their %paper. We give a detailed proof for that. In Lemma \ref{lemma:tilingnew} we make important refinements %of the ideas of {\sc Ornstein} and {\sc Weiss} in order to obtain detailed control over the density of a %certain tiling.

\begin{Definition} \label{defi:epsdis}
Let $G$ be a group and assume that $T_1, T_2$ are subsets of $G$ with $|T_1|, |T_2| < \infty$. For $0 < \varepsilon < 1$, we say that the sets $T_1$ and $T_2$ are $\varepsilon$-disjoint if there are sets $S_i \subseteq T_i$, $i=1,2$ such that
\begin{enumerate}[(i)]
\item $S_1 \cap S_2 = \emptyset$,
\item $|S_i| \geq (1-\varepsilon)|T_i|$ for all $i=1,2$.
\end{enumerate}
Let $I$ be some index set. A family $\{T_i\}_{i\in I}$ of subsets of $G$ is called $\epsilon$-disjoint if any two sets of the family are $\epsilon$-disjoint.
\end{Definition}   

%The following two Lemmas have already been proven in \cite{OrnsteinW-87}, Section I.3.

We now turn to the two main results of this section which will be exploited extensively in the next section. The first one states that given $\varepsilon > 0$ and given some $(KK^{-1}, \delta)$-invariant set $T \subseteq G$, we can cover a portion of $T$, which lies between $\varepsilon -\delta$ and $\epsilon+\delta$ by $\varepsilon$-disjoint translates of $K$. The lower bound on the mass proportion of $T$ covered by $K$-translates is a new feature of Ornstein/Weiss covering statements. Moreover, we assume $K$ to be $(B,\zeta^2)$-invariant for some non-empty, compact $B\subseteq G$. This fact will guarantee that there are pairwise disjoint subsets of the translates which are $(B,4\zeta)$-invariant and which still cover the same set as the translates of $K$.
Our proof is constructive and rather technical. We give the proof in the appendix, cf.\@ Section~\ref{sec:appendix}.

\begin{Lemma} \label{lemma:ow4}
Let $G$ be some unimodular group, $0 < \varepsilon,\delta < 1/2$ and $0<\zeta<\delta/2$.
Furthermore let $T,K,B\subseteq G$ be compact sets such that $T$ is $(KK^{-1},\delta)$-invariant, $K$ is $(B,\zeta^2)$-invariant and let the sets $K$ and $B$ contain the unit element.
Then we can find finitely many elements $c_j, j=1,\dots,n$ in $T$ such that
\begin{enumerate}[(i)]
\item $Kc_j \subseteq T$, $j=1,\dots,n$.
\item $\{Kc_j\}_{j=1}^n $ is $\varepsilon$-disjoint.
\item for all $j=1,\dots,n$, there is some set $K_j \subseteq K$ with $|K_j| \geq (1-\varepsilon)|K|$ such that
\begin{itemize}
\item $K_j$ is $(B, 4\zeta)$-invariant
\item $|\partial_B(K_j)|\leq |\partial_B(K)| +\zeta|K|$
\item $\bigcup_{j=1}^l Kc_j = \bigcup_{j=1}^l K_j c_j$ for all $1 \leq l \leq n$ and the latter unions consist of pairwise disjoint sets. 
\end{itemize}
\item $(\epsilon -\delta)|T| \leq \big| \bigcup_{j=1}^n Kc_j  \big| \leq (\epsilon+\delta) |T|$.
\end{enumerate} 
\end{Lemma}

\begin{proof}
See appendix, Section~\ref{sec:appendix}. 
\end{proof}

In the following lemma we refine the result of Lemma~\ref{lemma:ow4}. We show that under additional assumptions we can ensure that the part of $T$ which is not covered by tiles, still satisfies an invariance property. 

\begin{Lemma}\label{lemma:tilingnew}
Let $G$ be some unimodular group, $0<\epsilon,\delta<1/6$, $0<\zeta<\delta/4$ and $\eta>0$. Furthermore let $T,K,L,B\subseteq G$ be compact sets with $\id\in L\subseteq K$, $\id\in B$ and $T$ be $(KK^{-1},\delta)$-invariant and $K$ be $(LL^{-1},\eta)$-invariant, as well as $(B,\zeta^2)$-invariant. Then there is a set $C\in \cF(G)$ such that $T\setminus KC$ is $(LL^{-1},2\delta+\eta)$-invariant and the properties (i) to (iv) of Lemma \ref{lemma:ow4} are satisfied: $\{Kc\}_{c\in C}$ is $\epsilon$-disjoint, $Kc\subseteq T$ for all $c\in C$ and
\begin{equation}\label{eqn:cover1}
 \epsilon-\delta\leq \frac{|K C|}{|T|}\leq \epsilon +\delta
\end{equation}
holds. Furthermore, for each $c\in C$ there is a set $K_c\subseteq K$ which is $(B,4\zeta)$-invariant and satisfies $|K_c|\geq (1-\epsilon)|K|$ and $|\partial_B(K_c)|\leq |\partial_B(K)|+\zeta |K|$. The sets $K_c c$, $c\in C$ are pairwise disjoint and $KC= \bigcup_{c\in C}K_c c$.
\end{Lemma}
\begin{proof}
As the assumptions of Lemma \ref{lemma:ow4} are satisfied, we get a set $C=\{c_1,\dots,c_n\}$ such that the properties (i) to (iv) therein are fulfilled. It remains to prove that $T\setminus KC$ is $(LL^{-1},2\delta+\eta)$-invariant. Therefore use the properties of the sets $K_c$, $c\in C$ and \eqref{eqn:cover1} to obtain
\begin{eqnarray} \label{eqn:cov1}
|T| \geq \frac{|KC|}{\epsilon+\delta}= \frac{\sum_{c\in C}|K_c|}{\epsilon+\delta} \geq \frac{1-\epsilon}{\epsilon + \delta}\,|K|\cdot \sharp(C).
\end{eqnarray}
In light of that, using the upper bounds on $\epsilon$ and $\delta$, we can compute
\begin{align} \label{eqn:cov1.1}
\frac{\sharp(C)}{|T|- |KC|} \leq \frac{\sharp(C)}{|T|-|K|\sharp(C)} 
\leq  \frac{\sharp(C)}{(\frac{1-\varepsilon}{\varepsilon + \delta}-1) |K|\sharp(C) }
= \frac{\varepsilon + \delta}{(1- 2\varepsilon-\delta)|K|} \leq \frac{1}{|K|}.
\end{align}
Now apply items (iv) and (vii) of Lemma \ref{prop:prop} to obtain 
\begin{align} \label{eqn:breakup1}
\frac{|\partial_{LL^{-1}}(T\setminus KC)|}{|T\setminus KC|} 
\leq \frac{|\partial_{LL^{-1}}(T)|}{|T \setminus KC|} + \frac{|\partial_{LL^{-1}}(KC)|}{|T \setminus KC|} 
\leq \frac{|\partial_{LL^{-1}}(T)|}{|T|-|KC|} + \frac{\sharp(C)\cdot |\partial_{LL^{-1}}(K)|}{|T|-|KC|}.
\end{align} 
It follows from Inequality (\ref{eqn:cover1}) that $|T| - |KC| \geq (1-(\varepsilon + \delta))|T|$. Combining this fact with Inequality (\ref{eqn:cov1.1}), we deduce from (\ref{eqn:breakup1}) that
\begin{align*}
\frac{|\partial_{LL^{-1}}(T\setminus KC)|}{|T\setminus KC|}  
\leq \frac{|\partial_{LL^{-1}}(T)|}{\left(1-(\varepsilon + \delta)\right)|T|} + \frac{|\partial_{LL^{-1}}(K)|}{|K|} \leq 2\delta+\eta,
\end{align*} 
which shows our claim. Note that here, we used that $T$ is $(LL^{-1}, \delta)$-invariant since $L \subseteq K$ and since $T$ is $(KK^{-1}, \delta)$-invariant, c.f. Lemma \ref{prop:prop}, (v).
\end{proof}

\section{Tiling theorems} \label{sec:OW2}
In this section we provide results concerning so-called $\varepsilon$-quasi-tilings of amenable groups. Continuing as before, the corresponding theorems are inspired by the work of {\sc Ornstein} and {\sc Weiss} in \cite{OrnsteinW-87}, but also significantly extended by quantitative statements for the resulting tilings of invariant compact subsets of $G$. In light of that, we summarize the essential properties of such a quasi-tiling in Definition \ref{defi:STP} and we show in Theorem \ref{thm:STP} that these tilings exist in every unimodular amenable group. Though part of these properties have been shown by {\sc Ornstein} and {\sc Weiss}, our theorem gives a far more detailed description of the quasi-tiling. More precisely, we are able to calculate exactly the densities determining the portion of covered mass by the translates of each individual tiling set $T_i$ (see item (iv) of Definition \ref{defi:STP}). In part (b) of Theorem \ref{thm:STP} we also show how one obtains a tiling with pairwise disjoint sets 
which can be made arbitrarily invariant with respect to any fixed compact subset of $G$. While Theorem \ref{thm:STP} proves the existence of one single $\varepsilon$-quasi-tiling, Theorem \ref{thm:USTP} deals with a family of $\varepsilon$-quasi-tilings and shows a uniform covering property for countable groups. Roughly speaking, the latter statement says that among the different tilings, all admissible $T_i$-translates occur as a tile with the same density/probability (see Definition \ref{defi:USTP}, property (II)).        

%The presentet theorems are based on the work of {\sc Ornstein} and {\sc Weiss} in \cite{OrnsteinW-87}. %However they are stengthend and generalized in various aspects. We summarize the most important %properies of a quasi-tiling in Definition \ref{defi:STP} and show in Theorem \ref{thm:STP} that they are %satisfied for an arbitrary unimodular amenable group. Though part of these properties have been shown by %{\sc Ornstein} and {\sc Weiss} our theorem gives a way more detailed characterization of the %quasi-tiling. Here the main aspect is that we are able to describe the densities with which each tile %$T_i$ appears in the associated quasi-tiling, c.f. item (iv) of Definition \ref{defi:STP}. In part (b) %of Theorem \ref{thm:STP} we show how one obtains a tiling with arbitrary invariant and pairwise disjoint %sets.
%While Theorem \ref{thm:STP} proves the existence of one quasi-tiling, Theorem \ref{thm:USTP} deals with %a family of quasi-tilings and shows a uniform covering property. Roughly spoken it says that among the %different tilings a fixed tile $T_i$ appears everywhere with the same density/probability.

As introduced in Equality \eqref{eq:def:N}, we will use $N(\epsilon)=\lceil \log(\varepsilon)/\log(1-\varepsilon)\rceil$ for $0<\epsilon<1$.

\begin{Definition}[Special tiling property] \label{defi:STP}
Let $G$ be an amenable group. We say that $G$ has the {\em special tiling property} if for any given $\beta,\epsilon$ with $0 <\beta < \varepsilon \leq 1/10$ and nested F\o lner sequence $\cS=(S_n)$ there is a number $N=N(\varepsilon) \in \NN$ and there are sets
\begin{eqnarray*}
\{\id \} \subseteq T_1 \subseteq T_2 \subseteq \dots \subseteq T_{N(\epsilon)}\quad\text{with}\quad T_i\in \{S_n\mid n\geq i\} \quad (1\leq i \leq N(\epsilon))
\end{eqnarray*}
and some $\delta_0=\delta(\beta) > 0$ depending only on $\beta$ such that for all positive $\delta < \delta_0$ and every $(T_N T_N^{-1}, \delta)$-invariant set $T\in\cB(G)$, we can find so-called {\em center sets} $C_i^T\in\cF(G), i=1,\dots,N(\epsilon)$ such that
\begin{enumerate}[(i)]
\item $T_i C_i^T \subseteq T$ for all $i=1,\dots,N(\epsilon)$,
\item $\{T_i c\}_{c \in C_i^T}$ is an $\varepsilon$-disjoint family for all $i=1,\dots,N(\epsilon)$,
\item $\{T_i C_i^T\}_{i=1}^{N(\epsilon)}$ is a disjoint family of sets,
\item $\left| \frac{|T_i C_i^T|}{|T|} -  \eta_i(\epsilon) \right| < \beta$ \quad for all $i=1,\dots,N(\epsilon)$, where $\eta_i(\epsilon):=\varepsilon(1-\varepsilon)^{N(\epsilon)-i} $.
\end{enumerate}
In this case, we say that the $\{T_i\}$ \emph{$\epsilon$-quasi tile} the group $G$ and if for $T \subseteq G$, the properties (i)-(iv) hold, we say that $T$ has the {\em special tiling property (STP)} with respect to $(\{T_i\}_{i=1}^{N(\epsilon)},\cS, \varepsilon, \beta)$ and that $T$ is \emph{$\varepsilon$-quasi tiled} (with parameter $\beta$) by the basis sets $T_i$ with finite center sets $C_i^T$.
\end{Definition}

%\begin{Definition}
% \begin{Definition}[Special tiling property] \label{defi:STP}
% Let $G$ be an unimodular amenable group. We say that $G$ has the {\em special tiling property} if for any given $0 < \varepsilon \leq 1/10$ there are sets
% \begin{eqnarray*}
% \{\id \} \subseteq T_1^\epsilon \subseteq T_2^\epsilon \subseteq \dots \subseteq T_N^\epsilon
% \end{eqnarray*}
% and some $\delta_\epsilon > 0$ such that for all positive $\delta < \delta_\epsilon$ and every $(T_N^\epsilon (T_N^\epsilon)^{-1}, \delta)$-invariant set $T$, we can find so-called {\em center sets} $C_i^T\subseteq G, i=1,\dots,N(\epsilon)$ such that
% \begin{enumerate}[(i)]
% \item $T_i^\epsilon C_i^T \subseteq T$ for all $i=1,\dots,N(\epsilon)$,
% \item $\{T_i^\epsilon c\}_{c \in C_i^T}$ is an $\varepsilon$-disjoint family for all $i=1,\dots,N(\epsilon)$,
% \item $\{T_i^\epsilon C_i^T\}_{i=1}^{N(\epsilon)}$ is a disjoint family of sets,
% \item $\left| \frac{|T_i^\epsilon C_i^T|}{|T|} -  \eta_i(\epsilon) \right| < 2^{-i}\varepsilon$ \quad for all $i=1,\dots,N(\epsilon)$, where $\eta_i(\epsilon):=\varepsilon(1-\varepsilon)^{N(\epsilon)-i} $
% \item $T_i^\epsilon\in \{S_n\mid n\geq i\}$, where $(S_n)_{n\in\NN}$ is some nested F\o lner sequence.
% \end{enumerate}
% If for $T \subseteq G$, the properties (i)-(v) hold, we say that $T$ has the {\em special tiling property (STP)} with respect to $(\{T_i^\epsilon\}_{i=1}^{N(\epsilon)}, \varepsilon)$.
% \end{Definition}
% 
\begin{Definition}
We say that $S \subseteq G$ $\alpha$-covers a set $T \subseteq G$ for $0 < \alpha \leq 1$ if
\[
|S \cap T| \geq \alpha \cdot |T|. 
\] 
\end{Definition}

\begin{Remark}Let us discuss the previous definitions.
 Item (iv) of the special tiling property shows that for $0<\epsilon\leq 1/10$ the values $\eta_i(\epsilon)=\epsilon(1-\epsilon)^{N(\epsilon)-i}$ can be interpreted as the ``almost densities'' of the elements  $T_i$, $i\in \{1,\dots,N(\epsilon)\}$ in the $\varepsilon$-quasi-tiling. This is emphasized by the fact that the $\eta_i(\epsilon)$, $i\in \{1,\dots,N(\epsilon)\}$ almost sum up to one (up to an $\epsilon$). In fact we have
\begin{equation*}
 1-\epsilon \leq \sum_{i=1}^{N(\epsilon)} \eta_i(\epsilon) \leq 1.
\end{equation*}
This is clear as $N(\epsilon)=\lceil\log (\epsilon)/\log(1-\epsilon)\rceil$ and 
\begin{equation}\label{eqn:remark4}
 \sum_{i=1}^{N(\epsilon)}\epsilon (1-\epsilon)^{N(\epsilon)-i} =\epsilon \sum_{i=0}^{N(\epsilon)-1} (1-\epsilon)^i =1-(1-\epsilon)^{N(\epsilon)}\leq 1 .
\end{equation}
Furthermore 
\[
 1-(1-\epsilon)^{N(\epsilon)} \geq 1-(1-\epsilon)^{\log(\epsilon)/\log(1-\epsilon)} = 1-\epsilon
\]
holds for all $\epsilon\in(0,1)$.

 Oftentimes we are interested in the case where $\beta=2^{-N(\epsilon)}\varepsilon$. Then it follows from Definition \ref{defi:STP} that each $T\subseteq G$ with the STP is $(1-2\varepsilon)$-covered by the corresponding $T_i$-translates. To see this, note that by the property (iv) and \eqref{eqn:remark4}, we can compute
\begin{align*}
\frac{\left|\bigcup_{i=1}^{N(\epsilon)} T_iC_i^T \cap T\right|}{|T|} &= \sum_{i=1}^{N(\epsilon)} \frac{|T_iC_i^T|}{|T|} 
\geq \sum_{i=1}^{N(\epsilon)} \left(\varepsilon (1-\varepsilon)^{N(\epsilon)-i} - 2^{-N(\epsilon)}\varepsilon\right)
\geq 1-(1-\epsilon)^{N(\epsilon)}-\epsilon
\end{align*}
The definition of $N(\epsilon)$ gives $(1-\epsilon)^{N(\epsilon)}\leq \epsilon$, which proves the claim.
\end{Remark}

We now prove an extension of Theorem 6 in \cite{OrnsteinW-87}. It shows that tilings of sufficiently invariant sets $T \subseteq G$ as described in Definition \ref{defi:STP} always exist for unimodular amenable groups.

\begin{Theorem} \label{thm:STP} Let $G$ be an unimodular amenable group. Then the following holds true.
 \begin{enumerate}
  \item[(a)] The group $G$ satisfies the special tiling property.
  \item[(b)] Let $\beta,\epsilon$ and $\cS$ be given as in Definition \ref{defi:STP} and let $B\subseteq G$ be compact with $\id\in B$ and $0 < \zeta < \varepsilon$.
Then we can find an $\epsilon$-quasi-tiling with $(B, \zeta^2)$-invariant sets $T_i$ $(i=1,\dots, N(\epsilon))$ fulfilling all the properties of Definition \ref{defi:STP} such that each $(T_NT_N^{-1},\delta)$-invariant set $T$ with $0<\delta< 6^{-N}\beta/4$ can be tiled by translates $T_ic$, $(1 \leq i \leq N(\varepsilon)$, $c\in C_i^T)$, which can be made disjoint in a way such that for all $1 \leq i \leq N(\epsilon)$ and all $c \in C_i^T$, there is some set $T_i^{(c)} \subseteq T_i$ with
\begin{itemize}
\item $|T_i^{(c)}| \geq (1-\varepsilon) |T_i|$,
\item $T_i^{(c)}$ is $(B, 4\zeta)$-invariant,
\item $|\partial_B(T_i^{(c)})| \leq |\partial_B(T_i)| + \zeta\,|T_i|$,
\item $ T_iC_i^T = \bigcup_{c \in C_i^T} T_i^{(c)}c$, where the latter union consists of pairwise disjoint sets.  
\end{itemize}
\end{enumerate}
\end{Theorem}

% \begin{Theorem} \label{thm:STP}
% Let $G$ be an unimodular amenable group and assume that $0 < \varepsilon < 1/10$. Then for each $0 < \beta < \varepsilon$ we can find F{\o}lner sets 
% \begin{eqnarray*}
% \{e\} \subseteq T_1 \subseteq T_2 \subseteq \dots \subseteq T_N
% \end{eqnarray*} 
% as well as some $\delta_0 > 0$ depending on $\beta$ such that if $T$ is some $(T_NT_N^{-1}, \delta)$-invariant set for $\delta < \delta_0$, we can find center sets $C_i(T),\, i=1\dots N$ such that the corresponding $T_i$-right translates satisfy the conditions (i)-(iii) of Definition \ref{defi:STP}, as well as
% \begin{enumerate}[(iv)]
% \item $\left| \frac{|T_iC_i(T)|}{|T|} - \varepsilon(1-\varepsilon)^{N-i} \right| < \beta$ \quad for all $1 \leq i \leq N$.
% \end{enumerate}
% \end{Theorem}
\begin{proof}
 We will only prove (b) since (a) follows as a special case. Let $\epsilon$ and $\beta$ with $0<\beta<\epsilon\leq 1/10$ and $\zeta>0$ be given and let $\cS=(S_n)$ be some nested F\o lner sequence. Choose some $0<\delta< \frac \beta 4 6^{-N(\epsilon)}$. Without loss of generality we assume that $S_n$ is $(B,\zeta^2)$-invariant for all $n\in\NN$ and that $\zeta<\delta/4$ (If $\zeta$ is not chosen to be smaller than $\delta/4$, then we can take some $\tilde{\zeta}<\delta/4$ and repeat all the steps of the proof again. Hence all claimed statements will hold for the original $\zeta$ as well). As usual, we use the notation $N:=N(\epsilon)$. 

 We start choosing the sets $T_i\in\{S_n\mid n\in \NN \}$, $i=1,\dots,N(\epsilon)$ inductively in the following way: set $T_1:=S_1$ and if $T_i=S_k$ then take $T_{i+1}\in\{S_n\mid n\geq k+1\}$ which is $(T_iT_i^{-1},\delta)$-invariant. Then obviously $\id\in T_i\subseteq T_{i+1}$ and $T_i\in\{S_n\mid n\geq i\}$ for all $i=1,\dots,N(\epsilon)-1$.

Now assume that $T$ is a $(T_NT_N^{-1}, \delta)$-invariant subset of $G$. 
We apply Lemma \ref{lemma:tilingnew} with $T=T$, $K=T_N$, $L=T_{N-1}$, $B=B$ and $\eta = \delta$ to obtain a finite set $C_N^T$ such that
$\{T_Nc\}_{c\in C_N^T}$ consists of $\epsilon$-disjoint subsets of $T$ and $D_1:=T\setminus T_NC_N^T$ is $(T_{N-1}T_{N-1}^{-1},\delta_1)$-invariant, where $\delta_1=3\delta$. Furthermore there are $(B,4\zeta)$-invariant subsets $T_N^{(c)}\subseteq T_N$ such that $\bigcup_{c\in C_N^T}T_N^{(c)} c= \bigcup_{c\in C_N^T} T_N c$ as well as
\[
 |T_N^{(c)}|\geq (1-\epsilon)|T_N|\quad\text{and}\quad |\partial_B(T_N^{(c)})| \leq |\partial_B(T_N)| + \zeta\,|T_N| \quad\text{for all }c\in C_N^T.
\]
 
Now we use Lemma \ref{lemma:tilingnew} inductively. If for some $l\in\{1,\dots,N-1\}$ the set $D_l$ is chosen as a $(T_{N-l}T_{N-l}^{-1},\delta_l)$-invariant set, we apply the Lemma with $T=D_l$, $K= T_{N-l}$, $L=T_{N-l-1}$,$B=B$, $\delta=\delta_l$, $\eta=\delta$ and $\zeta=\zeta$. Note that here it is important that $\delta_l$ is small enough, which we will ensure afterwards. This gives an appropriate set $C_{N-l}^T\in \cF(G)$ such that $D_{l+1}:=D_l\setminus T_{N-l}C_{N-l}^T$ is $(T_{N-l-1}T_{N-l-1}^{-1}, \delta_{l+1})$-invariant, where $\delta_{l+1}:=2\delta_l + \delta$. Again there are $(B,4\zeta)$-invariant sets $T_{N-1}^{(c)}\subseteq T_{N-1}$ with $\bigcup_{c\in C_{N-l}^T}T_{N-l}^{(c)} c= \bigcup_{c\in C_{N-l}^T} T_{N-l} c$ as well as
\[
|T_{N-l}^{(c)}|\geq (1-\epsilon)|T_{N-l}| 
\quad\text{and}\quad
|\partial_B(T_{N-l}^{(c)})| \leq |\partial_B(T_{N-l})| + \zeta\,|T_{N-l}|
\quad\text{for all }c\in C_{N-l}^T.
\]
This will give the additional properties listed in (b).

With $\delta_0=\delta$ we obtain $\delta_l=(2^{l+1}-1)\delta$ for all $l=1,\dots,N-1$. Therefore for arbitrary $l\in \{1,\dots,N-1\}$ we have $\delta_l\leq \delta_{N-1}=(2^N-1)\delta\leq 1/6$, which shows that all $\delta_l$ are small enough to apply Lemma \ref{lemma:tilingnew}. Furthermore the Lemma implies the inequalities
\begin{equation}\label{eqn:cover1_1}
 \epsilon- \delta_l
\leq
 \frac{|T_{N-l}C_{N-l}^T|}{|D_{l}|}
\leq
 \epsilon+ \delta_l
\end{equation}
for all $l=0,\dots,N-1$, where $D_0:=T$. We claim that for all $l=0,\dots,N-1$, there is some constant $\kappa_l$ {\em independent} of the parameters $\varepsilon, \beta$ and $\delta$ such that
\begin{eqnarray} \label{eqn:main_claim}
\left| \frac{|T_{N-l}C_{N-l}^T|}{|T|} - \varepsilon(1-\varepsilon)^{l} \right| \leq \kappa_{l} \cdot \delta_l.
\end{eqnarray}  
We proceed by induction on $l$. Note that we have treated the case $l=0$ in inequality (\ref{eqn:cover1_1}) with $\kappa_0 = 1$. Now let $l \in \NN$ and assume that \eqref{eqn:main_claim} holds for all $k = 0\dots, l-1$. By the induction hypothesis, we can sum up the resulting inequalities and arrive at
\begin{eqnarray*}
\varepsilon \cdot \sum_{k=0}^{l-1} (1-\varepsilon)^{k} - \sum_{k=0}^{l-1} \kappa_k \delta_k 
\leq \frac{|\bigcup_{k < l} T_{N-k}C_{N-k}^T|}{|T|} 
\leq \varepsilon \cdot \sum_{k=0}^{l-1} (1-\varepsilon)^{k} + \sum_{k=0}^{l-1} \kappa_k \delta_k.
\end{eqnarray*} 
Moreover, it follows from this and the definition of $D_l$ that $T\setminus D_l=\bigcup_{k<l}T_{N-k}C_{N-k}^T$ and hence
\begin{eqnarray} \label{eqn:global}
1- \varepsilon\cdot \sum_{k=0}^{l-1} (1-\varepsilon)^{k} - \sum_{k=0}^{l-1} \kappa_k  \delta_k 
\leq \frac{|D_l|}{|T|} 
\leq 1- \varepsilon\cdot \sum_{k=0}^{l-1} (1-\varepsilon)^{k} + \sum_{k=0}^{l-1} \kappa_k  \delta_k.
\end{eqnarray}
Combining this inequality (\ref{eqn:global}) with the estimate (\ref{eqn:cover1_1}), we obtain
\begin{align*}
(\varepsilon - \delta_l)\left( 1- \varepsilon \sum_{k=0}^{l-1} (1-\varepsilon)^{k} - \sum_{k=0}^{l-1} \kappa_k \delta_k \right) 
&\leq \frac{|T_{N-l}C_{N-l}^T|}{|T|} \\
&\leq  (\varepsilon + \delta_l) \left( 1- \varepsilon  \sum_{k=0}^{l-1} (1-\varepsilon)^{k} + \sum_{k=0}^{l-1} \kappa_k  \delta_k \right).
\end{align*}
It follows using $0<\epsilon<1$ that
\begin{eqnarray*}
\varepsilon\left(1-\varepsilon \, \sum_{k=0}^{l-1} (1-\varepsilon)^k \right) &-& \delta_l\left( 1 + \sum_{k=0}^{l-1} \kappa_k  \delta_k \right) - \sum_{k=0}^{l-1} \kappa_k  \delta_k \leq \frac{|T_{N-l}C_{N-l}^T|}{|T|} \\
&\leq& \varepsilon \left(1-\varepsilon \, \sum_{k=0}^{l-1} (1-\varepsilon)^k \right) + \delta_l\left( 1 + \sum_{k=0}^{l-1} \kappa_k  \delta_k \right) + \sum_{k=0}^{l-1} \kappa_k \delta_k,
\end{eqnarray*}
which implies with $\delta_k \leq \delta_l \leq 1$ and $1-\varepsilon \, \sum_{k=0}^{l-1} (1-\varepsilon)^k = (1-\epsilon)^l$ that
\begin{align*}
\left| \frac{|T_{N-l}C_{N-l}|}{|T|} - \varepsilon(1-\varepsilon)^l \right| 
\leq \delta_l\left( 1 + \sum_{k=0}^{l-1} \kappa_k  \delta_k \right) + \sum_{k=0}^{l-1} \kappa_k \delta_k 
\leq \left(1 + 2 \sum_{k=0}^{l-1} \kappa_k \right) \delta_l.
\end{align*}
This shows the claim (\ref{eqn:main_claim}) with $\kappa_l:= 1 + 2 \sum_{k=0}^{l-1} \kappa_k$. Since $\kappa_0 = 1$ we can compute the $\kappa_l$ recursively, namely $\kappa_l=3^{l}$ for all $l\geq 0$.
In particular, we have for all $1 \leq i \leq N$ that
\begin{eqnarray*}
\left| \frac{|T_{i}C_{i}|}{|T|} - \varepsilon(1-\varepsilon)^{N-i} \right| 
\leq 3^{N} (2^{N+1}-1)\delta 
\leq 2\cdot 6^{N}\delta
< \beta,
\end{eqnarray*}
by the choice of $\delta$. This proves (iv) of Definition \ref{defi:STP}. Properties (i), (ii) and (iii) follow by the construction of the sets $C_k^T$, $k=1,\dots,N$. The additional properties concerning the disjoint sets $T_i^{(c)}$ can  immediately be deduced from Lemma \ref{lemma:tilingnew}.
\end{proof}

%% file: Zerlegungen2_neu.tex
For countable amenable groups, we now verify the existence of a {\em family} of coverings of the type as in Theorem \ref{thm:STP} which possesses a uniform covering property on average. To be precise, it will be shown that for each element $u$ of the covered set $T$, the probability for this element being a center set of a covering of the family is equal to some number which only depends on $\varepsilon$ as well as on the tiling set $T_i$. \\
To do so, we will need the concept of the so-called {\em uniform special tiling property}.

\begin{Definition} \label{defi:USTP}
Let $G$ be a countable amenable group. We say that $G$ satisfies the {\em uniform special tiling property} (USTP) if for all strong F{\o}lner sequences $(U_n)$ in $G$, the following statements hold true. 
\begin{itemize}
\item For given $0<\varepsilon \leq 1/10$, $N:=N(\varepsilon):= \lceil \log(\varepsilon)/\log(1-\varepsilon) \rceil$, for arbitrary $0 <\beta < 2^{-N}\varepsilon$, and for any nested F{\o}lner sequence $(S_n)$, the group $G$ satisfies the special tiling property according to Definition \ref{defi:STP} with tiling sets
\begin{eqnarray*}
\{\operatorname{id}\} \subseteq T_1 \subseteq T_2 \subseteq \dots \subseteq T_{N(\varepsilon)},
\end{eqnarray*}
where $T_i \in \{S_n \,|\, n \geq i\}$ for $1 \leq i \leq N(\varepsilon)$.
\item For fixed numbers $\varepsilon$ and $\beta$, there is a number $K \in \NN$ and a finite set $Q\subseteq G$ depending on $\varepsilon, \beta$ and on the basis sets $T_i$ such that for each $k \geq K$, we find a finite set $\Lambda_k \subseteq G$, along with a {\em family} $\{C_i^{\lambda}(U_k) \, |\, \lambda \in \Lambda_k,\, 1 \leq i \leq N\}$ of finite center sets for the $T_i$ such that the set $U_k$ is $\varepsilon$-quasi tiled with the properties (i)-(iii) of Definition \ref{defi:STP} and 
\begin{enumerate}[(I)]
\item $\frac{|\bigcup_{i=1}^N T_iC_i^{\lambda}(U_k)|}{|U_k|} \geq 1 - 4\varepsilon$ for all $\lambda \in \Lambda_k$. 
\item $\left| |\Lambda_k|^{-1} \, \sum_{\lambda \in \Lambda_k} \one_{C_i^{\lambda}(U_k)}(u) - \frac{\varepsilon(1-\varepsilon)^{N-i}}{|T_i|} \right| < 3 \frac{\beta}{|T_i|} + \varepsilon\cdot \gamma_i$ for all $1 \leq i \leq N$ and all $u \in U_k\setminus \partial_Q(U_k)$, where the $\gamma_i>0$ satisfy $\sum_{i=1}^N \gamma_i |T_i| \leq 2$.
\end{enumerate}
\end{itemize}
\end{Definition}

\begin{Theorem}[Uniform Decompositions] \label{thm:USTP}
Each countable, amenable group satisfies the uniform special tiling property.  
\end{Theorem}

\begin{proof}
Let $0 < \varepsilon \leq 1/10$ and $0 < \beta < 2^{-N}\varepsilon$ be fixed and let $(U_n)$ be a strong F{\o}lner sequence. Further, assume that $0 < \delta_0 < 6^{-N}\beta/16$. 

Note that by Theorem \ref{thm:STP}, $G$ is $\varepsilon$-quasi tiled by a finite sequence 
\[
\{\operatorname{id}\} \subseteq T_1 \dots \subseteq T_N
\]
of compact basis sets taken from a nested F{\o}lner sequence $(S_n)$, where as usual, $N:=N(\varepsilon):= \lceil \log(\varepsilon)/\log (1-\varepsilon)\rceil$. 

Let $0 <\delta < 1/100$. At various steps of the proof, we will have to make this parameter smaller. For the sake of the reader, we prefer doing this in a successive manner instead of imposing many technical conditions on $\delta$ right now. This is possible since the restrictions will only depend on $\varepsilon, \beta$ and the basis sets $T_i$, $1 \leq i \leq N$, but not on things developed in the proof. One can then think of starting the proof all over again with a new condition on the parameter $\delta$. 
We proceed in nine steps.

\begin{enumerate}[(1)]
\item We let $M:=\lceil \log(\delta)/\log(1-\delta) \rceil$ and following Theorem \ref{thm:STP}, part (b), we find $(T_NT_N^{-1}, \delta_0^2)$-invariant sets $\overline{T}_l \supseteq T_N$, $1 \leq l \leq M$, taken from a nested F{\o}lner sequence $(S_n)$, such that the $\overline{T}_l$ $\delta$-quasi tile the group $G$. Then we can find some integer $K \in \NN$ such that for each $k \geq K$, the set $T:=U_k$ is $(\overline{T}_l\overline{T}_l^{-1},2^{-l} \delta)$-invariant for all $1 \leq l \leq M$. Since $\delta$ will depend on $\varepsilon, \beta$ and the basis sets $T_i$, so does the integer number $K$. Further, we choose $\hat{T}$ to be a $(TT^{-1}, \delta)$-invariant compact set inheriting all the mentioned invariance properties of $T$. Using Theorem \ref{thm:STP}, we can also make sure that $\hat{T}$ has the special tiling property with respect to $(\{\overline{T}_l\}_{l=1}^M, (S_n), \delta, \beta_1)$, where $0 < \beta_1 < 2^{-M}\delta$ (For 
instance, take $\hat{T}:=U_{\tilde{K}}$ for $\tilde{K} \in \NN$ large enough). 

For $\overline{A}:= \{a \in \hat{T} \,|\, TT^{-1} a \subseteq \hat{T}\}$, it follows from Lemma \ref{lemma:ow1} that $|\overline{A}| \geq (1-\delta) |\hat{T}|$. 
Since $\overline{A} \subseteq s\,A$ with $A:= \{g \in G\,|\, Tg \subseteq \hat{T}\}$ for every $s \in T$, we also have by unimodularity $|A| \geq (1-\delta)|\hat{T}|$.

\item 
Since $\hat{T}$ has the special tiling property with respect to $(\{\overline{T}_l\}_{l=1}^M, (S_n), \delta, \beta_1)$ ($0 < \beta_1 < 2^{-M}\delta$), we can we fix an $\delta$-tiling of $\hat{T}$ as in Theorem \ref{thm:STP}, part (b), where we can make the $\overline{T}_l$-translates in this $\delta$-tiling actually disjoint such that the resulting disjoint translates $\overline{T}_l'(c)c \subseteq \overline{T}_lc$ are still $(T_NT_N^{-1}, 4\delta_0)$-invariant. We note that these disjoint translates $(1-2\delta)$-cover the set $\hat{T}$, i.e. 
\begin{align}
\frac{\left| B(\delta) \right|}{|\hat{T}|} =  \frac{\sum_{l=1}^M \sum_{c \in \overline{C}_l} |\overline{T}'_l(c)c| }{|\hat{T}|} 
\geq& (1-2\delta), 
\quad\text{where}\quad
B(\delta):=\bigcup_{l=1}^M \bigcup_{c \in \overline{C}_l} \overline{T}'_l(c)c.
\label{eqn:inter_cover}
\end{align}

\item 
We have already mentioned that all the sets $\overline{T}_l'(c)$ must still be $(T_NT_N^{-1}, 4\delta_0)$-invariant (recall that $4\delta_0 < 6^{-N}\beta/4$).
Therefore, by Theorem \ref{thm:STP}, we can fix in each translate $\overline{T}_l'(c)c$, $(c \in \overline{C}_l)$ an $\varepsilon$-quasi tiling of $(T_i)_{i=1}^N$ with center sets $C_i(\overline{T}_l'(c)c)$ and 
\begin{eqnarray} \label{eqn:iv}
\left| \frac{|T_iC_i(\overline{T}_l'(c)c)|}{|\overline{T}_l'(c)c|} - \eta_i(\delta) \right| < \beta
\end{eqnarray}
for $1 \leq i \leq N$. Further, we put
\begin{eqnarray*}
\hat{C}_i := \bigcup_{l=1}^M \bigcup_{c \in  \overline{C}_l} C_i(\overline{T}_l'(c)c)
\end{eqnarray*}
for $1 \leq i \leq N$ and note that the $\hat{C}_i$ can be considered as center sets for the F{\o}lner elements $T_i$ such that the $\{T_ic\}_{c \in \hat{C}_i}$ are $\varepsilon$-disjoint and such that for $1 \leq i < j \leq N$, the sets $T_i\hat{C}_i$ and $T_j\hat{C}_j$ are disjoint. Using the fact that $\beta < 2^{-N}\varepsilon$ we obtain for each $l\in\{1,\dots,M\}$:
\begin{align*}
 \left| \bigcup_{i=1}^N T_i C_i(\overline T'_l(c)c) \right|
\geq (1-2\epsilon)\left| \overline T'_l(c) \right|.
\end{align*}
% 
% \begin{eqnarray*}
% (1-2\delta)\,|\hat{T}| &\stackrel{(\ref{eqn:inter_cover})}{\leq}& \left| \bigcup_{l=1}^M \bigcup_{c \in \overline{C}_l} \overline{T}'_l(c)c \right| \, =\, \sum_{l=1}^M \sum_{c \in \overline{C}_l} |\overline{T}'_l(c)c| \\
% &\stackrel{\beta < 2^{-N}\varepsilon}{\leq}& (1-2\varepsilon)^{-1} \, \sum_{l=1}^M \sum_{c \in \overline{C}_l}\left| \bigcup_{i=1}^N T_iC_i(\overline{T}_l'(c)c) \right| \\
% &=& (1-2 \varepsilon)^{-1} \, \left| \bigcup_{i=1}^N T_i\hat{C}_i  \right|
% \end{eqnarray*}
% such that 
% \begin{eqnarray} \label{eqn:cover_next}
% \left| \bigcup_{i=1}^N T_i \hat{C}_i \right| \geq (1-2\delta - 2 \varepsilon)\, |\hat{T}|.
% \end{eqnarray}

\item We now would like to determine the portion of $\hat{T}$ that is covered by each set $T_i\hat{C}_i$. We will see that up to some small error ($2\beta$), this will be $\eta_i(\epsilon)=\varepsilon(1-\varepsilon)^{N-i}$.

Let $i\in\{1,\dots,N\}$ be given. Using the disjointness of the $\overline{T}'_l(c)c$ for all $c \in \overline{C}_l$ and all $1 \leq l \leq M$, inequality \eqref{eqn:inter_cover} and \eqref{eqn:iv} we obtain
\begin{align*}
 |T_i\hat{C}_i| &= \sum_{l=1}^M \sum_{c \in \overline{C}_l} |T_iC_i(\overline{T}_l'(c)c)|   \\
&\geq  \sum_{l=1}^M \sum_{c \in \overline{C}_l} |\overline{T}_l'(c)|(\eta_i(\epsilon)-\beta) 
\geq (1-2\delta)(\eta_i(\epsilon)-\beta)|\hat T|
\geq (\eta_i(\epsilon)-2\beta)|\hat T|,
%
%\frac{|\overline{T}_l'(c)|}{|\hat{T}|} \, \frac{|T_iC_i(\overline{T}_l'(c)c)|}{|\overline{T}_l'(c)|}.
\end{align*}
where the last step is true for sufficiently small $\delta$. Note that this is the first of the announced conditions on the smallness of $\delta$. Let us estimate in the other direction. Estimates \eqref{eqn:iv} and \eqref{eqn:inter_cover} lead to
\begin{align*}
  |T_i\hat{C}_i| \leq \sum_{l=1}^M \sum_{c \in \overline{C}_l} |\overline{T}_l'(c)|(\eta_i(\epsilon)+\beta) 
= (\eta_i(\epsilon)+\beta) |B(\delta)|
\leq  (\eta_i(\epsilon)+\beta) |\hat T|
\end{align*}
The above bounds yield for all $i\in\{1,\dots, N\}$
\begin{align} \label{eqn:small3}
\left| \frac{|T_i\hat{C}_i|}{|\hat{T}|} - \eta_i(\epsilon) \right| < 2 \beta.
\end{align}

\item We will see below that for each $1 \leq i \leq N$, the ratio $\gamma_i:=|\hat{C}_i|/|\hat{T}|$ plays an essential role in our argumentation. Hence, we compare this expression with the ratio $\eta_i(\epsilon)/|T_i|$.
By exploiting $\varepsilon$-disjointness of the $T_i$-translates, as well as (\ref{eqn:small3}), we obtain for each $i\in\{1,\dots,N\}$:
\begin{align}\label{eqn:key1}
 \left|\gamma_i-\frac{\eta_i(\epsilon)}{|T_i|} \right|
\leq  \frac{1}{|T_i|} \, \left| \frac{|\hat{C}_i|\,|T_i|}{|\hat{T}|} - \frac{|T_i\hat{C}_i|}{|\hat{T}|} \right| + \frac{1}{|T_i|}\, \left| \frac{|T_i\hat{C}_i|}{|\hat{T}|} - \eta_i(\epsilon)\right|
\leq 
  \varepsilon \, \frac{|\hat{C}_i|}{|\hat{T}|} + \frac{2\beta}{|T_i|} \, = \, \gamma_i \,\varepsilon + \frac{2\beta}{|T_i|}.
\end{align}

% \begin{eqnarray} 
% & &\left| \frac{|\hat{C}_i|}{|\hat{T}|} - \frac{\varepsilon(1-\varepsilon)^{N-i}}{|T_i|} \right| \, =\, \frac{1}{|T_i|} \, \left| \frac{|\hat{C}_i|\,|T_i|}{|\hat{T}|} - \varepsilon(1-\varepsilon)^{N-i}  \right| \nonumber \\
% &\leq& \frac{1}{|T_i|} \, \left| \frac{|\hat{C}_i|\,|T_i|}{|\hat{T}|} - \frac{|T_i\hat{C}_i|}{|\hat{T}|} \right| + \frac{1}{|T_i|}\, \left| \frac{|T_i\hat{C}_i|}{|\hat{T}|} - \varepsilon(1-\varepsilon)^{N-i}\right| \nonumber \\
% &\leq& \frac{1}{|T_i|} \cdot \varepsilon \cdot \frac{|\hat{C}_i|\,|T_i|}{|\hat{T}|} + \frac{2\beta}{|T_i|} \nonumber \\
% &=& \varepsilon \, \frac{|\hat{C}_i|}{|\hat{T}|} + \frac{2\beta}{|T_i|} \, = \, \gamma_i \,\varepsilon + \frac{2\beta}{|T_i|},
% \end{eqnarray}  
%where for $1 \leq i \leq N$, we put $\gamma_i:= |\hat{C}_i|/|\hat{T}|$. 
Using the $\varepsilon$-disjointness of the sets $T_i$ (and the rough bound $\varepsilon < 1/2$) it is easy to show that $\sum_{i=1}^N \gamma_i \, |T_i| \leq 2$ holds true.

\item In the next step of the proof, it will be shown that most of the $T$-translates contained in $\hat{T}$ will be $(1-3\varepsilon)$-covered by the fixed pattern $\cup_{i=1}^N T_i \hat{C}_i$. Here, we will have to impose a second restriction on $\delta$.
We recall from step (1) that we chose the set $A$ as the collection of elements $a \in G$ such that the translate $Ta$ lies entirely in $\hat{T}$.
%  and that
% \begin{eqnarray} \label{eqn:sizeA1}
% |A| \geq (1-\delta)|\hat{T}|.
% \end{eqnarray}
For each $a \in A$, we set
\begin{align*}
X(a):= \frac{\left|Ta \cap \, \left( \hat{T} \setminus B(\delta) \right) \right|}{|Ta|} 
= \frac{\left|Ta \setminus B(\delta) \right|}{|T|}
\end{align*}
and treat $X$ as a uniformly distributed (w.r.t. the Haar measure) random variable on the set~$A$. Thus, it follows from the Chebyshev Inequality that
\begin{align*}
\sqrt{\delta}\left|\{ a \in A\,|\, X(a) > \sqrt{\delta} \}\right|
\leq 
\sum_{a\in A} X(a)
=
\frac{1}{|T|}
 \sum_{a\in A}\sum_{g\in \hat T\setminus B(\delta)}\1_{Ta}(g).
\end{align*}
Using \eqref{eqn:inter_cover} we continue estimating 
\begin{align*}
\frac1{| T|}\sum_{a\in A}\sum_{g\in \hat T\setminus B(\delta)}\1_{Ta}(g)
=  \frac1{| T|}\sum_{g\in \hat T\setminus B(\delta)}\sum_{a\in A}\1_{Ta}(g) 
\leq |\hat T\setminus B(\delta)| \leq 2\delta |\hat T|,
\end{align*}
Applying $|A| \geq (1-\delta)|\hat{T}|$ (see step (1)) and $\delta\leq 1/2$, this yields
\begin{align*}
 \left|\{ a \in A\,|\, X(a) > \sqrt{\delta} \}\right| 
\leq 2\sqrt{\delta}|\hat T| 
\leq \frac{2\sqrt{\delta}|A| }{1-\delta}\leq 4\sqrt{\delta}|A|.
\end{align*}
or equivalently
\begin{align}\label{def:Lambda}
 |\Lambda|\geq (1-4\sqrt{\delta})|A|,\quad\text{where}\quad \Lambda:=\{a\in A\mid X(a)\leq \sqrt{\delta}\}.
\end{align}
%%%%%%%%%%%%%%%%%%%% kopiert aus der Diss ->
We have seen that, up to a portion of $4\sqrt{\delta}$, the translates of $T$ which lie entirely in $\hat T$ are $(1-\sqrt{\delta})$-covered by our tiling. However, as for each $a$ we are interested in a tiling of $Ta$ with \emph{subsets} of $Ta$, we need to delete elements of this covering, which have a non-empty intersection with $G\setminus Ta$. Define for $l\in\{1,\dots,M\}$ and $a\in A$ the sets
\begin{align*}
  \partial(a,l)&:=\Bigl\{c\in \overline C_l\mid \overline T'_l(c)c\cap Ta\neq\emptyset,\, \overline T'_l(c)c\cap (G \setminus Ta)\neq\emptyset \Bigr\},\text{ and}\\
  I(a,l)&:=\Bigl\{c\in \overline C_l\mid \overline T'_l(c)c \subseteq Ta \Bigr\}.
\end{align*}
Then we have for $\tilde c\in \partial(a,l)$ that $\overline T'_l(\tilde c)\tilde c\subseteq \partial_{\overline T_l\overline T_l^{-1}}(Ta)$, which gives
\begin{align}\label{eq:ustp:boundsets}
 \bigcup_{l=1}^{M}\bigcup_{c\in\partial(a,l)}\overline T'_l(c)c
\subseteq \bigcup_{l=1}^{M}\partial_{\overline T_l\overline T_l^{-1}}(Ta).
\end{align}
Hence, using the assumed invariance properties of $T$ in step (1), this yields
 \begin{align}\label{eq:ustp:boundary}
  \frac1{\abs T}\biggl|\bigcup_{l=1}^{M}\bigcup_{c\in\partial(a,l)}\overline T'_l(c)c \biggr|
 \leq 
 \frac1{\abs T}\sum_{l=1}^{M}\bigabs{\partial_{\overline T_l\overline T_l^{-1}}(T)}
 \leq 
 \sum_{l=1}^{M}\frac{\delta}{2^l}
 \leq  \delta.
 \end{align}
Therefore, by the definition of $\Lambda$ we have for each $a\in \Lambda$ the estimate
\begin{align}\label{eq:int_cov2}
 \biggl|\bigcup_{l=1}^{M}\bigcup_{c\in I(a,l)}\overline T'_l(c)c \biggr|
\geq
 \abs{Ta\cap B(\delta)} - \biggl|\bigcup_{l=1}^{M}\bigcup_{c\in\partial(a,l)}\overline T'_l(c)c \biggr|
\geq
 (1-\sqrt{\delta}-\delta)\abs{T}.
\end{align}
Now let us estimate the part of $Ta$, which is covered by translates $T_i$, $i=1,\dots,N$, which lie completely in $Ta$. To this end, we set for each $a\in A$
\[
\tilde C_i(a):=\bigcup_{l=1}^{M}\bigcup_{c\in I(a,l)} C_i(\overline T'_l(c)c),\quad\text{and}\quad
 D(a):=\bigcup_{i=1}^{N} T_i  \tilde C_i(a)\subseteq Ta.
\]
Then, using the disjointness of $\overline T'_l(c)c$, $l\in\{1,\dots,M\}$, $c\in \overline C_l$ and the estimate \eqref{eqn:inter_cover} we have for $a\in\Lambda$:
\begin{align*}
 \abs{D(a)}=\sum_{l=1}^{M}\sum_{c\in I(a,l)}\biggabs{\bigcup_{i=1}^{N}T_i  C_i(\overline T'_l(c)c)}
\geq (1-2\epsilon)\sum_{l=1}^{M}\sum_{c\in I(a,l)}\abs{\overline T'_l(c)}. 
\end{align*}
Thus, imposing another condition on the size of $\delta$ and using \eqref{eq:int_cov2}, we obtain for all $a\in\Lambda$:
\begin{align}\label{eq:prop(iv)1}
\abs{D(a)}
 \geq (1-2\epsilon)(1-\sqrt{\delta}-\delta)\abs{T}
\geq (1-3\epsilon)\abs T.
\end{align}

\item Next, we define the desired family of center sets and verify properties (i) - (iii) Definition \ref{defi:STP} as well as property (I) of Definition \ref{defi:USTP}. In the previous step we already defined the set $\Lambda$. Recall that $\Lambda$ depends on $T=U_k$, which justifies the notion $\Lambda_k:=\Lambda$. 
For each $k\in \NN$, $\lambda\in\Lambda_k$ and $i\in\{1,\dots,N\}$ we set
\begin{align*}
 C_i^\lambda:=C_i^\lambda(T) := \tilde C_i(\lambda)\lambda^{-1}.
\end{align*}
Then we have 
\begin{align}\label{eq:Dlambda}
 \bigcup_{i=1}^{N} T_i C_i^\lambda = D(\lambda)\lambda^{-1}\subseteq T
\end{align}
which shows (i) of Definition \ref{defi:STP}. Properties (ii) and (iii) are fulfilled by construction. The bound in (I) of Definition \ref{defi:USTP} follows from \eqref{eq:Dlambda} and \eqref{eq:prop(iv)1}.

\item 
This step is devoted to prepare the proof of the uniform covering principle, i.e.\ property (II) of Definition \ref{defi:USTP}.
To this end we define $Q:= T_M T_M^{-1}$, which is independent of $k$ the index of $T=U_k$. The choice of $Q$ implies 
\[
  \partial_{Q} (T) \supseteq \bigcup_{l=1}^M\partial_{\overline T_l\overline T_l^{-1}}(T).
\]
 We will show that for each $i\in\{1,\dots,N\}$ and $u\in T\setminus \partial_Q(T)$ we have
\begin{align}\label{eq:aim8}
 \biggl|\frac{1}{\abs\Lambda}\sum_{\lambda\in \Lambda}\1_{C_i^\lambda}(u)-\gamma_i\biggr|\leq \frac{\beta}{\abs{T_i}}.
\end{align}
 To this end, we use $\tilde C_i(\lambda)\subseteq \hat C_i$ to obtain
\begin{align}\label{eq:estsum}
 \sum_{\lambda\in\Lambda}\1_{C_i^\lambda}(u)
=\sum_{\lambda\in\Lambda}\1_{u^{-1}\tilde C_i(\lambda)}(\lambda)
\leq \sum_{\lambda\in\Lambda}\1_{u^{-1}\hat C_i}(\lambda)
= \abs{u\Lambda\cap \hat C_i}
\leq \abs{\hat C_i}.
\end{align}
Next, use \eqref{def:Lambda} and $|A|\geq (1-\delta)\abs{\hat T}$ calculate
\begin{align*}
 \frac{1}{\abs\Lambda}\sum_{\lambda\in\Lambda}\1_{C_i^\lambda}(u)
\leq \frac{\abs{\hat C_i}}{(1-4\sqrt{\delta})\abs A}
\leq \frac{\abs{\hat C_i}}{(1-5\sqrt{\delta})\abs{\hat T}}
=\frac{\gamma_i}{(1-5\sqrt{\delta})}
\end{align*}
As $\gamma_i\leq 1$ we obtain
\begin{align}\label{eq:aim81}
 \frac{1}{\abs\Lambda}\sum_{\lambda\in\Lambda}\1_{C_i^\lambda}(u)-\gamma_i
\leq \frac{1}{1-5\sqrt{\delta}}-1
\leq \frac{\beta}{|T_i|},
\end{align}
where the last inequality holds only for $\delta$ small enough, which is the third assumption on the size of $\delta$.
Let us verify the other bound. To this end, we claim that for each $\lambda\in \Lambda$ and $i\in\{1,\dots,N\}$ we have
\begin{align}\label{eq:setclaim}
 \hat C_i\lambda^{-1} \cap (T\setminus \partial_Q(T)) \subseteq C_i^\lambda
\end{align}
To see this, let $x\in \hat C_i\lambda^{-1} \cap (T\setminus \partial_Q(T)) $ be given. Then we find $l\in\{1,\dots,M\}$ and $c\in \overline C_l$ with $x \in C_i(l,c) \lambda^{-1}$. If $c\in I(\lambda,l)$ we are done, since then $x\in \tilde C_i(\lambda)\lambda^{-1}= C_i^\lambda$. Now, assume that $c\notin I(\lambda,l)$. This implies that $\overline T'_l(c)c\cap G\setminus T\lambda \neq \emptyset$. But as $x\lambda\in  C_i(l,c)$ and $x\in T$ we have 
\[
 x\lambda \in T_i x\lambda \subseteq \overline T'_l(c)c\quad\text{and}\quad x\lambda\in T\lambda
\]
This implies that $x\lambda\in \partial_{\overline T'_l(c)\overline T'_l(c)}(T\lambda)$ and therefore $x\in \partial_{Q}(T)$, which is a contradiction.
 Thus, $x\in I(\lambda,l)$ and \eqref{eq:setclaim} is proven.

Use \eqref{eq:setclaim} and $u\in T\setminus \partial_Q(T)$ to obtain
\[
 \1_{C_i^\lambda}(u) 
\geq  \1_{\hat C_i\lambda^{-1} \cap (T\setminus \partial_Q(T))}(u) 
= \1_{\hat C_i\lambda^{-1} }(u) 
= \1_{u^{-1}\hat C_i}(\lambda)
\]
which implies together with \eqref{eq:estsum} that
\[
 \sum_{\lambda\in \Lambda}\1_{C_i^\lambda}(u)=\abs{u\Lambda \cap \hat C_i}.
\]
Use $\hat C_i\subseteq \hat T$ and $u\Lambda \subseteq \hat T$ to calculate
\begin{align*}
 \gamma_i - \frac{1}{\abs\Lambda}\sum_{\lambda\in \Lambda}\1_{C_i^\lambda}(u)
\leq \frac{|\hat C_i|}{|\hat T|}- \frac{|u\Lambda \cap \hat C_i|}{|\hat T|}
\leq \frac{|\hat T\setminus u \Lambda|}{|\hat T|}
=1- \frac{|\Lambda|}{|\hat T|}
\end{align*}
Now, Inequalities \eqref{def:Lambda} and $|A|\geq (1-\delta)|\hat T|$ yield
\[
 \gamma_i - \frac{1}{\abs\Lambda}\sum_{\lambda\in \Lambda}\1_{C_i^\lambda}(u)
 \leq 
1- \frac{(1-4\sqrt{\delta})(1-\delta)|\hat T|}{|\hat T|}
\leq 5\sqrt{\delta} \leq \frac{\beta}{|T_i|},
\]
where the last inequality holds only for $\delta$ small enough, which is the fourth and last condition on $\delta$.
This, together with \eqref{eq:aim81} shows \eqref{eq:aim8}.

\item Finally, we are able to prove claim (II) of the statement. We combine the Inequalities (\ref{eqn:key1}) and (\ref{eq:aim8}) and yield by means of the triangle inequality for each $u\in T\setminus \partial_Q(T)$:
\begin{align*}
\left| \frac{1}{|\Lambda|} \sum_{\lambda \in \Lambda} \one_{C_i^{\lambda}}(u) - \frac{\eta_i(\epsilon)}{|T_i|} \right| 
\leq \left| \frac{1}{|\Lambda|} \sum_{\lambda \in \Lambda} \one_{C_i^{\lambda}}(u) - \gamma_i \right| + \left| \gamma_i - \frac{\eta_i(\epsilon)}{|T_i|} \right| 
\leq \gamma_i  \varepsilon + \frac{3\beta}{|T_i|} 
\end{align*}
for all $1 \leq i \leq N$ and all $u \in T$, where $\sum_{i=1}^N \gamma_i |T_i| \leq 2$. 
\end{enumerate}
So, we have finally finished the proof of the theorem. 
\end{proof}

%% file: ET_IDS.tex
\section{An almost additive Ergodic Theorem}

Let $G$ be a discrete, countable amenable group and denote by $\mathcal{F}(G)$ the set of compact (finite) subsets in $G.$
For the following elaborations, we refer to the setting in \cite{LenzSV-10}. 
In their work, the authors prove a Banach space valued ergodic theorem for functions $F:\mathcal{F}(G) \rightarrow (X,\|\cdot\|)$ 
with certain boundedness and additivity conditions (Theorem 3.1). 
However, it has been necessary to impose strong restrictions on the group $G$ under consideration. 
More precisely, one has to assume that $G$ possesses a F{\o}lner sequence $\{Q_n\}$ 
with the property that each element $Q_n$ is a monotile of $G$, 
i.e. for each $n \in \NN$ there is a grid $G_n \subseteq G$ 
such that $\cup_{g \in G_n} Q_ng$ is a disjoint tiling of the group. 
In addition to that it turned out that this assumption is even not sufficient, 
but one also has to require the grid sets to be symmetric (cf.\@ \cite{LenzSV-11}).           
Using the $\varepsilon$-quasi tilings of the previous section, 
we can drop all these restrictions to prove a general Banach valued ergodic theorem (cf.\@ Theorem \ref{thm:ET}).

Assume that we are given a finite set $\mathcal{A}$ of colors. Then each map $\cC: G \rightarrow \mathcal{A}$ defines a coloring of the group. 
If $\cF(G)$ denotes the set of all finite subsets of $G$, then we call a map
\begin{eqnarray*}
P:D(P) \rightarrow \mathcal{A}
\end{eqnarray*} 
a {\em pattern} with $D(P) \in \cF(G)$ as the {\em domain} of $P$. 
The set of all patterns is denoted by $\cP$ and for a fixed $Q \in \cF(G)$ 
the subset of $\cP$ which contains only the patterns with domain $Q$ is denoted by $\cP(Q)$. 
Given a set $Q \subseteq D(P)$ and an element $x \in G$, 
we furthermore define a {\em restriction of a pattern} by
\begin{eqnarray*}
P_{|Q}:Q \rightarrow \mathcal{A}: g \mapsto P_{|Q}(g) = P(g),
\end{eqnarray*}
as well as a  {\em translation of a pattern} by
\begin{eqnarray*}
Px: D(P)x \rightarrow \mathcal{A}: yx \mapsto P(y).
\end{eqnarray*}
Translations and restrictions of colorings are defined equivalently.
Two patterns are called {\em equivalent} if one is the translation of the other. 
The equivalence class of a pattern $P$ is then denoted by $\tilde{P}$. We write $\tilde{\cP}$ for the induced set of equivalence classes in $\cP$. For two patterns $P$ and $P^{'}$,
 the {\em number of occurrences} of the pattern $P$ in $P^{'}$ is denoted by
\begin{eqnarray*}
\#_P(P^{'}):= \left|\{x \in G\,|\, D(P)x \subseteq D(P^{'}), \, P^{'}_{|D(P)x}=Px\}\right|.
\end{eqnarray*} 
Counting occurrences of patterns along a F{\o}lner sequence $(U_j)_{j \in \NN}$ leads to the definition of {\em frequencies}. 
If for a pattern $P$ and a F{\o}lner sequence $(U_j)_{j \in \NN}$ the limit
\begin{eqnarray*}
\nu_P:= \lim_{j \rightarrow\infty} \frac{\#_P(\cC_{|U_j})}{|U_j|}
\end{eqnarray*} 
exists, we call $\nu_P$ the {\em frequency of $P$ in the coloring $\cC$ along $(U_j)_{j \in \NN}$}.

\begin{Definition}[boundary term] \label{defi:BT}
A function $b:\cF(G) \rightarrow [0,\infty)$ is called a {\em boundary term} if
\begin{enumerate}[(i)]
\item $b(Q) = b(Qx)$ for all $x \in G$ and all $Q \in \cF(G)$.
\item $\lim_{j \rightarrow \infty} \frac{b(U_j)}{|U_j|} = 0$ for any F{\o}lner sequence $(U_j)_{j \in \NN}$.
\item there exists $D > 0$ with $b(Q) \leq D|Q|$ for all $Q \in \cF(G)$.
\item one has for all $Q,Q'\in \cF(G)$
\[
 b(Q\cap Q')\leq b(Q)+b(Q'),\quad 
 b(Q\cup Q')\leq b(Q)+b(Q'),\quad
 b(Q\setminus Q') \leq b(Q)+b(Q').
\]
\end{enumerate}
For a pattern $P$ we define $b(P):=b(D(P))$. Note that due to property (i), the value $b(P)$ depends only on the equivalence class of a pattern. 
\end{Definition}

\begin{Definition}\label{defi:AA2}
 Let a Banach-space $(X,\Vert \cdot \Vert)$, a finite set $\cA$, a coloring $\cC: G\to \cA$ and a function $F:\cF(G)\to X$ be given.
\begin{itemize}
 \item $F$ is called \emph{almost additive} if there is boundary term $b:\cF(G)\to [0,\infty)$ such that for any disjoint $Q_1,\dots,Q_k\in \cF(G)$ one has
\[
 \left\Vert F(Q) - \sum_{i=1}^k F(Q_i) \right\Vert  \leq \sum_{i=1}^k b(Q_i),
\]
 where $Q=\bigcup_{i=1}^k Q_i$.
 \item $F$ is called \emph{$\cC$-invariant} if for any $Q,U\in \cF(G)$ the equivalence of the patterns $\cC_{|Q}$ and $\cC_{|U}$ implies $F(Q)=F(U)$.
\end{itemize}
\end{Definition}
Given an almost additive and $\cC$-invariant function $F:\cF(G)\to X$ we define $\tilde F:{\cP} \to X$ by
\begin{align}\label{eqn:tildeF}
 \tilde F(P) =\begin{cases} F(Q)&\text{ if }Q \in \cF(G)\text{ such that } \tilde{\cC_{|Q}}=\tilde P\\ 0 &\text{ else.}  \end{cases}
\end{align}
This is well defined by the $\cC$-invariance of $F$. The next result gives properties of the functions $F$ and $\tilde F$.

\begin{Lemma}\label{la:epsaad}
 Let a Banach space $(X,\Vert \cdot \Vert)$,a finite set $\cA$ and a coloring $\cC: G\to \cA$ be given. Furthermore let $F:\cF(G)\to X$ be  $\cC$-invariant and almost additive with boundary term $b$.
\begin{itemize}
 \item[(i)] Then there exists a constant $C>0$ such that
\[
   \Vert F(Q)\Vert\leq C|Q| 
\quad\text{and}\quad\Vert 
  \tilde F(P)\Vert\leq C|D(P)|,
\]
for all $Q\in\cF(G)$ and $P\in\cP$, where $\tilde F$ is given by \eqref{eqn:tildeF}.
 \item[(ii)] If furthermore $0<\epsilon<1/2$ is given and $Q_i$, $i=1,\dots,k$ are $\epsilon$-disjoint sets and $Q=\bigcup_{i=1}^k Q_i$, then 
\[
 \left\Vert F(Q)-\sum_{i=1}^k F(Q_i)\right\Vert \leq (3C+9D)\epsilon |Q| +  3\sum_{i=1}^k b(Q_i),  
\]
where $C$ is the constant from (i) and $D$ is given by property (iii) of the boundary term.
\end{itemize}
\end{Lemma}

\begin{proof}
These are straightforward consequences of the properties of $F$ and $b$. 
\end{proof}

For the sake of clarity, we summarize our major assumptions. 

% In order to refer on the assumptions of the main theorem later let us fix them before stating the result

\begin{Assumption}\label{ass:first}
Denote by $G$ a countable amenable group, 
let ${\mathcal A}$ be a finite set and consider a map ${\mathcal C}: G\rightarrow {\mathcal A}$, which will be called a coloring. 
Also, we assume that we are given a F{\o}lner sequence $(U_j)_{j\in \NN}$ in $G$ 
along which the frequencies $\nu_P$ exist for all patterns $P\in \cP$. 
As pointed out before, $(X,\Vert \cdot \Vert)$ stands for a Banach-space. 
\end{Assumption}

\begin{Theorem}\label{thm:ET}
 Assume \ref{ass:first} and let the function $F:\cF(G) \to X$ be almost additive and $\cC$-invariant and let $\tilde F$ be given as in \eqref{eqn:tildeF}. 
Let $(S_n)_{n\in \NN}$ be a nested F\o lner sequence. 
Then the following statements hold.
\begin{itemize}
 \item [(i)] There exists an element $\bar F\in X$ such that
\[
 \lim_{j\to\infty}\left\Vert \bar F - \frac{F(U_j)}{|U_j|}  \right\Vert = 0.
\]
 \item [(ii)] The element $\bar F$ can be expressed as the limit
\[
 \bar F=\lim_{\varepsilon \searrow 0} \,\sum_{i=1}^{N(\epsilon)} \eta_i(\varepsilon) \sum_{P \in \cP(T_i^\epsilon)} \nu_P \, \frac{\tilde{F}(P)}{|T_i^\epsilon|},
\]
where for each $0<\epsilon<1/10$, we set $N(\epsilon):= \lceil \log(\varepsilon)/\log(1- \varepsilon)] \rceil$ 
and $\eta_i(\varepsilon) := \varepsilon(1-\varepsilon)^{N(\epsilon)-i}$ for $i=1,\dots,N(\epsilon)$ 
and where the finite sequence $(T_i^\epsilon)_{i=1}^{N(\epsilon)}$ is given as in Definition \ref{defi:STP} 
with parameters $\beta=2^{-N(\varepsilon)-1} \varepsilon$ and $\delta_0(\beta) < 6^{-N(\varepsilon)}\beta / 16$. 
Each $T_i^\epsilon$ is an element of the sequence~$(S_n)$.
\item [(iii)] For every $0<\epsilon<1/10$, there is some $j_0 := j_0(\varepsilon, \beta) \in \NN$ such that for every $j \geq j_0$, the difference
\begin{eqnarray*}
\Delta(j, \varepsilon) := \left\| \frac{F(U_j)}{|U_j|} - \sum_{i=1}^{N(\epsilon)} \eta_i(\varepsilon) \sum_{P \in \cP(T_i^\epsilon)} \nu_P \, \frac{\tilde{F}(P)}{|T_i^\epsilon|}  \right\|
\end{eqnarray*}
satisfies the estimate 
\begin{align}
\Delta(j, \varepsilon) &\leq (12C+33D) \epsilon  
+C \, \sum_{i=1}^{N(\epsilon)} \eta_i(\varepsilon) \sum_{P \in \cP(T_i^\epsilon)} \left| \frac{\#_P(\cC_{|U_j})}{|U_j|} - \nu_P \right|\nonumber \\
&\quad +4 \sum_{i=1}^{N(\epsilon)}  \eta_i(\varepsilon) \frac{b(T_i^\epsilon )}{|T_i^\epsilon|}
+(C+4D)\frac{|\partial_Q(U_j)|}{|U_j|}\sum_{i=1}^{N(\epsilon)}|T_i^\epsilon|. \label{eqn:EST}
\end{align}
\end{itemize}
\end{Theorem}

\begin{Remark}
Note that this theorem is in fact a true generalization of Theorem 3.1 in \cite{LenzSV-10, LenzSV-11}. 
Due to the fact that in the latter works, 
the authors are able to work with a F{\o}lner consisting of monotiles for symmetric grids, 
all the tilings under consideration consist only of one tile rather than of a basis $\{T_i^{\varepsilon}\}_{i=1}^{N(\varepsilon)}$. 
Thus, in this specific situation, there occurs no $\varepsilon$-parameter, neither in the semi-explicit formula, nor in the corresponding error estimate. 
More precisely, the first sum from $1$ to $N(\varepsilon)$ with weights $\eta_i(\varepsilon)$ does not show up in all the latter expressions. 
\end{Remark}

\begin{proof}
Throughout the proof, we let $0< \varepsilon < 1/10$ be fixed. We first show the Estimate (\ref{eqn:EST}). 
Choose $j_0 = j_0(\varepsilon, \beta, T_i^{\varepsilon})\in \NN$ such that for each $j\geq j_0$ the set $U_j$ 
is sufficiently invariant to apply Theorem \ref{thm:USTP}. 
Therefore for each $j\geq j_0$ we find a finite family $\Lambda_j^\epsilon$ of $\varepsilon$-quasi tilings for the set $T=U_j$ satisfying the uniform special tiling property (USTP), 
cf.\@ Definition \ref{defi:USTP}. 
With no loss of generality, we assume that all the $T_i = T_i^{\varepsilon}$ are taken from a subsequence $\{S_{n_k}\}_{k=1}^{\infty}$ 
such that the expressions $b(S_{n_k})/|S_{n_k}|$ converge to zero monotonically as $k \rightarrow \infty$. 
More precisely, we make sure that $T_i^{\varepsilon} \in \{S_{n_l} \,|\, l \geq i \}$ for all $1 \leq i \leq N$. 
Then, for fixed $j\geq j_0$ we estimate
\begin{align*} 
\Delta(j,\varepsilon) &\leq 
\left\| \frac{F(U_j)}{|U_j|} - \frac{1}{|\Lambda_j^\epsilon|}\,\sum_{\lambda \in \Lambda_j^\epsilon} \sum_{i=1}^{N(\epsilon)} \sum_{c \in C_i^{\lambda}(U_j) } \frac{F(T_i^\epsilon c)}{|U_j|} \right\| \\
&\quad + \left\| \frac{1}{|\Lambda_j^\epsilon|}\,\sum_{\lambda \in \Lambda_j^\epsilon} \sum_{i=1}^{N(\epsilon)} \sum_{c \in C_i^{\lambda}(U_j)} \frac{F(T_i^\epsilon c)}{|U_j|} - \sum_{i=1}^{N(\epsilon)} \eta_i(\varepsilon) \sum_{P \in \cP(T_i^\epsilon)} \frac{\#_P(\cC_{|U_j})}{|U_j|}\frac{\tilde{F}(P)}{|T_i^\epsilon|} \right\| \\
&\quad + \left\| \sum_{i=1}^{N(\epsilon)} \eta_i(\varepsilon) \sum_{P \in \cP(T_i^\epsilon)} \Big( \frac{\#_P(\cC_{|U_j})}{|U_j|} - \nu_P \Big) \, \frac{\tilde{F}(P)}{|T_i^\epsilon|}  \right\|. 
\end{align*}
Again by the triangle inequality, we then obtain
\begin{eqnarray*}
\Delta(j, \varepsilon) \leq D_1(j, \varepsilon) + D_2(j, \varepsilon) + D_3(j,\varepsilon),
\end{eqnarray*}
where
\begin{align*}
D_1(j, \varepsilon) &:= \frac{1}{|U_j||\Lambda_j^\epsilon|}\sum_{\lambda\in\Lambda_j^\epsilon} \left\| F(U_j)- \sum_{i=1}^{N(\epsilon)} \sum_{c \in C_i^{\lambda}(U_j)} F(T_i^\epsilon c) \right\|, \\
D_2(j, \varepsilon) &:= \frac 1 {|U_j|} \left\| \frac{1}{|\Lambda_j^\epsilon|}\,\sum_{\lambda \in \Lambda_j^\epsilon} \sum_{i=1}^{N(\epsilon)} \sum_{c \in C_i^{\lambda}(U_j)} F(T_i^\epsilon c) - \sum_{i=1}^{N(\epsilon)} \frac{\eta_i(\varepsilon)}{|T_i^\epsilon|} \sum_{P \in \cP(T_i^\epsilon)} \#_P(\cC_{|U_j})\tilde{F}(P) \right\|, \text{ and} \\
D_3(j, \varepsilon) &:= \sum_{i=1}^{N(\epsilon)} \eta_i(\varepsilon) \sum_{P \in \cP(T_i^\epsilon)} \left| \frac{\#_P(\cC_{|U_j})}{|U_j|} - \nu_P \right| \, \frac{\|\tilde{F}(P)\|}{|T_i^\epsilon|}.
\end{align*}
With the boundedness of $\tilde{F}$, see Lemma \ref{la:epsaad}, we arrive at
\begin{eqnarray} \label{eqn:ESTD3}
D_3(j, \varepsilon) \leq C \, \sum_{i=1}^{N(\epsilon)} \eta_i(\varepsilon) \sum_{P \in \cP(T_i^\epsilon)} \left| \frac{\#_P(\cC_{|U_j})}{|U_j|} - \nu_P \right|.
\end{eqnarray}
In order to estimate $D_1(j,\varepsilon)$ we make use of the almost additivity of the function $F$ and use part (ii) of Lemma \ref{la:epsaad}. 
So, for each $j\geq j_0$ and $\lambda\in\Lambda_j^\epsilon$ we have
\begin{multline*}
\left\| F(U_j)- \sum_{i=1}^{N(\epsilon)} \sum_{c \in C_i^{\lambda}(U_j)} F(T_i^\epsilon c) \right\|
\leq 
\left\| F(U_j)- F\left(A_{j,\lambda}^\epsilon\right) \right\|
+\left\| F\left(A_{j,\lambda}^\epsilon\right)- \sum_{i=1}^{N(\epsilon)} \sum_{c \in C_i^{\lambda}(U_j)} F(T_i^\epsilon c) \right\|\\
\leq b(A_{j,\lambda}^\epsilon)+b(U_j\setminus A_{j,\lambda}^\epsilon) + \|F(U_j\setminus A_{j,\lambda}^\epsilon)\|+ 3 \sum_{i=1}^{N(\epsilon)} \sum_{c \in C_i^{\lambda}(U_j)}\!\!\!\! b(T_i^\epsilon c) +(3C+9D)\epsilon |U_j|,
\end{multline*}
where
\[
 A_{j,\lambda}^\epsilon=\bigcup_{i=1}^{N(\epsilon)} \bigcup_{c \in C_i^{\lambda}(U_j)} T_i^\epsilon c.
\]
Recall that by the USTP, for each $\lambda \in \Lambda$, the set $U_j$ is $(1-4\varepsilon)$-covered by those translates $T_ic$, $1 \leq i \leq N$, $c \in C_i^{\lambda}(U_j)$ that are fully contained in $U_j$. 
Therefore we have $|U_j\setminus A_{j,\lambda}^\epsilon|\leq 4\epsilon |U_j|$. 
Using this and properties of the boundary term $b$ we obtain
\begin{align*}
 D_1(j,\epsilon)&\leq \frac{1}{|U_j||\Lambda_j^\epsilon|}  \sum_{\lambda\in\Lambda_j^\epsilon} \left( 4\sum_{i=1}^{N(\epsilon)} \sum_{c \in C_i^{\lambda}(U_j)}\!\!\!\! b(T_i^\epsilon c) + (7C+13D)\epsilon |U_j|\right) \\
&\leq (7C+13D)\epsilon +\frac{4}{|U_j||\Lambda_j^\epsilon|}  \sum_{\lambda\in\Lambda_j^\epsilon} \left( \sum_{i=1}^{N(\epsilon)}  |C_i^{\lambda}(U_j)| b(T_i^\epsilon )\right) \\
&= (7C+13D)\epsilon+4\sum_{i=1}^{N(\epsilon)}  \frac{b(T_i^\epsilon )}{|\Lambda_j^\epsilon|}  \sum_{\lambda\in\Lambda_j^\epsilon}    \frac{|C_i^{\lambda}(U_j)|}{|U_j|} .
\end{align*}
Next, we make use of property (II) of Definition \ref{defi:USTP}:
\begin{align*}
 \frac{1}{|\Lambda_j^\epsilon|}\sum_{\lambda\in\Lambda_j^\epsilon}    \frac{|C_i^{\lambda}(U_j)|}{|U_j|}
&\leq \frac{1}{|\Lambda_j^\epsilon||U_j|}\sum_{u\in U_j\setminus\partial_Q(U_j)}\sum_{\lambda\in\Lambda_j^\epsilon}    \1_{C_i^\lambda(U_j)}(u)
+\frac{1}{|\Lambda_j^\epsilon||U_j|}\sum_{u \in \partial_Q(U_j)} \sum_{\lambda\in\Lambda_j^\epsilon}    \1_{C_i^\lambda(U_j)}(u)\\
&\leq \frac{|U_j\setminus\partial_Q(U_j)|}{|U_j|} \left(\frac{\eta_i(\varepsilon)}{|T_i^\epsilon|}+\frac{3\beta}{|T_i^\epsilon|}+\epsilon \gamma_i^{\epsilon,j} \right) + \frac{1}{|\Lambda_j^\epsilon||U_j|}\sum_{u \in \partial_Q(U_j)} \sum_{\lambda\in\Lambda_j^\epsilon}    1\\
&\leq \frac{\eta_i(\varepsilon)}{|T_i^\epsilon|}+\frac{3\beta}{|T_i^\epsilon|}+\epsilon \gamma_i^{\epsilon,j} + \frac{|\partial_Q(U_j)|}{|U_j|} 
\end{align*}
Inserting this in the last estimate for $D_1(j,\epsilon)$ using properties of the boundary term $b$ and $\sum_{i=1}^{N(\epsilon)}\gamma_i^{\epsilon,j} |T_i^\epsilon| \leq 2$ yields
\begin{align}
 D_1(j,\epsilon)
&\leq   (7C+13D)\epsilon+ 4\sum_{i=1}^{N(\epsilon)}  b(T_i^\epsilon )\left(\frac{\eta_i(\varepsilon)}{|T_i^\epsilon|}+\frac{3\beta}{|T_i^\epsilon|}+\epsilon \gamma_i^{\epsilon,j} + \frac{|\partial_Q(U_j)|}{|U_j|}  \right) \nonumber  \\
 &\leq  (7C+13D)\epsilon+ 4 \sum_{i=1}^{N(\epsilon)}  \eta_i(\varepsilon) \frac{b(T_i^\epsilon )}{|T_i^\epsilon|}+12\beta D N(\epsilon) +8\epsilon D + 4 \frac{|\partial_Q(U_j)|}{|U_j|} \sum_{i=1}^{N(\epsilon)}  b(T_i^\epsilon ) \nonumber\\
&\leq  (7C+33D)\epsilon+ 4 \sum_{i=1}^{N(\epsilon)}  \eta_i(\varepsilon) \frac{b(T_i^\epsilon )}{|T_i^\epsilon|}+4 D \frac{|\partial_Q(U_j)|}{|U_j|} \sum_{i=1}^{N(\epsilon)} |T_i^\epsilon|, \label{eqn:ESTD1}
\end{align}
where the last step uses $\beta N(\epsilon)\leq \epsilon$.
% 
% Summing up over $U_j$ the expression in Definition \ref{defi:USTP}, property (II), yields with $\beta = 2^{-N(\varepsilon)-1}\varepsilon$ that for each $j \geq j_0$, there are non-negative numbers $\gamma_{i}^{\epsilon,j}$, $i=1,\dots,N(\epsilon)$, satisfying $\sum_{i=1}^{N(\epsilon)}\gamma_i^{\epsilon,j} |T_i^\epsilon| \leq 2$ such that
% \begin{align*}
%  D_1(j,\epsilon)
% &\leq 4\sum_{i=1}^{N(\epsilon)}  \frac{b(T_i^\epsilon )}{|\Lambda_j^\epsilon|}  \sum_{\lambda\in\Lambda_j^\epsilon}    \frac{|C_i^{\lambda}(U_j)|}{|U_j|} + (7C+13D)\epsilon \\
% &\leq 4\sum_{i=1}^{N(\epsilon)}  b(T_i^\epsilon )   \left(\frac{\eta_i(\varepsilon)}{|T_i^\epsilon|}+\frac{3\beta}{|T_i^\epsilon|}+\epsilon \gamma_i^{\epsilon,j} \right)+ (7C+13D)\epsilon \\
% &\leq 4 \sum_{i=1}^{N(\epsilon)}  \eta_i(\varepsilon) \frac{b(T_i^\epsilon )}{|T_i^\epsilon|}+12\beta D N(\epsilon) +8\epsilon D + (7C+13D)\epsilon.  
% \end{align*}
% By the choice of $\beta$ we have $\beta N(\epsilon)\leq \epsilon$ and we arrive at
% \begin{align}\label{eqn:ESTD1}
%  D_1(j,\epsilon)
% \leq 4 \sum_{i=1}^{N(\epsilon)}  \eta_i(\varepsilon) \frac{b(T_i^\epsilon )}{|T_i^\epsilon|}+ (7C+33D)\epsilon.
% \end{align}

By exploiting the nice property (II) of Definition \ref{defi:USTP} again, we finally estimate $D_2(j, \varepsilon)$. Before we do so, we change the order of summation to obtain
\[
 \sum_{P \in \cP(T_i^\epsilon)} \#_P(\cC_{|U_j})  \tilde{F}(P)
=\sum_{\ato{u \in U_j}{T_i^\epsilon u \subseteq U_j}}   {F}(T_i^\epsilon u).
\]
Therefore, we have
\begin{align*}
D_2(j, \varepsilon) 
&= \frac{1}{|U_j|} \left\| \sum_{i=1}^{N(\epsilon)} \sum_{u \in U_j} \frac{1}{|\Lambda_j^\epsilon|} \sum_{\lambda\in\Lambda_j^\epsilon} \one_{C_i^{\lambda}(U_j)} (u) F(T_i^\epsilon u) - \sum_{i=1}^{N(\epsilon)} \frac{\eta_i(\varepsilon)}{|T_i^\epsilon|} \sum_{\ato{u \in U_j}{T_i^\epsilon u \subseteq U_j}}   {F}(T_i^\epsilon u) \right\| \\
&\leq \frac{1}{|U_j|}  \sum_{i=1}^{N(\epsilon)} \sum_{u \in U_j} \left| \frac{1}{|\Lambda_j^\epsilon|} \sum_{\lambda\in\Lambda_j^\epsilon} \one_{C_i^{\lambda}(U_j)} (u)  -  \frac{\eta_i(\varepsilon)}{|T_i^\epsilon|} \right|   \left\|{F}(T_i^\epsilon u) \right\|.
\end{align*}
Again, we spit the sum over $U_j$ into two parts: one where we are able to apply property (II) and a ``small'' one. Besides this we use the boundedness of $F$ and $\beta N(\epsilon)\leq \epsilon$ to end up with
\begin{align}
 D_2(j,\epsilon) &\leq
\frac{C}{|U_j|}  \sum_{i=1}^{N(\epsilon)}|T_i^\epsilon|\left( \!\sum_{\ato{u \in U_j}{u\notin \partial_Q(U_j)}} \left|  \sum_{\lambda\in\Lambda_j^\epsilon}\! \frac{\one_{C_i^{\lambda}(U_j)} (u)}{|\Lambda_j^\epsilon|}   -  \frac{\eta_i(\varepsilon)}{|T_i^\epsilon|} \right| + \!\!\! \sum_{u \in \partial_Q(U_j)} \left| \sum_{\lambda\in\Lambda_j^\epsilon} \! \frac{\one_{C_i^{\lambda}(U_j)} (u)}{|\Lambda_j^\epsilon|}  -  \frac{\eta_i(\varepsilon)}{|T_i^\epsilon|} \right| \right) \nonumber\\
& \leq
\frac{C}{|U_j|}  \sum_{i=1}^{N(\epsilon)}|T_i^\epsilon|\left(  |U_j|\biggl(\frac{3\beta}{|T_i^\epsilon|} +\epsilon \gamma_i^{\epsilon,j}\biggr) + | \partial_Q(U_j)| \right) \nonumber\\
& \leq
C  \sum_{i=1}^{N(\epsilon)}  \bigl(3\beta +\epsilon  \gamma_i^{\epsilon,j}|T_i^\epsilon|\bigr) + \frac{C| \partial_Q(U_j)|}{|U_j|}  \sum_{i=1}^{N(\epsilon)}|T_i^\epsilon|  
\leq 5\epsilon C +\frac{C| \partial_Q(U_j)|}{|U_j|}  \sum_{i=1}^{N(\epsilon)}|T_i^\epsilon|  \label{eqn:ESTD2}
\end{align}

To finish the proof of (iii), we combine the inequalities (\ref{eqn:ESTD1}), (\ref{eqn:ESTD2}) and (\ref{eqn:ESTD3}) and finally arrive at
\begin{align*} 
\Delta(j, \varepsilon) 
&\leq D_1(j, \varepsilon) + D_2(j, \varepsilon) + D_3(j, \varepsilon)  \\
&\leq (12C+33D) \epsilon  +     C \, \sum_{i=1}^{N(\epsilon)} \eta_i(\varepsilon) \sum_{P \in \cP(T_i^\epsilon)} \left| \frac{\#_P(\cC_{|U_j})}{|U_j|} - \nu_P \right| \\
&\quad + 4 \sum_{i=1}^{N(\epsilon)}  \eta_i(\varepsilon) \frac{b(T_i^\epsilon )}{|T_i^\epsilon|}+(C+4 D) \frac{|\partial_Q(U_j)|}{|U_j|} \sum_{i=1}^{N(\epsilon)} |T_i^\epsilon|
\end{align*}
for all $j \geq j_0(\varepsilon, \beta, T_i^{\varepsilon})$. Since $0 < \varepsilon < 1/10$ (and therefore also $\beta$) was arbitrarily chosen, this shows the desired estimate (\ref{eqn:EST}) for $\Delta(j,\varepsilon)$, $j \geq j_0(\varepsilon, \beta, T_i^{\varepsilon})$.

%Note that as $(S_n)$ is a F\o lner sequence and each $T_i^\epsilon$ is an element of this 
%sequence we have $|T_i^\epsilon|^{-1}b(T_i^\epsilon)$ is bounded by constant uniformly for 
%all $\epsilon>0$ and $i=1,\dots,N(\epsilon)$. 

Furthermore, 
%by Theorem \ref{thm:STP} we have $T_i^{\epsilon}\in \{S_i,S_{i+1},\dots\}$ which gives %that the limit $\lim_{\epsilon\searrow 0}|T_i^\epsilon|^{-1}b(T_i^\epsilon)=0$ exists %uniformly in $\epsilon$. 
Lemma \ref{lemma:ele4} yields with the choice of the $T_i^{\varepsilon}$, as well as with the monotonicity assumption on the sequence $b(S_{n_k})/|S_{n_k}|$ (see above) that 
\[
\lim_{\epsilon\searrow 0}\sum_{i=1}^{N(\epsilon)}  \eta_i(\varepsilon) |T_i^\epsilon|^{-1} b(T_i^\epsilon ) \leq \lim_{\epsilon \rightarrow 0} \sum_{i=1}^{N(\epsilon)}  \eta_i(\varepsilon) |S_{n_i}|^{-1} b(S_{n_i})  =0. 
\]
This and the fact that the frequencies $\nu_P$ along $(U_j)_j$ exist shows with (\ref{eqn:EST}) that
\begin{eqnarray} \label{eqn:lim}
\lim_{\varepsilon \rightarrow 0} \, \lim_{j \rightarrow \infty} \Delta(j,\varepsilon) = 0.
\end{eqnarray}
Now the triangle inequality shows that
\begin{eqnarray*}
\left\| \frac{F(U_j)}{|U_j|} - \frac{F(U_m)}{|U_m|} \right\| &\leq& \Delta(j, \varepsilon) + \Delta(m, \varepsilon)
\end{eqnarray*}
for all $0 < \varepsilon < 1/10$ and every $j \geq j_0(\varepsilon)$. 
By (\ref{eqn:lim}) $(|U_j|^{-1}F(U_j))_{j \in \NN}$ must be a Cauchy sequence and hence converges in the Banach space $X$ to some element $\overline{F}$. 
The limit in (\ref{eqn:lim}) also shows
\begin{equation*}
\overline F = \lim_{\epsilon\searrow 0}\sum_{i=1}^{N(\varepsilon)} \varepsilon(1-\varepsilon)^{N(\varepsilon)-i} \sum_{P \in \cP(T_i)} \nu_P \, \frac{\tilde{F}(P)}{|T_i|}. \qedhere
\end{equation*}
\end{proof}	

We can use Inequality \eqref{eqn:EST} to obtain explicit bounds on the speed of convergence. This will be shown in the next corollary.
\begin{Corollary}\label{cor:ET}
 In the situation of Theorem \ref{thm:ET}, the following estimates hold true:
\begin{align*}
\left\| \overline F -\frac{F(U_j)}{|U_j|}\right\|
&\leq
(24C+66D) \epsilon  +     C \, \sum_{i=1}^{N(\epsilon)} \eta_i(\varepsilon) \sum_{P \in \cP(T_i^\epsilon)} \left| \frac{\#_P(\cC_{|U_j})}{|U_j|} - \nu_P \right| \\
&\quad +8 \sum_{i=1}^{N(\epsilon)}  \eta_i(\varepsilon) \frac{b(T_i^\epsilon )}{|T_i^\epsilon|}
+ (C+4 D) \frac{|\partial_Q(U_j)|}{|U_j|} \sum_{i=1}^{N(\epsilon)} |T_i^\epsilon|
\end{align*}
for $j \geq j_0(\varepsilon)$ and
\[
 \left\| \overline F -\sum_{i=1}^{N(\epsilon)} \eta_i(\varepsilon) \sum_{P \in \cP(T_i^\epsilon)} \nu_P \, \frac{\tilde{F}(P)}{|T_i^\epsilon|}\right\|
\leq
(12C+33D) \epsilon  + 4 \sum_{i=1}^{N(\epsilon)}  \eta_i(\varepsilon) \frac{b(T_i^\epsilon )}{|T_i^\epsilon|}.
\]
\end{Corollary}
\begin{proof}
Fix $0 \leq \varepsilon < 1/10$. For the first estimate we get, using the definition of $\Delta(\cdot,\cdot)$, as well as the triangle inequality
\[
 \left\| \overline F -\frac{F(U_j)}{|U_j|}\right\|
=\lim_{k\to\infty}\left\| \frac{F(U_k)}{|U_k|} -\frac{F(U_j)}{|U_j|}\right\|
\leq \lim_{k\to\infty} \left[\Delta(k,\epsilon) +  \Delta(j,\epsilon) \right].
\]
Now the Estimate \eqref{eqn:EST} yields the desired bound. To verify the second bound we write
\[
 \left\| \overline F -\sum_{i=1}^{N(\epsilon)} \eta_i(\varepsilon) \sum_{P \in \cP(T_i^\epsilon)}\!\!\! \nu_P \, \frac{\tilde{F}(P)}{|T_i^\epsilon|}\right\|
=
\lim_{k\to\infty} \left\| \frac{F(U_k)}{|U_k|} -\sum_{i=1}^{N(\epsilon)} \eta_i(\varepsilon) \sum_{P \in \cP(T_i^\epsilon)}\!\!\! \nu_P \, \frac{\tilde{F}(P)}{|T_i^\epsilon|}\right\|
=
\lim_{k\to\infty} \Delta(k,\epsilon)
\]
and again, the claim follows immediately from \eqref{eqn:EST}.		
\end{proof}

%%%%%%%%%%%%%%%%%%%%%%%%%%%%%%%%%%%%%%%%%%%%%%%%%%%%%%%%%%%%%%%%%%%%%%
\section{Sufficient conditions for the existence of frequencies}

%In this section we prove a pointwise ergodic theorem for amenable groups, based on the tiling theorems in Section 4. This theorem is valid for 
%F\o{}lner sequences with strong invariance conditions (thin F{\o}lner sequences).
%The most general result on the topic has been obtained in \cite{Lindenstrauss-01}, where convergence is shown for all tempered F\o{}lner sequences.

In this section we use the Lindenstrauss pointwise ergodic theorem, 
cf.\@ \cite{Lindenstrauss-01} to prove the existence of frequencies in 
a randomly colored Cayley graph along a tempered F\o{}lner sequence. 
This is motivated by the Banach space-valued ergodic theorem in the previous section, 
as the existence of the frequencies are a basic assumption for its validity. 
In fact, referring to classic ergodic theory, we justify the labeling 'ergodic theorem' for our main statement. 
Precisely, the dynamics in Theorem~\ref{thm:ET} is given by translations of colored patterns in the group. 
In this context, the existence of the pattern frequencies can heuristically be 
interpreted as an assumption which guarantees the existence of some ergodic measure 
in the pattern dynamical system. 
We show in Theorem~\ref{theorem:freq} that for random colorings, 
this latter condition can be deduced from the ergodicity of measure preserving group actions 
on the corresponding probability space.          

We consider a countable, amenable group $G$, 
as well as a probability space $(\Omega,\cS,\mu)$. 
Let $\tau:G\times \Omega\to \Omega$, $(g,\omega)\mapsto \tau_g(\omega)$ be a measurable action of $G$ on $\Omega$. 
We say that $\tau$ is \emph{measure preserving}, if for any $A\in\cS$ and $g\in G$ one has $\mu(A)=\mu(\tau_g(A))$.
Furthermore, a measure preserving action is said to be \emph{ergodic} if $\mu(A)\in\{0,1\}$, whenever $A\in\cS$ with $A=\tau_g^{-1}(A)$ for all $g\in G$.

In this situation, the Lindenstrauss ergodic theorem then reads as follows.
\begin{Theorem}[Lindenstrauss' pointwise ergodic theorem] \label{thm:linde}
 Let $G$ be a countable amenable group and let $\tau$ be a measure preserving and ergodic action of $G$ on the probability space $(\Omega,\cS,\mu)$. 
Furthermore let $(Q_j)$ be a tempered F\o{}lner sequence and $f\in L^1(\mu)$. 
Then one has for $\mu$-almost all $\omega$
\[
 \lim_{j\to\infty}\frac{1}{|Q_j|}\sum_{g\in Q_j}f(\tau_g \omega ) =\int\limits_{\Omega} f(\omega) d \mu(\omega).
\]
\end{Theorem}

%\begin{Theorem}[Ergodic theorem for thin F\o{}lner sequences] \label{thm:thin}
%Fix arbitrary $0 < \varepsilon \leq 1/10$ and $0 < \beta 2^{-N(\varepsilon)}\varepsilon$. Further, assume that $\{T_i\}_{i \in \NN}$ is a nested %strong F{\o}lner sequence satisfying  
%\end{Theorem}

%\begin{proof}
%...
%\end{proof}

All proofs of the above theorem are based on the verification of a 
so-called maximal inequality via abstract, 
geometric covering arguments for very invariant subsets of the group, 
see e.g.\@ \cite{Lindenstrauss-01} and \cite{Weiss-03}. 
Hence, one might wonder if one can use the $\varepsilon$-quasi tiling arguments 
of Section \ref{sec:OW2} to give a direct proof for a pointwise ergodic theorem. 
However, as one can infer from the literature, the assumptions on the 
corresponding decompositions can be kept much milder than in the situation of 
Definition \ref{defi:STP} or even of Definition \ref{defi:USTP}. 
As a consequence, by giving effective estimates for the $\varepsilon$-quasi tilings, 
we have to pay the price of strong invariance conditions. 
% which are not satisfied for inifinite F{\o}lner sequences. 
More precisely, for decreasing $\epsilon$, the tiles used in this $\epsilon$-quasi tiling 
are taken from a F\o{}lner (sub-)sequence consisting of increasingly invariant sets, 
such that the notion of temperedness is too mild to describe this condition.
In that sense, our decomposition results should be considered 
as complementary to the classical results: 
we obtain more information on the shape, as well as on the degree of uniformity 
of the $\varepsilon$-quasi tilings, but we lose the properties of a tempered 
F{\o}lner sequence which lead to a maximal inequality. 

%The reason for the restrictions in the above theorem lies 
%$in the strong invariance conditions imposed on the $\varepsilon$-quasi tiles used to derive 
%the transfer inequalities. , . More precisely, the temperedness of the 
%F{\o}lner sequence under consideration is sufficient. Nevertheless, Theorem 
%\ref{thm:thin} provides a nice link between the combinatorics of the group and the 
%proof of the maximal inequality for classical pointwise ergodic %theorems.    

In the following, we use the above ergodic theorem to show that frequencies of patterns exist almost surely in an appropriate probability space.
We consider a countable amenable group $G$ and a finite set $\cA$, which we will interpret as the set of colors. 
The probability space $(\Omega,\cS,\PP)$ is given in the following way.
The sample space is the set 
\[
\Omega=\cA^G=\{\omega=(\omega_g)_{g\in G}\mid \omega_g \in\cA \text{ for all }g\in G \}.
\]
The sigma-algebra $\cS$ is generated by the cylinder sets  
and $\PP$ is a probability measure on $(\Omega,\cS)$. 
Setting for each $\omega\in\Omega$
\[
 \cC_\omega:G\to \cA, \quad g\mapsto \omega_g,
\]
shows that each $\omega$ can be interpreted as a coloring of $G$.
Let $\tau:G\times \Omega \to \Omega$ be given by
\begin{equation}\label{def:tau}
(g,\omega) \mapsto \tau_g\omega = \omega g^{-1}, 
\end{equation}
where $\omega g^{-1}\in \Omega$ is the element satisfying
\[
(\omega g^{-1})_x = \omega_{xg} \quad \quad (x\in G).
\]
We assume that the action $\tau$ of $G$ on $\Omega$ is measure preserving and ergodic.

% The following result is a special case of Lindenstrauss' pointwise ergodic theorem in \cite{Lindenstrauss-01}. 
% 
% \begin{Theorem}\label{theorem:linde}
%   Let $G$ act from the left on a measure space $(\Omega,\cA,\PP)$ by an ergodic 
%and measure preserving transformation $\tau$ an let $(Q_n)$ be a tempered F\o{}lner sequence. Then for any $ f\in L^1(\PP)$
%   \[
%   \lim_{n\rightarrow\infty}\frac{1}{\vert Q_n \vert}\sum_{g\in Q_n}f(\tau_g \omega ) =\int\limits_{\Omega} f(\omega) d \PP(\omega)
%   \]
%   holds almost surely.
% \end{Theorem}

Using Theorem \ref{thm:linde} we can prove the existence of the frequencies $\nu_P$ along any tempered F\o{}lner sequence $(Q_j)$. 
This has been shown in similar situations for example in \cite{LenzSV-10,Schwarzenberger-12,PogorzelskiSS-11}.

\begin{Theorem}\label{theorem:freq}
Let the probability space $(\Omega, \cS,\PP)$ be given and let the action $\tau$ of $G$ on $\Omega$ be measure preserving and ergodic. 
Furthermore let $(Q_j)$ be a tempered F\o{}lner sequence. 
Then there exists a set $\tilde\Omega$ of full measure such that the limit
\[
\lim_{n\to\infty}\frac{\sharp_{P}\left((\cC_\omega)_{|Q_j}\right)}{|Q_j|}
\]
exists for all $P\in\cP$ and all $\omega\in\Omega$ and the limit is independent of the specific choice of $\omega$.
\end{Theorem}
\begin{proof}
Let $P:D(P)\to \cA$ be some pattern. 
As the number of occurrences of two equivalent patterns $P_1$ and $P_2$ in another pattern $P_3$ is the same, 
we can assume without loss of generality that $\id\in D(P)$. 
Set $A_P:=\{\omega\in\Omega\mid (\cC_\omega)_{|D(P)}=P\}$ 
and let $f_P:\Omega\to\{0,1\}$ be the indicator function of $A_P$. 
Now we can estimate the number of occurrences of $P$ in $(\cC_\omega)_{|Q_j}$ by
\begin{equation}\label{eq:freq1}
 \sum_{g\in Q_j } f_P(\omega g^{-1}) - |\partial_{D(P)}Q_j|
\leq
 \sum_{g\in Q_j \setminus (\partial_{D(P)}Q_j)} f_P(\omega g^{-1})
\leq 
 \sharp_P\left((\cC_\omega)_{|Q_j}\right)
\leq
\sum_{g\in Q_j } f_P(\omega g^{-1})
\end{equation}
We apply Theorem \ref{thm:linde}, which is possible as $\tau$ acts measure preserving and ergodic and since $f_P\in L^1(\PP)$. 
This yields that there is a set $\Omega_P$ of full measure such that
\[
\lim_{j\to\infty}\frac{1}{|Q_j|}\sum_{g\in Q_j } f_P(\omega g^{-1})= \EE(f_P)
\]
holds for all $\omega\in\Omega_P$. Using this with \eqref{eq:freq1} and the fact that $(Q_j)$ is a F\o{}lner sequence we obtain
\[
\lim_{j\to\infty}\frac{\sharp_P\left((\cC_\omega)_{|Q_j}\right)}{|Q_j|} = \EE(f_P)
\] 
for all $\omega\in\Omega_P$. 
Next, set $\tilde\Omega=\bigcup_{P\in\cP}\Omega_P$ and use the fact that $\cP$ is countable to get the desired set $\tilde\Omega$ 
of full measure such that the frequencies along $(Q_j)$ exist for all patterns $P\in\cP$ and all $\omega\in\tilde\Omega$. 
The independence of the specific choice of $\omega$
is clear as $\EE(f_P)$ is independent of $\omega$. 
\end{proof}

\begin{Remark}
 In the case where the measure $\PP$ has a product structure $\PP =\Pi_{g\in G} \mu$ 
and $\mu$ is some measure on $\cA$, it is easy to show that $\tau$, 
defined as in (\ref{def:tau}) is measure preserving and ergodic. 
This shows that Theorem \ref{theorem:freq} applies in particular to i.i.d. models.
\end{Remark}

%\subsection*{Acknowledgement}
%We would like to take this chance to thank our advisers {\sc Daniel Lenz} and {\sc Ivan %Veseli\'c} for their support and their guidance during this work. We are very grateful that they %generously shared their knowledge and their ideas about the topic with us.
%% In particular, {\sc FP} expresses his thanks to {\sc Daniel Lenz} for drawing his attention to %the {\sc Ornstein/Weiss} theory of amenable groups.       
%% Special thanks go to {\sc Ivan Veseli\'c} as he developed the idea to apply a Banach %space-valued ergodic theorem to obtain results on densities of percolation clusters. 
%Moreover, {\sc FP} would like to give thanks to {\sc Benjamin Weiss} for a fruitful discussion %at the Hebrew University, as well as for his very helpful remarks on the theory of %$\varepsilon$-quasi tilings. 
%{\sc FP} gratefully points out that his work was partially supported by the German Research %Council (DFG) and the German National Academic Foundation (Studienstiftung des deutschen %Volkes). 

%%%%%%%%%%%%%%%%%%%%%%%%%%%%%%%%%%%%%%%%%%%%%%%%%%%%%%%%%%%%%%%%%%%%%%

%% file: IDS.tex
\section{Integrated density of states}

In this section, we are interested in the approximation of the integrated density of states via finite volume analogues. For periodic, finite hopping range operators, recent results show that uniform convergence of the integrated density of states can be obtained in the very general geometric situation of Benjamini-Schramm convergent, hyperfinite graph sequences, cf.\@ \cite{Elek-06pre, Elek-08, Pogorzelski-13}.   
Considering groups with a coloring, we are able to cover non-periodic operators. The corresponding coefficients take their values according to the colors of geometric patterns. In this context, we show that for each finitely generated, amenable group, the approximants converge uniformly in the energy variable to a certain limit function. This model has been investigated by {\sc Lenz, Schwarzenberger, Veseli\'c} in their work \cite{LenzSV-10} in a far more restricted framework. In the following, we will show briefly how the Ergodic Theorem \ref{thm:ET} implies uniform convergence of the approximating functions. For further details, we refer to \cite{LenzSV-10}. We will also stick close to the notation in \cite{LenzSV-10}.

Let $\cH$ be a finite dimensional Hilbert space with norm $\Vert\cdot\Vert$ and denote by $\ell^2(G,\cH)$ the Hilbert space of the functions $u:G\to \cH$ such that $\sum_{g\in G} \| u(g)\|^2 <\infty$ for the usual $\ell^2$-norm. For given $Q\in \cF(G)$ and an operator $H:\ell^2(G,\cH)\to\ell^2(G,\cH)$ set
\[
 H[Q]:=p_Q H i_Q : \ell^2(G,\cH) \to \ell^2(G,\cH),
\]
where $i_Q: \ell^2(Q,\cH)\to\ell^2(G,\cH)$ and $p_Q: \ell^2(G,\cH)\to\ell^2(Q,\cH)$ are the canonical inclusions and projections given by
\[ 
i_Q(v)(x):=\begin{cases} u(x) &\text{if $x\in Q$}\\ 0& \mbox{else} \end{cases}\, , 
\quad \quad \quad
p_Q(u)(x):=u(x)\quad \text{for all $x\in Q$}
\]
and $\ell^2(Q,\cH)$ is the space of all functions $v:Q\to\cH$. Furthermore we set $H(x,y):=p_{\{x\}}Hi_{\{y\}}$ for all $x,y\in G$.
\begin{Definition}
 Let $G$ be a finitely generated amenable group, $d:G\times G\to \NN_0$ the induced word metric, $\cA$ some finite set, $\cC:G\to\cA$ a coloring and assume that $H:\ell^2(G,\cH)\to\ell^2(G,\cH)$ is self-adjoint. Then we say $H$ is of \emph{finite hopping range} if there is some $M\in\NN$ such that $d(g,h)\geq M$ implies $H(g,h)=0$. Furthermore $H$ is called \emph{$\cC$-invariant} if there is some $N\in\NN$ such that whenever
\[
 \big(  {\mathcal C}\vert_{B_{N}(g) \cup B_{N}(h)}\big)x = {\mathcal C}\vert_{B_{N}(gx) \cup B_{N}(hx)}
\]
for some $x,g,h\in G$, then we have $H(g,h)=H(gx,hx)$.
If both conditions are fulfilled, then we call $R=\max\{M,N\}$ {\em the overall range} of $H$.
\end{Definition}

\begin{Definition}
Let ${\mathcal B}(\mathbb R)$ denote the Banach space of right-continuous, bounded functions $f:\RR\rightarrow \RR$, equipped with the supremum norm. Given a self-adjoint operator $A$ on a finite dimensional Hilbert space $V$, we define its cumulative eigenvalue counting function $n(A)\in {\mathcal B}(\mathbb R)$ by setting
\[ n(A)(E):= \vert \{ i\in \mathbb N \ \vert\ \lambda_i\leq E \} \vert\]
for all $E \in \mathbb R$, where $\lambda_i, i=1,\dots,\dim V$ are the eigenvalues of $A$, counted according to their multiplicities.
\end{Definition}

\begin{Assumption}\label{ass:second}
 Assume \ref{ass:first} and additionally that $G$ is finitely generated, $\cH$ is a finite dimensional Hilbert space, $H:\ell^2(G,\cH)\to\ell^2(G,\cH)$ is a self-adjoint, finite hopping range and $\cC$-invariant operator with overall range $R$.
\end{Assumption}

\begin{Proposition}\label{prop:idsaad}
 Assume \ref{ass:second}. The function $F_R^H: {\mathcal F}(G)\rightarrow {\mathcal B}(\mathbb R),\ Q\mapsto F_R^H(Q):= n(H[Q_R])$ is ${\mathcal C}$-invariant and almost-additive with the boundary term $b(Q):=4\vert \partial^RQ \vert \dim ({\mathcal H})$.
\end{Proposition}
\begin{proof}
 see \cite{LenzSV-10}, Proposition 4.6.
\end{proof}
Note that indeed, this $b:\cF(G)\to [0,\infty)$ is a boundary term since $\lim_{j\to\infty}|\partial^R(Q_j)|/|Q_j|=0$ for any F\o{}lner sequence $(Q_j)$, c.f. \cite{LenzSV-10}. Furthermore $b$ is obviously invariant under translations and we have $b(Q)\leq D|Q|$ with $D=4\dim(\cH)|B_R|$. Also property (iv) of Definition \ref{defi:BT} is true for $b$ as it holds for the function $|\partial^R(\cdot)|$. Moreover, we know from the definition of $F_R^H$ that $F_R^H(Q)\leq C|Q|$, where $C:=\dim(\cH)$.

\begin{Theorem}
 Assume \ref{ass:second}. Then there exists a unique probability measure $\mu_H$ with distribution function $N_H$, such that
\[
 \lim_{j\to\infty}\left\| \frac{n(H[U_{j,R}])}{\dim(\cH)|U_{j}|} - N_H \right\|_\infty = 0
\]
The function $N_H$ is called \emph{integrated density of states}. Furthermore for each $0<\epsilon<1/10$ there is a constant $j_0=j_0(\epsilon)$ such that the estimate
\begin{align*}
 \left\| N_H - \frac{n(H[U_{j,R}])}{\dim(\cH)|U_{j}|} \right\|_\infty   
&\leq  
(24+264|B_R|) \epsilon  +   \sum_{i=1}^{N(\epsilon)} \eta_i(\varepsilon) \sum_{P \in \cP(T_i^\epsilon)} \left| \frac{\#_P(\cC_{|U_j})}{|U_j|} - \nu_P \right|  \\
&\quad +(1+16 |B_R|) \frac{|\partial_Q(U_j)|}{|U_j|}\sum_{i=1}^{N(\epsilon)}|T_i^\epsilon|+  32 \sum_{i=1}^{N(\epsilon)}  \eta_i(\varepsilon) \frac{|\partial^R(T_i^\epsilon)|}{|T_i^\epsilon|}     
\end{align*}
holds for all $j\geq j_0$. As before we have  $N(\epsilon):=\lceil\log(\varepsilon)/\log(1- \varepsilon)\rceil$ and $\eta_i(\varepsilon) := \varepsilon(1-\varepsilon)^{N(\epsilon)-i}$ for $i=1,\dots,N(\epsilon)$. The sets $T_i^\epsilon$, $i=1,\dots,N(\epsilon)$ are given as in Definition \ref{defi:STP} with parameters $\beta=2^{-N(\varepsilon)-1} \varepsilon$ and $\delta_0(\beta)$ and a nested F\o lner sequence $(S_n)$.
\end{Theorem}
\begin{proof}
 By Proposition \ref{prop:idsaad} we have that $F_R^H: {\mathcal F}(G)\rightarrow {\mathcal B}(\mathbb R),\ Q\mapsto F_R^H(Q):= n(H[Q_R])$ is ${\mathcal C}$-invariant and almost-additive with the boundary term $b(Q):=4\vert \partial^RQ \vert \dim ({\mathcal H})$. Recall that $C=\dim(\cH)$. We apply Theorem \ref{thm:ET} and Corollary \ref{cor:ET} in order to find a function $\tilde N_H\in \cB(\RR)$ such that:
\begin{align*}
 \frac{1}{C}\left\|\tilde N_H -  \frac{F_R^H (U_j)}{|U_j|}  \right\|_\infty
&\leq (24+264|B_R|) \epsilon  +   \sum_{i=1}^{N(\epsilon)} \eta_i(\varepsilon) \sum_{P \in \cP(T_i^\epsilon)} \left| \frac{\#_P(\cC_{|U_j})}{|U_j|} - \nu_P \right| \\
&\quad +(1+16 |B_R|) \frac{|\partial_Q(U_j)|}{|U_j|}\sum_{i=1}^{N(\epsilon)}|T_i^\epsilon|
  +  32 \sum_{i=1}^{N(\epsilon)}  \eta_i(\varepsilon) \frac{|\partial^R(T_i^\epsilon)|}{|T_i^\epsilon|}     
\end{align*}
This proves the claimed estimate with $N_H:=\tilde N_H/C$. As in the proof of Theorem \ref{thm:ET}, we may assume without loss of generality that $|S_n|^{-1}b(S_n)$ converges monotonically to zero, which gives by the choice of the $T_i^\epsilon$ and by Lemma \ref{lemma:ele4} that 
\[
 \lim_{\epsilon\searrow 0} \sum_{i=1}^{N(\epsilon)}  \eta_i(\varepsilon) \frac{b(T_i^\epsilon )}{|T_i^\epsilon|}      = 0,
\]
This shows
\[
  \lim_{j\to\infty}\left\| N_H -  \frac{F_R^H (U_j)}{\dim(\cH)|U_j|}  \right\|_\infty =0.
\]
The proof that the $N_H$ is the distribution function of a probability measure follows easily using the fact that we obtained convergence with respect supremum norm.
%  The monotonicity of $N_H$ is clear since we have for any $E'\leq E$
% \[
%  N_H(E)-N_H(E')=\lim_{j\to\infty}\frac{F_R^H(U_j)(E)-F_R^H(U_j)(E')}{\dim(\cH)|U_j|}\geq 0,
% \]
% as the functions $F_R^H(U_j)$ are monotone. By the uniform convergence, the right continuity of the functions $F_R^H(U_j)$ carries over to the limit $N_H$. In fact we use the uniform convergence to interchange the limits in the computation:
% \[
%  \lim_{E'\searrow E}N_H(E') 
% =\lim_{E'\searrow E} \lim_{j\to\infty} \frac{F_R^H(U_j)(E')}{\dim(\cH)|U_j|}
% =\lim_{j\to\infty}\lim_{E'\searrow E}  \frac{F_R^H(U_j)(E')}{\dim(\cH)|U_j|}
% =N_H(E).
% \]
% Another application of uniform convergence yields
% \[
%  \lim_{E\to\infty}N_H(E) 
% = \lim_{E\to\infty} \lim_{j\to\infty} \frac{F_R^H(U_j)(E)}{\dim(\cH)|U_j|}
% = \lim_{j\to\infty} \lim_{E\to\infty} \frac{F_R^H(U_j)(E)}{\dim(\cH)|U_j|}
% = \lim_{j\to\infty} \frac{\dim(\cH)|U_{j,R}|}{\dim(\cH)|U_j|}
% =1.
% \]
% Similarly, one obtains $\lim_{E\to -\infty}N_H(E)=0$. Hence, $N_H$ is the distribution function of a probability measure.
\end{proof}

For completeness reasons let us state two other results, concerning properties of the integrated density of states. Both are taken from \cite{LenzSV-10}, where one also finds the proofs, which still hold true in the present setting.

\begin{assumption}\label{ass:third}
  Assume \ref{ass:second} and additionally that the frequencies $\nu_P$ are strictly positive for all patterns $P\in {\mathcal P}$ which occur in ${\mathcal C}$, i.e. for which there exists $g\in G$ with ${\mathcal C}\vert_{D(P)g}= Pg $.
\end{assumption}

\begin{Theorem}
 If we assume \ref{ass:third}, then the spectrum of $H$ equals the topological support of $\mu_H$.
\end{Theorem}
\begin{Corollary}\label{cor_disc}
 Assume \ref{ass:third} and let $E \in \mathbb R$. Then $E$ is a point of discontinuity of $N_H$, if and only if there exists a compactly supported eigenfunction of $H$ corresponding to $E$.
\end{Corollary}

%% file: perc.tex
\section{Distribution of percolation clusters}
In this section, we consider a bond percolation model for Cayley graphs given by finitely generated amenable groups. The aim is to study densities and distribution functions of percolation clusters. Parts of these results have been obtained in the special case $G=\ZZ^d$ (with site percolation) in the thesis \cite{Weissbach-11} supervised by {\sc Veseli\'c} and {\sc Schwarzenberger}. Our Theorem \ref{thm:ET} allows us to generalize these ideas to arbitrary, finitely generated amenable groups. 
In this sense, we also generalize certain results of {\sc Grimmett} in \cite{Grimmett-76,Grimmett-99}. 
While in the first subsection we prove the existence of quantities, describing the density of clusters of a fixed size in the percolated graph, the second subsection is devoted to a close look on the dependence of these quantities on the percolation parameter. Note that as we deal with countable groups, the Haar measure $|\cdot|$ of a set is equal to the counting measure.

\subsection{Distribution function and densities}
Let $G$ be an amenable group and suppose that $S\subseteq G$ is a finite and symmetric set of generators. As before, the set of all finite subsets of $G$ is $\cF(G)$.
We denote the undirected graph by $\Gamma=\Gamma(G,S)=(V,E)$, where the vertex set $V$ is equal to the set of elements in $G$ and two vertices
$x$ and $y$ are adjacent if there exists some $s\in S$ such that $sx=y$. We write $e=[x,y]=[y,x]$ for an edge connecting the vertices $x,y\in G$. Note that the induced graph metric on $\Gamma$ corresponds to the word metric defined in Remark \ref{remark:fingen}. In the following, we will also use the notation $\partial^r(\Lambda)$ for the $r$-boundary as introduced in Remark \ref{remark:fingen}.

To obtain a random subgraph we set 
\[
\Omega=\{0,1\}^{V\times S}=\{(\omega_{v,s})_{v\in V, s\in S}\mid \omega_{v,s}\in \{0,1\} \text{ for all }v\in V, s\in S\} 
\]
and consider the $\sigma$-algebra $\cZ$ which is generated by the cylinder sets.
Furthermore, we set $\PP$ to be the product measure $\PP=\prod_{v\in V, s\in S}\nu_{s}$, where $\nu_{s}$ is a Bernoulli measure with values $\{0,1\}$. More precisely, we fix $(p_s)_{s\in S}\in [0,1)^S$ and require $\nu_{s}$ to fulfill $\nu_{s}(X_{v,s}=1)=p_s=1-\nu_{s}(X_{v,s}=0)$ for all $v\in V$ and $s\in S$. Here for each pair $(v,s)\in V\times S$ the mapping $X_{v,s}:\Omega\to\{0,1\}$ is the projection on the coordinate $(v,s)$, i.e. $X_{v,s}(\omega)=\omega_{v,s}$.

In this situation each $\omega=(\omega_{v,s})_{v\in V,s\in S}\in \Omega$ defines a subgraph $\Gamma_\omega=(V,E_\omega)$ of $\Gamma$ by setting $E_\omega=\{[x,y]\in E\mid X_{x,yx^{-1}}=X_{y,xy^{-1}}=1\}$. Therefore we have for given $e=[x,sx]\in E$ with $s\in S\setminus\{\id \}$ that $\PP(e\in E_\omega)=p_s p_{s^{-1}}$ and (if $\id\in S$) we have $\PP([x,x]\in E_\omega))=p_{\id}$. An edge $e\in E$ is called active in the configuration $\omega$ respectively in the graph $\Gamma_\omega$ if $e\in E_\omega$ and non-active otherwise. Note that in the case where there exists a $p\in[0,1)$ such that $p_s^2= p$ for all $s\in S\setminus\{\id \}$ and (if $\id\in S$) $p_{\id}=p$, each edge $e\in E$ will be active with probability $p$.

\begin{Remark}
 Let us briefly explain the reasons for which we choose this probability space. First of all, our aim is to interpret each $\omega\in\Omega$ as a coloring of the graph with finitely many colors in order to apply the ergodic theorem developed in Section 5. Therefore, it is necessary to define $\Omega$ as a product over the set of vertices. Furthermore, we would like to allow for different probabilities determining the existence of the edges. However, the latter should be invariant under translation via multiplication of group elements. Hence it is useful to couple these probabilities with the generators. In the case where $S$ only contains elements which are not self-inverse, this turns out to be easy. Then one could write $S$ as the disjoint union of $\bar S$ and $\bar S^{-1}$ and express the set of all edges as $E=\dot\bigcup_{v\in V,s\in \bar S} [v,sv]$. Hence it is enough to consider the space
\[
 \bar\Omega=\{0,1\}^{V\times \bar S}=\{(\omega_{v,s})_{v\in V, s\in \bar S}\mid \omega_{v,s}\in \{0,1\} \text{ for all }v\in V, s\in \bar S\}, 
\]
along with the measure
\[
 \bar\PP=\prod_{v\in V, s\in \bar S}\bar\nu_{s}\quad\text{where}\quad \bar\nu_s(X_{v,s}=1)=p_s=1-\bar\nu_s(X_{v,s}=0)\text{ and }p\in[0,1)^{\bar S}.
\]
Here again $X_{v,s}(\omega)=\omega_{v,s}$ for all $v\in V$ and $s\in\bar S$. In this case we would set an edge $[v,sv]\in E$ to be active in the configuration $\omega$ if and only if $X_{v,s}(\omega)=1$. This strategy is not possible when $S$ contains a self-inverse element as then one can not decompose $S$ in some set $\bar S$ and its inverse and hence the representation of $E$ as a disjoint union over the set of vertices (as above) is not longer possible. In this case we can write $E=\bigcup_{v\in V,s\in S} [v,sv]$, which is a union that counts each edge (which is not a loop) twice. This leads to the probability space introduced above.
\end{Remark}

Given $e=[x,y]\in E$ and $z\in G$ we set $ez=[xz,yz]\in E$. For $\Lambda\subseteq G$ and $\tilde E\subseteq E$ we define $\tilde E|_\Lambda=\{[x,y]\in \tilde E\mid x,y\in \Lambda\}$.
% and $\omega|_{\Lambda}\in \{0,1\}^{E|_\Lambda}$ by $(\omega_\Lambda)_{e} = \omega_{e}$ for all $e\in E|_\Lambda$.
Two sets $\Lambda,\Lambda'\subseteq G$ are called equivalent in a given configuration $\omega\in\Omega$ (and we write $\Lambda\stackrel{\omega}{\sim}\Lambda')$ if there exists $x\in G$ with $\Lambda x=\Lambda'$ and $\omega_{v,s}=\omega_{vx,s}$ for all $e\in E|_\Lambda$, $s\in S$.
Two vertices $x,y\in G$ are said to be connected in the configuration $\omega$ if there exists an active path in $\Gamma_\omega$ connecting $x$ and $y$, i.e. a sequence of vertices $x=v_0,v_1,\dots,v_n=y\in G$ such that $[v_{i-1},v_i]\in E$ and $X_{v_{i-1},v_iv_{i-1}^{-1}}(\omega)=X_{v_i,v_{i-1}v_i^{-1}}(\omega)=1$ for all $i=1,\dots,n$. In this situation we write $x\stackrel{\omega}{\leftrightarrow} y$. For $x\in G$ denote by
\[
 C_x(\omega)=\{y\in G\mid x\stackrel{\omega}{\leftrightarrow} y\}
\]
the cluster of $x$ in the configuration $\omega$. For a given subset $\Lambda\subseteq G$, we define similarly 
\[
 C_x^\Lambda(\omega)=\{y\in G\mid x\stackrel{\omega,\Lambda}{\leftrightarrow} y\}\subseteq \Lambda,
\]
where $x\stackrel{\omega,\Lambda}{\leftrightarrow} y$ means that there exists an active path in $\Gamma_\omega$ such at all the vertices in the path are elements of $\Lambda$. Note that $x\stackrel{\omega,\Lambda}{\leftrightarrow} y$ implies $C_x^\Lambda(\omega)=C_y^\Lambda(\omega)$. 

Now we define a function which counts clusters in a given set up to a certain size. For given $\omega\in \Omega$ and $\Lambda\in \cF(G)$ we define $F_\omega(\Lambda):\RR\to \NN_0$ by 
\begin{equation}\label{eq:defFclust}
 F_\omega(\Lambda)(m):=\sharp\left\{   C_x^\Lambda(\omega)\mid x\in \Lambda, |C_x^\Lambda(\omega)|\leq m  \right\}.
\end{equation}
For each $\Lambda\in \cF(G)$ and $\omega\in \Omega$ the function $F_\omega(\Lambda)$ is an element of $\cB(\RR)$, which is the space of all bounded and right-continuous functions, equipped with the supremum norm. We are interested in the limit $F_\omega(\Lambda_n)/K_\omega(\Lambda_n)$ as $n\to\infty$, where for some $\Lambda\in\cF(G)$ the expression $K_\omega(\Lambda)$ denotes the number of clusters in $\Lambda$, i.e.
\begin{equation}\label{eq:defKn}
 K_\omega(\Lambda):=\sharp \left\{C_x^\Lambda(\omega)\mid x\in \Lambda \right\} .
\end{equation}
Let us state the main result of this section.
\begin{Theorem}\label{thm:mainperc}
 Let $G$ be an amenable group, $(\Lambda_n)$ a tempered F\o lner sequence and $(\Omega,\cZ,\PP)$ the same probability space as above. Furthermore, let the function $F_\omega:\cF(G)\to \cB(\RR)$ be given as in \eqref{eq:defFclust} and $K_\omega(\Lambda_n)$ as in \eqref{eq:defKn}. Then there exists a set $\tilde\Omega\subseteq\Omega$ of full measure and a distribution function $\Phi\in \cB(\RR)$ associated to a probability measure, such that
\[
 \lim_{n\to\infty}\left\Vert  \frac{F_\omega(\Lambda_n)}{K_\omega(\Lambda_n)}- \Phi \right\Vert = 0
\]
holds true for all $\omega\in\tilde\Omega$, where $\Vert\cdot\Vert$ denotes the supremum norm in $\cB(\RR)$.
\end{Theorem}

By definition it is clear that $F_\omega(\Lambda)=F_\omega(\Lambda')$ whenever $\Lambda\stackrel{\omega}{\sim}\Lambda'$. Furthermore, if $\Lambda\subseteq G$ and $\omega,\omega'\in \Omega$ are such that 
\[
 \sharp \big( E_\omega|_\Lambda \bigtriangleup E_{\omega'}|_\Lambda\big)
=\sharp \big( (E_\omega|_\Lambda\setminus E_{\omega'}|_\Lambda) \cup (E_{\omega'}|_\Lambda\setminus E_{\omega}|_\Lambda) \big) =1,
\]
%there exists an $\bar e\in E|_\Lambda$ with
%\[\omega'_e=\begin{cases}
%             \omega_e & \mbox{ if } e \in E|_\Lambda\setminus\{\bar e\}\\
%	     1-\omega_e & \mbox{ if } e=\bar e\\
%            \end{cases}
%\]
then
\[
 |F_\omega(\Lambda)(m)- F_{\omega'}(\Lambda)(m)|\leq 2.
\]
for all $m\in \RR$. This is clear since the change of one edge from active to non-active (respectively from non-active to active) will perturb the number of clusters in $\Lambda$ of a fixed size at most by 2. Generalizing this idea leads to the following lemma.
\begin{Lemma}\label{la:aad}
Let $\omega\in \Omega$, $\Lambda\subseteq G$ finite and disjoint $\Lambda_1,\dots,\Lambda_k\subseteq \Lambda$ with $\bigcup_{i=1}^k \Lambda_i=\Lambda$ be given. Then
 \[
\left| F_\omega(\Lambda)(m) - \sum_{i=1}^k F_\omega(\Lambda_i)(m)\right| \leq  2|S|\sum_{i=1}^k \left|\partial^1(\Lambda_i)\right|  
 \]
holds true for all $m\in \RR$.
\end{Lemma}
\begin{proof}
 Fix $m\in \RR$ and let $B:=\{[x,y]\in E|_\Lambda\mid x\in \Lambda_i, y\in \Lambda_j\text{ for some }i\neq j\}$ and  $\omega'\in \Omega$ be such that $e\notin E_{\omega'}$ for all $e\in B$ and $e'\in E_{\omega'}\setminus B$ if and only if $e'\in E_\omega\setminus B$.
Using the preliminary considerations one can show inductively that
\[
 |F_\omega(\Lambda)(m)- F_{\omega'}(\Lambda)(m)|\leq 2\cdot \sharp(B).
\]
Furthermore, as each edge in $B$ connects a vertex in $\partial^1(\Lambda_i)$ with a vertex in $\partial^1(\Lambda_j)$ for some distinct $i,j\in\{1,\dots,n\}$ and each vertex in $\Gamma$ is adjacent to exactly $|S|$ vertices, we estimate the number of elements in $B$ by $\sharp(B)\leq  |S|\sum_{i=1}^k |\partial^1(\Lambda_i)|$. It remains to show $F_{\omega'}(\Lambda)(m)=\sum_{i=1}^k F_\omega(\Lambda_i)(m)$. Here we use a decoupling argument. In the configuration $\omega'$ we have for arbitrary $i\in\{1,\dots,k\}$ and $x\in \Lambda_i$ that
$C_x^\Lambda(\omega')\cap \Lambda_j =\emptyset$ for all $j\neq i$ as all the edges in $B$ are non-active in $\omega'$. Therefore,
\[
F_{\omega'}(\Lambda)(m)=\sum_{i=1}^k F_{\omega'}(\Lambda_i)(m)=\sum_{i=1}^k F_\omega(\Lambda_i)(m)
\]
holds true, where the last equality follows from the fact that $e\in E_{\omega'} \Leftrightarrow e\in E_\omega$ for all $e\in E|_{\Lambda_i}$, $i\in\{1,\dots,k\}$. The claim follows.
\end{proof}
Note that by setting $\cA:=\{0,1\}^S $ and $\cC_\omega: V\to \cA$, $\cC_\omega(v):= (\omega_{v,s})_{s\in S}$, each $\omega\in\Omega$ may be interpreted as coloring of the vertices. It is obvious that for each $\omega\in\Omega$ the function $F_\omega(\cdot)$ is $\cC_\omega$-invariant, i.e. for any $Q,U\in\cF(G)$ the equivalence of the patterns $\cC_\omega\vert_Q$ and $\cC_\omega\vert_U$ implies $F_\omega(Q)=F_\omega(U)$, cf.\@ Definition~\ref{defi:AA2}.

In order to apply Theorem \ref{thm:ET} it remains to prove the existence of the frequencies. Here we use the notation
\[
 \sharp_P (\cC_\omega|_{\Lambda}):=\sharp \left\{ x\in G\mid  \tilde\Lambda x \subseteq\Lambda, P(y)=\cC_\omega(yx) \text{ for all }y\in \tilde\Lambda \right\},
\]
where $\Lambda,\tilde\Lambda\in\cF(G)$, $P:\tilde\Lambda\to\cA$ and $\omega\in\Omega$.
\begin{Lemma}\label{la:freq}
 There exists a set $\tilde\Omega\subseteq \Omega$ of full measure such that for all $\omega\in\tilde\Omega$ and all $P:\Lambda\to A$, $\Lambda\in\cF(G)$ the limit
\[
 \nu_P:= \lim_{n\to\infty}\frac{\sharp_P (\cC_\omega|_{\Lambda_n})}{|\Lambda_n|} 
= \prod_{v\in \Lambda}\prod_{s\in S}\left( p_s \delta_{1,(P(v))_s} + (1-p_s) \delta_{0,(P(v))_s} \right)
\]
exists and does not depend on $\omega$. Here $\delta_{i,j}$ denotes the Kronecker delta.
\end{Lemma}
 This lemma is basically an application of Lindenstrauss' pointwise ergodic theorem (cf.\@ \cite{Lindenstrauss-01}) in the special case quoted in Theorem \ref{thm:linde}. The proof of Lemma \ref{la:freq} follows the lines of the proof of Theorem \ref{theorem:freq} using the mapping
\begin{equation}\label{def:T}
T:G\times \Omega\to\Omega,\quad (g,\omega)\mapsto T_g\omega, \quad \text{where}\quad (T_g\omega)_{v,s}=\omega_{vg,s}.
\end{equation}
Note that $T$ is an ergodic measure preserving left action on the probability space.
% \begin{Theorem}\label{thm:lindenstrauss}
%   Let $G$ act from the left on a measure space $(\Omega,\cZ,\PP)$ by an ergodic measure preserving transformation $T$ and let $(Q_n)$ be a tempered F\o{}lner sequence. Then for any $ f\in L^1(\PP)$
%   \[
%   \lim_{n\rightarrow\infty}\frac{1}{\vert Q_n \vert}\sum_{g\in Q_n}f(T_g \omega ) =\int_\Omega f(\omega) d \PP(\omega)
%   \]
%   holds $\PP$-almost surely.
% \end{Theorem}

The number of clusters per vertex is defined by
\[
 \kappa := \EE(|C_{\id}|^{-1})=\sum_{n=1}^\infty \frac{1}{n}\PP(|C_{\id}|=n).
\]
Note that $\kappa\in(0,1]$ as we assumed $p\in[0,1)^S$. The following Lemma was proven in \cite{Grimmett-99} for $G=\ZZ^d$.
\begin{Lemma}\label{la:kappa}
 Let $G$ be an amenable group, $(\Lambda_n)$ a tempered F\o lner sequence and $(\Omega,\cZ,\PP)$ the same probability space as above. Then
\[
 \lim_{n\to\infty}\frac{K_\omega(\Lambda_n)}{|\Lambda_n|}= \kappa
\]
holds true for $\PP$-almost all $\omega$.
\end{Lemma}
\begin{proof}
We adapt the proof of Theorem 4.2 in \cite{Grimmett-99} to the case of amenable groups. For $x\in G$, $\omega\in \Omega$ and $n\in\NN$ set
\[
 \gamma_x(\omega):=\begin{cases}
             |C_x(\omega)|^{-1},& \text{if } |C_x(\omega)|<\infty\\ 0,&\text{else}
            \end{cases}
\quad\text{and}\quad
\gamma_x^{(n)}(\omega):=|C_x^{\Lambda_n}(\omega)|^{-1}.
\]
As $|C_x(\omega)|\geq |C_x^{\Lambda_n}(\omega)|$, this obviously yields
\begin{equation}\label{eq:la:kappa1}
 \gamma_x(\omega)\leq \gamma_x^{(n)}(\omega).
\end{equation}
Furthermore, since for each $y\in\Lambda_n$, one has 
\[
 \sum_{x\in C_y^{\Lambda_n}(\omega)} \gamma_x^{(n)}(\omega)=|C_y^{\Lambda_n}(\omega)|\gamma_y^{(n)}(\omega)=1,
\]
we obtain
\[
 K_\omega(\Lambda_n)=\sum_{x\in\Lambda_n}\gamma_x^{(n)}(\omega).
\]
This, together with \eqref{eq:la:kappa1}, gives
\[
 \frac{1}{|\Lambda_n|}K_\omega(\Lambda_n)\geq \frac{1}{|\Lambda_n|}\sum_{x\in \Lambda_n} \gamma_x(\omega).
\]
Note that with $T:\Omega\to\Omega$ given as in \eqref{def:T} we have $\gamma_x(\omega)=\gamma_{\id}(T_x\omega)$. Therefore, with the Lindenstrauss ergodic theorem, we arrive at
\[
 \frac{1}{|\Lambda_n|}\sum_{x\in \Lambda_n}\gamma_x(\omega)= \frac{1}{|\Lambda_n|}\sum_{x\in \Lambda_n}\gamma_{\id}(T_x\omega) \to \EE(\gamma_{\id}) \quad\text{as}\quad n\to \infty
\]
for $\PP$-almost all $\omega\in\Omega$. This proves the following lower bound
\[
 \liminf_{n\to\infty}\left(\frac 1{|\Lambda_n|}K_\omega(\Lambda_n) \right)\geq \EE(\gamma_{\id})=\kappa\quad \PP-\text{a.s.}
\]
It remains to show that $\kappa$ is also an upper bound for $\limsup_n (K_\omega(\Lambda_n)/|\Lambda_n|)$. To this end, we estimate
\begin{align*}
 \sum_{x\in\Lambda_n}\gamma_x^{(n)}(\omega)
&=\sum_{x\in\Lambda_n}\gamma_x(\omega)+ \sum_{x\in\Lambda_n}(\gamma_x^{(n)}(\omega)-\gamma_x(\omega))\\
&\leq \sum_{x\in\Lambda_n}\gamma_x(\omega)+ \sum_{\ato{x\in\Lambda_n}{C_x(\omega)\cap \partial^1(\Lambda_n)\neq\emptyset}}\gamma_x^{(n)}(\omega).
\end{align*}
Here we used that $\gamma_x^{(n)}(\omega)=\gamma_x(\omega)$ whenever $C_x(\omega)$ is a subset of $\Lambda_n$. Note that the second sum is nothing but the number of all clusters $C_x(\omega)$, $x\in\Lambda_n$ which have a non-empty intersection with the boundary $\partial^1(\Lambda_n)$. This number is bounded from above by $|\partial^1(\Lambda_n)|$. In light of that, 
\[
 \frac{1}{|\Lambda_n|}\sum_{x\in\Lambda_n}\gamma_x^{(n)}(\omega)
\leq \frac{1}{|\Lambda_n|}\sum_{x\in\Lambda_n}\gamma_x(\omega) + \frac{|\partial^1(\Lambda_n)|}{|\Lambda_n|}.
\]
We conclude that
\[
 \limsup_{n\to\infty}\frac{1}{|\Lambda_n|}\sum_{x\in\Lambda_n}\gamma_x^{(n)}(\omega)
\leq \limsup_{n\to\infty}\frac{1}{|\Lambda_n|}\sum_{x\in\Lambda_n}\gamma_x(\omega) = \kappa \quad \PP-\text{a.s.},
\]
since $(\Lambda_n)$ is a F\o lner sequence.
\end{proof}
We are now in position to prove Theorem \ref{thm:mainperc}.
\begin{proof}[Proof of Theorem \ref{thm:mainperc}]
 Let $\tilde\Omega\subseteq \Omega$ be given as in Lemma \ref{la:freq} and fix some $\omega\in\tilde\Omega$. We have already seen that $F_\omega:\cF\to \cB(\RR)$ is $\cC_\omega$-invariant and in Lemma \ref{la:aad} we proved that $F_\omega$ is almost additive with boundary term $2|S||\partial^1(\cdot)|$. Therefore, Theorem \ref{thm:ET} yields that there exists a function $F:\cF(G)\to \cB(\RR)$ with
\[
 \lim_{n\to\infty}\left\Vert \frac{F_\omega(\Lambda_n)}{\vert \Lambda_n\vert}- F \right\Vert = 0
\]
where $\|\cdot\|$ denotes the supremum norm. 
Thus, we get by Lemma \ref{la:kappa},
 \[
  \lim_{n\to\infty}\left\Vert \frac{F_\omega(\Lambda_n)}{ K_\omega(\Lambda_n)}- \Phi \right\Vert 
= \lim_{n\to\infty}\left\Vert \frac{F_\omega(\Lambda_n)}{\vert \Lambda_n\vert}\frac{|\Lambda_n|}{K_\omega(\Lambda_n)}- \Phi \right\Vert
= 0,
 \]
where $\Phi=\frac 1 \kappa F$. Using uniform convergence, it is easy to prove that $\Phi$ is a distribution function of a probability measure. 
% To show this, we note first that for $m\leq \tilde m$ one has
% \[
%  \Phi(\tilde m)-\Phi(m)=\lim_{n\to\infty}\frac{F_\omega(\Lambda_n)(\tilde m)-F_\omega(\Lambda_n)(m)}{\kappa |\Lambda_n|}\geq 0
% \]
% since each $F_\omega(\Lambda_n)$ is monotonically non-decreasing. To obtain the right-continuity of $\Phi$, one needs to show $\Phi(m)=\lim_{\tilde m\searrow m} \Phi(\tilde m)$ for all $m\in \RR$. This is clear from
% \begin{align*}
%  \lim_{\tilde m\searrow m} \Phi(\tilde m)
% = \lim_{\tilde m\searrow m}\lim_{n\to\infty} \frac{F_\omega(\Lambda_n)(\tilde m)}{\kappa|\Lambda_n|}
% = \lim_{n\to\infty}\lim_{\tilde m\searrow m} \frac{F_\omega(\Lambda_n)(\tilde m)}{\kappa|\Lambda_n|}
% = \lim_{n\to\infty} \frac{F_\omega(\Lambda_n)(m)}{\kappa|\Lambda_n|}=\Phi(m),
% \end{align*}
% where we used the uniform convergence to interchange the limits as well as the right-continuity of the functions $F_\omega(\Lambda_n)$. Applying again the uniform convergence of the approximants gives
% \begin{equation} \label{eq:limF}
%  \lim_{m\to\infty}  F(m) 
% = \lim_{m\to\infty} \lim_{n\to\infty}\frac{F_\omega(\Lambda_n)(m)}{|\Lambda_n|}
% = \lim_{n\to\infty}\lim_{m\to\infty}\frac{F_\omega(\Lambda_n)(m)}{|\Lambda_n|}
% = \lim_{n\to\infty} \frac{K_\omega(\Lambda_n)}{|\Lambda_n|} =\kappa,
% \end{equation}
% which implies $\lim_{m\to\infty}\Phi(m)=1$. The fact that $\lim_{m\to -\infty}\Phi(m)=0$ is obvious.
\end{proof}
\begin{Corollary}
Let $G$ be an amenable group, $(\Lambda_n)$ a tempered F\o lner sequence and $(\Omega,\cZ,\PP)$ the same probability space as above. Then there exists a set of full measure $\tilde\Omega\subseteq\Omega$ such that for each $m\in \NN$ and $\omega\in\tilde\Omega$ the densities of clusters of size $m$ defined as the limits
\[
 c_m:=\lim_{n\to\infty}\frac{1}{K_\omega(\Lambda_n)}\sharp\left\{ C_x^{\Lambda_n}  \mid x\in\Lambda_n,\, \left|C^{\Lambda_n}_x(\omega)\right|= m \right\},
\]
\[
 d_m:=\lim_{n\to\infty}\frac{1}{|\Lambda_n|}\sharp\left\{x\in \Lambda_n \mid \left|C_x(\omega)\right|= m \right\}\ \text{ and }\
d_\infty :=\lim_{n\to\infty}\frac{1}{|\Lambda_n|}\sharp\left\{x\in \Lambda_n \mid \left|C_x(\omega)\right|= \infty \right\}
\]
exist and do not depend on $(\Lambda_n)$ and $\omega\in\tilde\Omega$. Furthermore, the convergence of $c_m$ is uniform in $m$ and $\sum_{m\in\NN} c_m =1$ holds true.
\end{Corollary}

\begin{Remark}
 To define the densities $c_m$, one counts clusters of size $m$ in a certain set and normalizes by the number of all clusters in this set. In contrast to that, to define $d_m$ one counts the vertices in clusters of size $m$ in a certain set and normalizes by the number of all vertices in this set. The distribution function $F$ given by Theorem \ref{thm:mainperc} is associated with the densities $c_m$. Using the uniform convergence of the approximants one can show the uniform existence of the $c_m$ and the fact that the $c_m$ sum up to one. One can also define an associated distribution function $G_\omega(\Lambda):\RR\to \NN_0$ for the densities $d_m$ by setting for given $\omega\in \Omega$ and $\Lambda\in \cF(G)$ 
\begin{equation*}
 G_\omega(\Lambda)(m):=\sharp\left\{   x\in \Lambda \mid |C_x^\Lambda(\omega)|\leq m  \right\}.
\end{equation*}
 However, it is not possible to prove an adapted version of Lemma \ref{la:aad} for the functions $G_\omega(\Lambda)$ and hence one cannot apply the ergodic theorem to obtain uniform convergence. This causes the qualitative difference in the results for the densities $c_m$ and $d_m$.
\end{Remark}

\begin{proof}
 Let us begin with the existence of the density $d_\infty$. To do so we use $|C_x(\omega)|=|C_{\id}(T_x\omega)|$ to obtain
\begin{align*}
 \frac{\sharp\{x\in\Lambda_n\mid |C_x(\omega)|=\infty\}}{|\Lambda_n|}
&= \frac{1}{|\Lambda_n|}\sum_{x\in \Lambda_n}\mathbf 1_{\{x\in\Lambda_n\mid |C_x(\omega)|=\infty\}}(x)\\
&= \frac{1}{|\Lambda_n|}\sum_{x\in \Lambda_n}\mathbf 1_{\{x\in\Lambda_n\mid |C_{\id}(T_x\omega)|=\infty\}}(x).
\end{align*}
The application of Theorem \ref{thm:linde} with $\phi:\Omega\to\{0,1\}$, $\phi(\omega):=\mathbf 1_{\{|C_{\id}(\omega)|=\infty\}}(\omega)$ shows that there exists a set $\Omega'\subseteq\Omega$ of full measure such that the following limits exist for all $\omega\in\Omega'$
\begin{align*}
\lim_{n\to\infty}\frac{\sharp\{x\in\Lambda_n\mid |C_x(\omega)|=\infty\}}{|\Lambda_n|}
&= \lim_{n\to\infty}\frac{1}{|\Lambda_n|}\sum_{x\in \Lambda_n}\phi(T_x\omega)= \EE \phi = \PP(|C_{\id}|=\infty)=:d_\infty.
\end{align*}

Now let $\Omega''\subseteq\Omega$ be the set of full measure such that Theorem \ref{thm:mainperc} holds true for all $\omega\in\Omega''$ and set $\tilde\Omega:=\Omega'\cap\Omega''$. From now on we fix an element $\omega\in\tilde\Omega$. For each $\Lambda\in\cF(G)$, set
\[
 f_\omega(\Lambda):\RR\to\NN_0,\quad  m\mapsto f_\omega(\Lambda)(m):=\sharp\left\{ C_x^\Lambda\mid x\in\Lambda, |C_x^\Lambda|= \lfloor m\rfloor \right\},
\]
\[
 g_\omega(\Lambda):\RR\to\NN_0,\quad m\mapsto g_\omega(\Lambda)(m):=\sharp \left\{ x\in\Lambda \mid  |C_x^\Lambda|= \lfloor m\rfloor \right\}
\]
and $f,g:\RR\to \RR$ by $f(m):=F(m)-F(m-1)$ and $g(m):=m f(m)$, where $F$ is the limit function given by $F=\lim_{n\to\infty}F_\omega(\Lambda_n)/|\Lambda_n|$, see the proof of Theorem \ref{thm:mainperc}.
Let $\epsilon>0$ be given. As $F_\omega(\Lambda_n)/|\Lambda_n|$ converges uniformly to $F$, there is an $n_0\in\NN$ such that 
$$\left|\frac{F_\omega(\Lambda_n)(m)}{|\Lambda_n|}-F(m)\right|\leq \epsilon$$ for all $m\in\RR$ and all $n > n_0$. Therefore,
\begin{align*}
 \left\vert \frac{f_\omega(\Lambda_n)(m)}{|\Lambda_n|}-f(m) \right\vert 
&= \left| \frac{F_\omega(\Lambda_n)(m)-F_\omega(\Lambda_n)(m-1)}{|\Lambda_n|}-(F(m)-F(m-1)) \right|\\
&\leq  \left| \frac{F_\omega(\Lambda_n)(m)}{|\Lambda_n|}- F(m)\right|+\left|\frac{F_\omega(\Lambda_n)(m-1)}{|\Lambda_n|}-F(m-1) \right|
\leq 2\epsilon
\end{align*}
holds true for all $m\in\RR$ and all $n > n_0$. This shows that $f_\omega(\Lambda_n)/|\Lambda_n|$ converges uniformly to $f$. 
In order to prove the existence of the densities $d_m$, fix an arbitrary $m\in\NN$.
Furthermore, let $\epsilon>0$ and choose $n_0$ as above, then we get for all $n\geq n_0$
\[
 \left| \frac{g_\omega(\Lambda_n)(m)}{|\Lambda_n|}-g(m) \right|
= m\left| \frac{f_\omega(\Lambda_n)(m)}{|\Lambda_n|}-f(m) \right|
\leq 2\epsilon m.
\]
Since
\begin{align*}
 &\lim_{n\to\infty}\left|\frac{\sharp\left\{x\in \Lambda_n \mid \left|C_x(\omega)\right|= m \right\}}{|\Lambda_n|} - g(m)\right|\\
%&=\lim_{n\to\infty}\frac{\left|\left\{x\in \Lambda_n \mid \left|C_x(\omega)\right|= m \right\}\right|-g_\omega(\Lambda_n)(m) }{|\Lambda_n|}\\
&=\lim_{n\to\infty}\left|\frac{\sharp\left\{x\in \Lambda_n \mid \left|C_x(\omega)\right|= m \right\}-\sharp\left\{x\in \Lambda_n \mid \left|C_x^{\Lambda_n}(\omega)\right|= m \right\} }{|\Lambda_n|}\right|
&\leq \lim_{n\to\infty}\frac{|\partial^m(\Lambda_n)|}{|\Lambda_n|}=0,
\end{align*}
this proves the existence of the densities $d_m=g(m)$ for each $m\in \NN$. 

Further, given $\epsilon>0$, choose $n_1\in\NN$ such that 
\[
\left|\frac{|\Lambda_n|}{K_\omega(\Lambda_n)}-\frac 1 \kappa\right|\leq \frac \epsilon 2\quad\text{and}\quad \left|\frac{f_\omega(\Lambda_n)(m)}{|\Lambda_n|}-f(m)\right|\leq \frac{\kappa\epsilon}{2}
\]
for all $m\in\RR$ and $n\geq n_1$. Then we have
\[
 \left|\frac{f_\omega(\Lambda_n)(m)}{K_\omega(\Lambda_n)}-\frac 1 \kappa f(m)\right|
\leq  \left|\frac{f_\omega(\Lambda_n)(m)}{|\Lambda_n|}\frac{|\Lambda_n|}{K_\omega(\Lambda_n)}-\frac{f_\omega(\Lambda_n)(m)}{\kappa|\Lambda_n|}\right|+          \left|\frac{f_\omega(\Lambda_n)(m)}{\kappa|\Lambda_n|}- \frac 1 \kappa f(m)\right|\leq \epsilon
\]
for all $m\in\RR$ and $n\geq n_1$, which proves the uniform existence of $c_m$, where $c_m=\frac 1 \kappa f(m)$. Furthermore, we have indeed
\[
 \sum_{m\in\NN}c_m=\frac 1 \kappa \sum_{m\in\NN} f(m) =\frac 1 \kappa \lim_{m\to\infty} F(m) = \frac 1 \kappa \kappa =1.
\]
Here we used $\lim_{m\to\infty}F(m)=\kappa$, which can easily be shown using uniform convergence and interchanging limits.
\end{proof}

%%%%%%%%%%%%%%%%%%%  p \in [0,1)^S

\subsection{Continuous dependence}
In this subsection we study the dependence of the limit distribution $\Phi$ given by Theorem \ref{thm:mainperc} on the percolation parameter $p=(p_s)_{s\in S}\in[0,1)^{S}$. We measure distances in $[0,1)^S$ with the usual $\ell^2$-metric.
To emphasize the dependence on $p$ we use the notation $\PP_p$, $\EE_p$, $\kappa(p)$ and $\Phi(p)$ instead of $\PP$, $\EE$, $\kappa$ and $\Phi$.
For each $\Lambda\in \cF(G)$ we define
\begin{equation}\label{def:barF}
 \bar F(\Lambda):[0,1)^S \to \cB(\RR),\quad p\mapsto \bar F_p(\Lambda),
\end{equation}
where
\[
 \bar F_p(\Lambda)(m) 
= \EE_p \left( F_\omega (\Lambda) (m)\right) 
= \EE_p \left(\sharp\left\{   C_x^\Lambda(\omega)\mid x\in \Lambda, |C_x^\Lambda(\omega)|\leq m  \right\}\right)
\]
for all $m\in \RR$. Note that for fixed $p$ and $\Lambda$ the function $\bar F_p(\Lambda)$ is constant on each interval $[k,k+1)$ for $k\in\ZZ$ and hence in $\cB(\RR)$. Furthermore, we have for arbitrary $m\in\RR$ and $\Lambda\in\cF(G)$
\begin{align*}
 \bar F_p(\Lambda)(m)
= \sum_{k=1}^{|\Lambda|} k\cdot \PP_p\left( \left\{ \omega\in\Omega\mid F_\omega(\Lambda)(m)=k \right\}\right)
= \sum_{k=1}^{|\Lambda|} k \sum_{a\in \{0,1\}^{\Lambda\times S}} \PP_p\left( Z(a) \right)  \mathbf 1_{ M(\Lambda,m,k) }(a),
\end{align*}
where
\[
 M(\Lambda,m,k):=\bigl\{a\in \{0,1\}^{\Lambda\times S} \mid F_\omega(\Lambda)(m) =k   \text{ for some } \omega\in Z(a)  \bigr\}
\]
and $Z(a)=\{\omega\in\Omega \mid \omega_{v,s} = a_{v,s} \text{ for all } v\in \Lambda, s\in S\}$ for $a=(a_{v,s})_{v\in\Lambda,s\in S} \in \{0,1\}^{\Lambda\times S}$. Here we used that for fixed $a,m$ and $\Lambda$ one has $F_\omega(\Lambda)(m)=k$ for all $\omega\in Z(a)$ if and only if $F_\omega(\Lambda)(m)=k$ for some $\omega\in Z(a)$.
Obviously we have
\[
 \PP_p(Z(a))= \prod_{v\in\Lambda}\prod_{s\in S}\left( p_s \delta_{1,a_{v,s}}+ (1-p_s )\delta_{0,a_{v,s}} \right),
\]
where $\delta_{i,j}$ is the Kronecker delta. This shows that $p\mapsto \PP_p(Z(a))$ is a multivariate polynomial in $p_s$, $s\in S$ and hence, it is continuous in $p=(p_s)_{s\in S}$. As mentioned before, we use the $\ell^2$-norm and the induced metric in $[0,1)^S$.
Now it is clear that for each $m\in \RR$ and $\Lambda\in\cF(G)$ the mapping $p\mapsto \bar F_p(\Lambda)(m)$ is continuous as well. Therefore, given $\epsilon>0$, $\Lambda\in\cF(G)$ and $p\in[0,1)^S$ we can find $\delta>0$ such that $|\bar F_p(\Lambda)(m)-  \bar F_{p'}(\Lambda)(m)|\leq \epsilon$ whenever $p'\in[0,1)^S$ with $\Vert p-p'\Vert_2\leq\delta$ for all $m\in\{0,1,\dots |\Lambda|\}$. Hence, for these $p'$ we have $\Vert \bar F_{p}(\Lambda)-\bar F_{p'}(\Lambda)\Vert\leq \epsilon$, where we $\|\cdot\|$ denotes the supremum norm in $\cB(\RR)$. This proves the continuity of the function $p\mapsto \bar F_p(\Lambda)$. This shows that $\Lambda\mapsto\bar F(\Lambda)$ maps elements of $\cF(G)$ into the space of all continuous functions mapping from $[0,1)^S$ to $\cB(\RR)$. We write
\[
 \cC([0,1)^S,\cB(\RR))=\{\phi:[0,1)^S\to \cB(\RR)\mid \phi \text{ is continuous}\}
\]
and we equip this space with the supremum norm 
$$\Vert\phi\Vert_\infty=\sup_{p\in[0,1)^S}\Vert \phi(p)\Vert=\sup_{p\in[0,1)^S}\sup_{m\in \RR} \vert \phi(p)(m)\vert .$$ 
Note that with this norm $\cC([0,1)^S,\cB(\RR))$ is a Banach space. 
Our next goal is to study the limit of $\bar F(\Lambda_n)/|\Lambda_n|$ for some F\o lner sequence $(\Lambda_n)$.

The following result has been shown by {\sc Grimmett} for $\ZZ^d$ in \cite{Grimmett-76}. We generalize these ideas to the case of amenable groups, using our ergodic theorem from Section~5.
\begin{Lemma}\label{la:kappacont}
 The mapping $\kappa:[0,1)^S\to (0,1]$, $p\mapsto \kappa(p)=\EE_p(|C_{\id}|^{-1})$ is continuous.
\end{Lemma}
\begin{proof}
 Let $(\Lambda_n)_{n\in\NN}$ be a tempered F\o lner sequence.
  By Lemma \ref{la:kappa}, for each $p\in[0,1)^S$ there is a set $\Omega_p$ with $\PP_p(\Omega_p)=1$ and $\kappa(p)=\lim_{n\to \infty} K_\omega(\Lambda_n)/|\Lambda_n|$ for all $\omega\in\Omega_p$, which gives
 \[
  \limsup_{n\to\infty}\frac{K_\omega(\Lambda_n)}{|\Lambda_n|}= \kappa(p) = \liminf_{n\to\infty}\frac{K_\omega(\Lambda_n)}{|\Lambda_n|}\quad \PP_p\text{-a.s.}
 \]
 Fatou's Lemma implies
 \[
  \limsup_{n\to\infty}\frac{\EE_p(K_\omega(\Lambda_n))}{|\Lambda_n|}\leq \kappa(p) \leq \liminf_{n\to\infty}\frac{\EE_p(K_\omega(\Lambda_n))}{|\Lambda_n|}
 \]
 and hence $\kappa(p)=\lim_{n\to\infty} \EE_p(K_\omega(\Lambda_n))/|\Lambda_n|$. 
Now we show that this limit exists even uniformly in $p$ and that the limit function is continuous.
To do so, set
\[
 H:\cF(G)\to \cB([0,1)^S),\quad H(\Lambda)(p):= \EE_p (K_\omega(\Lambda))= \EE_p \left(\sharp \{C_x^\Lambda(\omega)\mid x\in \Lambda \} \right),
\]
where $\cB([0,1)^S)$ denotes the Banach-space of bounded and continuous functions mapping from $[0,1)^S$ to $\RR$,
equipped with the supremum norm. 
Note that the mapping $p\mapsto\EE_p(K_\omega(\Lambda))$ is in $\cB([0,1)^S)$ since 
\[
 \EE_p(K_\omega(\Lambda))=\sum_{k=1}^{|\Lambda|} k\cdot\PP_p(\{\omega\in\Omega\mid K_\omega(\Lambda)=k\})
\]
and $\PP_p(\{\omega\in\Omega\mid K_\omega(\Lambda)=k\})$ is a polynomial in $p_s$, $s\in S$. By the translation invariance of the measure, 
we have $H(\Lambda)=H(\Lambda z)$ for all $z\in G$ and $\Lambda \in \cF(G)$, 
hence $H$ is $\cC_\omega$-invariant for each $\omega\in\Omega$. Moreover, we claim that for 
disjoint $\Lambda_1,\dots,\Lambda_k\in\cF(G)$, we have
\begin{equation}
 \left\Vert H(\Lambda) - \sum_{i=1}^k H(\Lambda_i)  \right\Vert \leq 2|S|\sum_{i=1}^k |\partial (\Lambda_i)|,
\end{equation}
where $\Lambda=\bigcup_{i=1}^k \Lambda_i$. This is true since for any $p\in [0,1)^S$,
\begin{align*}
 \left| \EE_p(K_\omega(\Lambda)) -\sum_{i=1}^k \EE_p(K_\omega(\Lambda_i)) \right| 
\leq \EE_p\left(\left| K_\omega(\Lambda) - \sum_{i=1}^k K_\omega(\Lambda_i) \right|\right)
\leq   2|S|  \sum_{i=1}^k \left|\partial (\Lambda_i) \right|.
\end{align*}
Note that the latter inequality holds true since $K_\omega(\Lambda)=F_\omega(\Lambda)(|\Lambda|)$ and by Lemma \ref{la:aad}. Therefore, Theorem \ref{thm:ET} yields that the functions $H(\Lambda_n)/|\Lambda_n|$ converge uniformly to $\kappa\in \cB([0,1)^S)$ as $n$ tends to infinity.
\end{proof}
The main result of this subsection is the following.
\begin{Theorem}
 Let $(\Lambda_n)$ be a tempered F\o lner sequence and let for each $p\in [0,1)^S$ the function $\Phi_p\in\cB(\RR)$ be the limit given by Theorem \ref{thm:mainperc}. Then the function $\Psi:[0,1)^S\to \cB(\RR)$, $p\mapsto \Phi_p$ is continuous.
\end{Theorem}
\begin{proof}
Let $\bar F:\cF(G)\to \cC([0,1)^S,\cB(\RR))$ be given as in \eqref{def:barF}. Furthermore, let $\Lambda_1,\dots,\Lambda_k\in \cF(G)$ be disjoint and set $\Lambda=\bigcup_{i=1}^k \Lambda_i$. Then for any $p\in[0,1)^S$ and $m\in\RR$ we have
\begin{align*}
  \left|\bar F_p(\Lambda)(m)-\sum_{i=1}^k\bar F_p(\Lambda_i)(m)\right|
\leq  \EE_p \left(\left| F_\omega(\Lambda)(m) - \sum_{i=1}^k F_\omega(\Lambda_i)(m) \right|\right)
\leq   2 |S| \sum_{i=1}^k \partial(\Lambda_i) ,
\end{align*}
where the last inequality follows from Lemma \ref{la:aad}. Therefore,
\[
\left\Vert \bar F(\Lambda)- \sum_{i=1}^k\bar F(\Lambda_i)\right\Vert_\infty
= \sup_{p\in[0,1)^S}\sup_{m\in\RR}\left\vert F_p(\Lambda)(m)-\sum_{i=1}^k \bar F_p(\Lambda_i)(m) \right\vert
\leq 2 |S| \sum_{i=1}^k \partial(\Lambda_i).
\]
Besides, we have $\bar F(\Lambda)=\bar F(\Lambda x)$ for arbitrary $\Lambda\in \cF(G)$ and $x\in G$ which allows us to apply Theorem \ref{thm:ET}. This proves the existence of a function $\tilde F\in \cC([0,1)^S,\cB(\RR))$ with
\[
 \lim_{n\to\infty}\left\Vert \frac{\bar F(\Lambda_n)}{|\Lambda_n|}-\tilde F \right\Vert_\infty=0.
\]
We claim that $\tilde F(p)=\kappa(p)\Phi_p$ for all $p\in[0,1)^S$. If this holds true, the proof is completed since $\kappa:[0,1)^S\to(0,1]$ is continuous by Lemma \ref{la:kappacont}.
To prove the claim, let $p\in[0,1)^S$ be given and calculate 
\begin{align*}
 \left\Vert \frac{F_\omega(\Lambda_n)}{|\Lambda_n|}-\kappa(p)\Phi_p \right\Vert
&\leq
 \left\Vert \frac{F_\omega(\Lambda_n)}{|\Lambda_n|}-\frac{K_\omega( \Lambda_n)}{|\Lambda_n|}\Phi_p  \right\Vert 
+\left\Vert \frac{K_\omega( \Lambda_n)}{|\Lambda_n|}\Phi_p   - \kappa(p)\Phi_p \right\Vert \\
&\leq
 \left|\frac{K_\omega(\Lambda_n)}{|\Lambda_n|} \right|\cdot \left\Vert \frac{F_\omega(\Lambda_n)}{K_\omega(\Lambda_n)}-\Phi_p  \right\Vert 
+  \left\vert \frac{K_\omega( \Lambda_n)}{|\Lambda_n|}   - \kappa(p) \right\vert \cdot \left\Vert \Phi_p \right\Vert,
\end{align*}
which implies
\[
 \lim_{n\to\infty}\left\Vert \frac{F_\omega(\Lambda_n)}{|\Lambda_n|}-\kappa(p)\Phi_p \right\Vert =0\quad\quad\quad \PP_p\text{-a.s.}
\]
Note that the norm we used here is the supremum norm in the space $\cB(\RR)$. Now, by the Lebesgue convergence theorem, we obtain 
\begin{align*}
 0=\EE_p\left(\lim_{n\to\infty}\left\Vert \frac{F_\omega(\Lambda_n)}{|\Lambda_n|}-\kappa(p)\Phi_p \right\Vert \right)
&=\lim_{n\to\infty} \EE_p\left(\left\Vert \frac{F_\omega(\Lambda_n)}{|\Lambda_n|}-\kappa(p)\Phi_p \right\Vert \right)\\
&\geq \limsup_{n\to\infty} \left\Vert \frac{\EE_p\left(F_\omega(\Lambda_n)\right)}{|\Lambda_n|}-\kappa(p)\Phi_p  \right\Vert 
\end{align*}
and hence
\[
 \lim_{n\to\infty} \left\Vert \frac{\bar F_p(\Lambda_n)}{|\Lambda_n|}-\kappa(p)\Phi_p  \right\Vert =0.
\]
This finishes the proof of the claim.
\end{proof}

%% file: appendix.tex
\section{Appendix} \label{sec:appendix}

For the convenience of the reader, we give the proofs of the Lemmas~\ref{lemma:folner} and~\ref{lemma:ow4} in this appendix. As the corresponding elaborations are rather technical, we did not put them into the main body of the present paper.

\begin{proof}[Proof of Lemma~\ref{lemma:folner}]
In order to prove (a) let $G$ be amenable, second countable and unimodular. We denote by $\{V_n\}$ an enumeration of the countable base of the topology of $G$. Since $G$ is locally compact, we can choose the $V_n$ to be pre-compact. We now set $K_n := \cup_{j=1}^n \overline{V}_j$. Then each $K_n$ is compact and we have $K_n \subseteq K_{n+1}$ for $n \geq 1$, as well as $\cup_n K_n = G$. %Hence, $G$ is $\sigma$-compact. \\

Let $K \subseteq G$ be compact. We claim that there is some $M \in \NN$ such that $K \subseteq K_M$. For the proof, note first that for any $g \in G$, there is some $n(g) \in \NN$ such that $g \in V_{n(g)}$. Hence the union $\cup_{g \in K} V_{n(g)}$ is an open cover of $K$. Since $K$ is compact there must be a finite subcover $K \subseteq \cup_{j=1}^m V_{n(g_j)} \subseteq \cup_{j=1}^m \overline{V}_{n(g_j)}$. The latter union denotes a compact set and by construction of the $K_n$, we have
$K \subseteq \cup_{j=1}^m \overline{V}_{n(g_j)} \subseteq K_M$, where $M:=\max \{n(g_j)\mid j=1,\dots,m\}$.

Take a sequence $(\varepsilon_n)$ of positive numbers converging to 0. By the above mentioned statement statement in \cite{OrnsteinW-87}, we find for each $n \in \NN$ a compact set $F_n$ such that
\begin{align*}
\frac{|\partial_{K_n}(F_n)|}{|F_n|} < \varepsilon_n.
\end{align*}
Hence, we conclude that for all $n \geq M$, one obtains with $K \subseteq K_M \subseteq K_n$ that
\begin{align*}
\frac{|\partial_K(F_n)|}{|F_n|} 
\leq \frac{|\partial_{K_M}(F_n)|}{|F_n|} 
\leq \frac{|\partial_{K_n}(F_n)|}{|F_n|} < \varepsilon_n.
\end{align*}
So, clearly $\lim_{n \rightarrow \infty} |\partial_K(F_n)|/|F_n| = 0$.

Now assume that $G$ is arbitrary and unimodular and $(F_n)$ is some strong F\o lner sequence. To show (b) it is enough to verify that for each compact $F,K\subseteq G$ one has 
\[
 F \triangle KF \subseteq \partial_{K \cup K^{-1} \cup \{\id \}}(F),
\]
since $K\cup K^{-1}\cup \{\id \}$ is a compact set. To this end assume first that $g \in KF \setminus F$. Then $g = kf$ with $k \in K$ and $f \in F$, but $kf \notin F$. Define $L_K := K \cup K^{-1} \cup \{\id \}$. Since $k^{-1} \in L_K$, one has $L_Kg \cap F \neq \emptyset$. As $\id \in L_K$, we derive $L_Kg \cap (G \setminus F) \neq \emptyset$, hence $g \in \partial_{L_K}(F)$. 
Now assume that $h \in F \setminus KF$. Then for all $f \in F$ and all $k \in K$, $h \neq kf$, which implies $k^{-1}h \notin F$ for all $k \in K$. It follows from $K^{-1} \subseteq L_K$ that $L_Kh \cap (G\setminus F) \neq \emptyset$. Since $\id \in L_K$ and $h \in F$, we also have $L_Kh \cap F \neq \emptyset$ and thus, $h \in \partial_{L_K}(F)$.

For part (c) of proof let $(F_n)$ be a weak F\o lner sequence in a unimodular group $G$. Furthermore, let a compact $K\subseteq G$ and $\epsilon>0$ be given. As $(F_n)$ is a weak F\o lner sequence, there exists some $n\in \NN$ such that
\[
 \frac{|K^{-1}F_n\setminus F_n| }{|F_n|} \leq \frac{|K^{-1}F_n \triangle F_n|}{|F_n|} \leq \epsilon.
\]
As for each $k\in K$ we have $|F_n\setminus k F_n|=|k^{-1}F_n\setminus F_n|\leq |K^{-1}F_n\setminus F_n|$, the properties (i) and (ii) of Definition \ref{defi:AMENABLE} hold with $K_0=K$.

To show (d) we assume that $G$ is countable and $|\cdot|$ denotes the counting measure. Let $T$ and $K$ be arbitrary compact sets and set $L_K:=K\cup K^{-1}\cup \{\id\}$. In this situation it is enough to prove
\[
 \partial_K(T)\subseteq \partial_{L_K}(T) \subseteq L_K(T\triangle L_K T),
\]
since then $|\partial_K(T)|\leq |L_K||T \triangle L_KT|$. Here the first inclusion holds by Lemma \ref{prop:prop}. To see the second inclusion, take some $g \in G$ such that $L_Kg$ intersects non-trivially both $T$ and $G \setminus T$. By the symmetry of $L_K$, we have $g \in L_K^{-1}T = L_KT$. If $g \notin T$, then $g \in L_KT \setminus T$ and since $\id \in L_K$, we prove the claim for this case. 
If $g \in T$, find some $k \in L_K$ such that $kg \in L_KT \setminus T$, which exists since $L_Kg\cap (G\setminus T)$ is non-empty. Again by the symmetry of $L_K$, we have $g \in L_K(L_KT \setminus T)$, which proves part (d) of the Lemma, as we have shown $g\in L_K(T\triangle L_KT)$ in both cases.

Now we prove (e). Let $(F_n)$ be a strong F\o lner sequence in $G$. We choose some $x\in F_1$ and set $T_1:=F_1x^{-1}$, then we proceed inductively. If $T_1,\dots, T_k$ are chosen, then there is an $n\in \NN$ such that $F_n$ is $(T_k,1)$-invariant. As $\id\in T_k$ we have $F_n\setminus \partial_{T_k}(F_n)\subseteq \{g\in F_n\mid T_k g\subseteq F_n\}=:S$ which gives with $|F_n\setminus \partial_{T_k}(F_n)|\geq |F_n|-|\partial_{T_k}(F_n)|>0$ that $S$ has positive Haar measure. Hence $S$ is non-empty and we take some $g\in S$ and define $T_{k+1}:=F_ng^{-1}$. This procedure gives a sequence $(T_n)$ which is by construction nested and which is a strong F\o lner sequence as it is up to shifts a subsequence of $(F_n)$. Note that here we used unimodularity of $G$ and (vi) of Lemma \ref{prop:prop} which gives that shifts do not change the measure of the $K$-boundary.

Statement (f) was shown in \cite{Lindenstrauss-01} for weak F\o{}lner sequences. By part (b) this holds for strong F\o{}lner sequences as well.

\end{proof}

Before we turn to the proof of Lemma~\ref{lemma:ow4}, we need two short auxiliary lemmas, which have already been proven in \cite{OrnsteinW-87}, Section I.3.

\begin{Lemma} \label{lemma:ow1}
Let $G$ be a unimodular group and assume that $K$ is a non-empty, compact set in $G$ containing the unit element $\id$ and let $T \subseteq G$ be $(K, \delta)$-invariant. Then for the set
\begin{eqnarray*}
S:= \{g \in G \,|\, Kg \subseteq T\},
\end{eqnarray*}
the following statements hold true:
\begin{enumerate}[(i)]
\item $|S| \geq (1- \delta)|T|$,
\item $\int_{S} \one_{Kc}(g) \, dc \leq |K|$ for all $g \in G$.
\end{enumerate} 
\end{Lemma}

\begin{proof}
Note first that $S = T \setminus \partial_K(T)$ such that by the fact that $\id \in K$, (i) is satisfied. Further, consider translates $Kc$ for $c \in  S$. Then we derive the following formula which proves the second statement:
\begin{eqnarray*} 
%& & \int_{\tilde{F}} |Kc| \, dc = \int_{\tilde{F}} \int_G \one_{Kc}(g) \,dg\, dc = %|\tilde{F}|\,|K|, \label{eqn:ow5A} \\
\int_{S} \one_{Kc}(g)\, dc = \int_{S} \one_K(gc^{-1}) dc \leq \int_G \one_K(h)\, dh = |K|. %\label{eqn:ow5B}.
\end{eqnarray*} 
\end{proof}

\begin{Lemma} \label{lemma:ow2}
Let $G$ be a unimodular group. Let $K,S,T \subseteq G$ be non-empty compact sets, where $\id \in K$, $|T|>0$ and $|S|/|T|\geq 1-\delta$  for some $0<\delta<1$. Then for all Borel sets $A \subseteq G$ with finite Haar measure there is some $c \in S$ such that
\begin{equation}\label{eq:lemma:ow2}
|Kc \cap A| \leq \frac{|A|\, |K|}{|T|(1-\delta)}.
\end{equation}
\end{Lemma}

\begin{proof}
Let $A$ be a subset of $G$ with finite Haar measure. Assume that for no $c \in S$ the Inequality \eqref{eq:lemma:ow2} is satisfied. In this case we get
\begin{equation} \label{eqn:contradict}
\int_S |Kc\cap A| dc > \int_S \frac{|A|\, |K|}{|T|(1-\delta)} dc = \frac{|A|\, |K|\, |S|}{|T|(1-\delta)} \geq |A|\, |K|,  
\end{equation}
where the last inequality is due to the fact that $|S|/|T|\geq 1-\delta$.
However, like in the proof of Lemma \ref{lemma:ow1}, we obtain
\begin{align*}
\int_S |Kc\cap A| dc 
= \int_S \int_A \one_{Kc}(g)\,dg \,dc 
= \int_A \int_S \one_{K}(gc^{-1})\,dc \,dg
\leq \int_A \int_G \one_{K}(h)\,dh \,dg
= |A|\, |K|
%\int_G \one_A(g) \left( \int_{S} \one_{Kc}(g) \, dc \right) \, dg \leq |K|\,|A|,
\end{align*}
which clearly is a contradiction to the strict Inequality (\ref{eqn:contradict}). Thus, we find $c\in S$ such that \eqref{eq:lemma:ow2} holds and our statement is proven.
\end{proof}

For the proof of Lemma~\ref{lemma:ow4}, we will use the notion of maximal $\epsilon$-disjointness. Let $\cP$ be a property which a subset of group $G$ can obey, let $I$ be some index set, $J\subseteq I$ and $\{K_i\}_{i\in I}$ a family of subsets of $G$. The family $\{K_i\}_{i\in J}$ is called \emph{maximal $\epsilon$-disjoint} with property $\cP$, if $\{K_i\}_{i\in J}$ is $\epsilon$-disjoint and each $K_i$ satisfies $\cP$, however for each $j\in I\setminus J$ such that $K_j$ satisfies $\cP$, the family $\{K_i\}_{i\in J \cup j}$ is no longer $\epsilon$-disjoint.
A family of \emph{maximal disjoint} sets with property $\cP$ is defined analogously. In our examples the property $\cP$ will be ``being a translate of a certain set'' or/and ``being a subset of a certain set''. We use for instance the term maximal $\epsilon$-disjoint family of translates of $K$ contained in $T$, where $K,T\subseteq G$.

Finally, we have everything together to prove Lemma~\ref{lemma:ow4}.
\begin{proof}[Proof of Lemma~\ref{lemma:ow4}]
We start the proof with
a simple calculation to estimate the proportion $|K|/|T|$. For each $g\in\partial_K(T)$ and $t\in K$ we have $tg\in \partial_{KK^{-1}}(T)$, which immediately gives $|K|\leq |\partial_{KK^{-1}}(T)|$. This implies
\begin{equation}\label{eq:lemma:ow4:1}
 \frac{|K|}{|T|}\leq \frac{|\partial_{KK^{-1}}(T)|}{|T|}<\delta
\end{equation}
as $T$ is $(KK^{-1},\delta)$-invariant.

 Now we formulate the following claim: If $c_j \in T$, $j=1,\dots,n$ are elements which fulfill conditions (i)-(iii) and
\[
 \biggl|\bigcup_{j=1}^n Kc_j\biggr| < \epsilon (1-2\delta) |T|,
\]
then there exists some $c_{n+1}\in T$ such that (i)-(iii) still hold for $c_j$, $j=1,\dots, n+1$.

We postpone the proof of the claim and we assume for the moment that it holds. Then we start with some maximal disjoint family $\{Kc_j\}_{j=1}^n$ of translates of $K$ contained in $T$ with $n|K|\leq (\epsilon+\delta)|T|$ and set $K_j:=K$, $j=1,\dots,n$. Then obviously (i)-(iii) hold. If 
\[
\biggl| \bigcup_{j=1}^n Kc_j  \biggr| \geq \varepsilon(1-2\delta)  |T|,
\] 
then we are done with the proof since $\epsilon\leq 1/2$.
Otherwise we apply the claim and get some $c_{n+1}\in T$ such that conditions (i)-(iii) are still fulfilled for $c_j$, $j=1,\dots,n+1$.
Beside this we have by \eqref{eq:lemma:ow4:1}
\begin{align*}
 \biggl|\bigcup_{j=1}^{n+1}Kc_{j}\biggr|
\leq \varepsilon(1-2\delta) |T| + \delta|T|
\leq (\epsilon+\delta)|T|.
\end{align*}
 If also the first inequality in condition (iv) is satisfied for $c_1,\dots,c_{n+1}$ we are done, if not, we apply the claim again. This procedure will end after finitely many steps since $T$ has finite measure and after each iteration we cover at least $(1-\epsilon)|K|$ more than before. Thus it remains to prove the claim.

Let $c_j\in T$, $j=1,\dots,n$ be such that (i)-(iii) hold with sets $K_j$, $j=1,\dots,n$ and $|A| < \epsilon (1-2\delta)|T|$, where $A:=\bigcup_{j=1}^n Kc_j$. We set $S := \{g \in T \,|\, Kg \subseteq T\}$ and
\[
 U := \left\{ g \in S \, \Big| \, \frac{|Kg \cap \partial_B(A)|}{|K|} \leq \zeta \right\}.
\]

By definition of $U$, it follows with $T\setminus U \subseteq (T\setminus S) \cup (S \setminus U)$ that
\begin{align*}
\frac{|T \setminus U|}{|T|} \leq \frac{|T \setminus S |}{|T|} + \frac{|S \setminus U|}{|T|} 
\leq \delta + \int_{S} \frac{\one_{S\setminus U}(g)}{|T|} \, dg 
\leq \delta + \int_{S} \frac{|Kg \cap \partial_B(A)|}{\zeta |T|\, |K|} \, dg 
\end{align*}
and further, we use 
%$|S| \geq (1-\delta)|T|$, which follows from Lemma \ref{lemma:ow1} part (i), 
Fubini's theorem and Lemma \ref{lemma:ow1} part (ii) in the second inequality to obtain
\begin{align*}
\int_{S} \frac{|Kg \cap \partial_B(A)|}{\zeta |T|\, |K|} \, dg 
\leq \frac{1}{\zeta |T| |K|} \int_{S}\int_{\partial_B(A)} \one_{Kg}(h) \, dh \, dg 
\leq  \frac{|\partial_B(A)| }{\zeta |T|}.
\end{align*}
Clearly, the maximal number $n$ of translates of $K$ that can belong to $A$ is bounded by $|T|/[(1-\varepsilon)|K|]$ such that we arrive at
\begin{align*}
\frac{|T \setminus U|}{|T|} 
\leq  \delta +  \frac{n\,| \partial_B(K)|}{\zeta |T|}  
\leq  \delta +  \frac{|\partial_B(K)|}{\zeta (1-\epsilon)|K|}  
\leq \delta + \frac{\zeta}{(1-\varepsilon)}
\leq 2\delta,
\end{align*}
where the last inequality follows from $\varepsilon, \delta < 1/2$ and $\zeta < \delta/2 $. This yields $|U|/|T|\geq 1-2\delta$  which allows us to apply Lemma \ref{lemma:ow2} to find some $c_{n+1} \in U$ such that
\begin{eqnarray*}
|Kc_{n+1} \cap A| \leq \frac{|A||K|}{|T|(1-2\delta)}  < \epsilon |K|.
\end{eqnarray*} 
and hence condition (ii) holds for $c_j$, $j=1,\dots,n+1$. As $c_{n+1}\in S$ we have $Kc_{n+1}\subseteq T$ which gives (i). We set $K_{n+1}:= \left(Kc_{n+1} \setminus A \right)c_{n+1}^{-1}$ then by the above inequality we get $|K_{n+1}|\geq (1-\epsilon)|K|$. 
%To obtain the validity of the remaining part of the statement (iii) we make use of the inclusions which follow from elementary use of the %definition of the $B$-boundary $\partial_B(\cdot)$
%\begin{align*}
% \partial_B(K\setminus Ac_{n+1}^{-1}) &\subseteq  ( \partial_B(K\setminus Ac_{n+1}^{-1})\cap K) \cup ( \partial_B(K\setminus Ac_{n+1}^{-1})\cap %G\setminus K)\\
%&\subseteq ((\partial_B(K)\cup \partial_B(Ac_{n+1}^{-1}))\cap K)\cup \partial_B(K) \\
%&\subseteq (\partial_B(Ac_{n+1}^{-1})\cap K)\cup \partial_B(K).
%\end{align*}
Thus, with the statement (viii) of Lemma \ref{prop:prop} and with $c_{n+1}\in U$, we have
\begin{align*}
|\partial_B(K_{n+1})| \leq |\partial_B(K \setminus A c_{n+1}^{-1})| 
\leq |K \cap \partial_B(A c_{n+1}^{-1})| + |\partial_B(K)| 
\leq \zeta|K| + |\partial_B(K)|
\end{align*}
and using $0<\epsilon<1/2$, one obtains
\begin{align*}
\frac{|\partial_B(K_{n+1})|}{|K_{n+1}|} \leq \frac{\zeta |K|}{(1-\varepsilon)|K|} + \frac{|\partial_B(K)|}{(1-\varepsilon)|K|}\leq
 2\zeta + 2\zeta^2 \leq 4\zeta.
\end{align*}
Thus (iii) holds as well and the claim is proven.
\end{proof}

\subsection*{Acknowledgment}
We would like to take this chance to thank our advisers {\sc Daniel Lenz} and {\sc Ivan Veseli\'c} for their support and their guidance during this work. We are very grateful that they generously shared their knowledge and their ideas about the topic with us. In particular, {\sc FP} expresses his thanks to {\sc Daniel Lenz} for drawing his attention to the {\sc Ornstein/Weiss} theory of amenable groups.       
Special thanks goes to {\sc Ivan Veseli\'c} as he developed the idea to apply a Banach space-valued ergodic theorem to obtain results on densities of percolation clusters. 
Moreover, {\sc FP} would like to give thanks to {\sc Benjamin Weiss} for fruitful discussions at the Hebrew University, as well as for his very helpful remarks on the theory of $\varepsilon$-quasi tilings. 
{\sc FP} gratefully points out that his work was partially supported by the German Research Council (DFG) and the German National Academic Foundation (Studienstiftung des deutschen Volkes). \\
The results of the present paper will also be part of the doctoral dissertations of FP and FS.